%% file: main.tex
\documentclass[review,onefignum,onetabnum]{siamonline220329}
\usepackage{bbm}

\input{ex_shared}

\ifpdf
\hypersetup{
  pdftitle={Hybrid Asymptotics},
  pdfauthor={M. Oprea, R. Huq, K. Iwasaki, D. Kassabova, A. Shaw, and W. Clark}
}
\fi

\begin{document}
\nolinenumbers
\maketitle

% REQUIRED
\begin{abstract}
Hybrid dynamical systems are systems which undergo both continuous and discrete transitions. As typical in dynamical analysis, an essential goal is to study the long-term behavior of these systems. 
In this work, we present two different novel approaches for studying these systems. The first approach is based on constructing an analog of the Frobenius-Perron (transport) operator for hybrid systems. 
Rather than tracking the evolution of a single trajectory, this operator encodes the asymptotic nature of an ensemble of trajectories.
The second approach presented applies to an important subclass of hybrid systems, mechanical impact systems. We develop an analog of Lie-Poisson(-Suslov) reduction for left-invariant impact systems on Lie groups. In addition to the Hamiltonian (and constraints) being left-invariant, the impact surface must also be a right coset of a normal subgroup. This procedure allows a reduction from a $2n$-dimensional system to an $(n+1)$-dimensional one. 
We conclude the paper by presenting numerical results on 
a diverse array of applications.
\end{abstract}

% REQUIRED

% REQUIRED
%
%%%%%%%%%%%%%%%%%%%%%%%%%%%%%%%%%%%%%%%%%%%

\section{Introduction}

\input{Introduction/Introduction}

%%%%%%%%%%%%%%%%%%%%%%%%%%%%%%%%%%%%%%%%%%%%%%%%%%%%%%%%%%%%%%%%%%%%%%%%
\section{Preliminaries}\label{sec:prelim}
    \input{Preliminaries/preliminariesintro}

%%%%%%%%%%%%%%%%%%%%%%%%%%%%%%%%%%%%%%%%%%

\subsection{Hybrid Dynamical Systems}\label{subsec:HDS}
    \input{Preliminaries/Hybrid_systems/General_definitions}

    \subsubsection{Mechanical Impact Systems}
        \input{Preliminaries/Hybrid_systems/Mechanical_impact_systems}
    \subsubsection{Constraints}
        \input{Preliminaries/Hybrid_systems/Constraints}

%%%%%%%%%%%%%%%%%%%%%%%%%%%%%%%%%%%%%%%%%%

\subsection{The Frobenius-Perron and Koopman Operators}\label{subsec:FP}
        \input{Preliminaries/Frobenius_Perron/Definitions_FP}

    \subsubsection{Infinitesimal Generators}
        \input{Preliminaries/Frobenius_Perron/infinitesimal_generators}

%%%%%%%%%%%%%%%%%%%%%%%%%%%%%%%%%%%%%%%%%%

\subsection{Reduction}\label{subsec:reduction}
    \input{Preliminaries/Reduction/reduction_general}

    \subsubsection{Lie-Poisson Reduction}
    \input{Preliminaries/Reduction/reduction_Lie_Poisson}

%%%%%%%%%%%%%%%%%%%%%%%%%%%%%%%%%%%%%%%%%%%%%%%%%%%%%%%%%%%%%%%%%%%%%%%%
\section{Results}\label{sec:results}
   \input{Results/introduction_results}
%%%%%%%%%%%%%%%%%%%%%%%%%%%%%%%%%%%%%%%%%%

\subsection{Hybrid Frobenius-Perron and Koopman operators}\label{subsec:hFP}
    \input{Results/Hybrid_FP/definition_of_hybrid_FP}
    \input{Results/Hybrid_FP/intermediary_lemmas}

\input{Results/Hybrid_FP/reduction_proofs}

%%%%%%%%%%%%%%%%%%%%%%%%%%%%%%%%%%%%%%%%%%

\subsection{Hybrid Reduction}\label{subsec:hLPO}
    \input{Results/Reduction/intro}
    \subsubsection{Hybrid Lie-Poisson reduction}
        \input{Results/Reduction/impact_reduction_theorems}
        \input{Results/Reduction/impact_reduction_proofs}

%%%%%%%%%%%%%%%%%%%%%%%%%%%%%%%%%%%%%%%%%%%%%%%%%%%%%%%%%%%%%%%%%%%%%%%%

\section{Applications}\label{sec:applications}
\input{Applications/Numerical_experiments_intro}
%%%%%%%%%%%%%%%%%%%%%%%%%%%%%%%%%%%%%%%%%%

\subsection{The bouncing ball}
\input{Applications/The_bouncing_ball}\label{subsec:BB}
%%%%%%%%%%%%%%%%%%%%%%%%%%%%%%%%%%%%%%%%%%%
    \subsection{The Chaplygin sleigh with angle impacts}\label{subsec:chapsleigh_angle}
\input{Applications/Chaplygin_Sleigh_2}
%%%%%%%%%%%%%%%%%%%%%%%%%%%%%%%%%%%%%%%%%%

    \subsection{Matrix groups}\label{subsec:matrix_groups}
\input{Applications/Matrix_Groups}

%%%%%%%%%%%%%%%%%%%%%%%%%%%%%%%%%%%%%%%%%%

%%%%%%%%%%%%%%%%%%%%%%%%%%%%%%%%%%%%%%%%%%

    \subsection{Modeling the spread of disease with human intervention}\label{subsec:SIR}
\input{Applications/Disease_spread}
%%%%%%%%%%%%%%%%%%%%%%%%%%%%%%%%%%%%%%%%%%%%%%%%%%%%%%%%%%%%%%%%%%%%%%%%
\section{Conclusions}\label{sec:conclusions}
\input{Conclusions/Conclusions}

\noindent\appendix
\label{appendix}

\input{Appendices/Proofs}

 \input{Appendices/Hybrid_Jacobian_for_Conservative_Holonomic_Systems}

\section*{Acknowledgments}
We would like to thank Mark Walth and Dr. Richard Rand for insightful conversations and inspiration for this work. We would additionally like to thank Mallory Gaspard for assistance on the numerical aspects of this project. Finally, we would like to thank Cornell University for hosting this REU and making this research possible.

\bibliographystyle{siamplain}
\bibliography{references}

\end{document}

%% file: ex_shared.tex
\usepackage{lipsum}
\usepackage{amsfonts}
\usepackage{amsmath}
\usepackage{graphicx}
\usepackage{epstopdf}
\usepackage{algorithmic}
\usepackage{ulem}
\usepackage{caption}
\usepackage{subcaption}
\usepackage{microtype}
\usepackage{mathrsfs}
\usepackage[scr=boondox]{mathalfa}
\DeclareMathOperator{\ad}{ad}
\DeclareMathOperator{\dom}{dom}
\DeclareMathOperator{\tr}{tr}
\DeclareMathOperator{\Ann}{Ann}
\DeclareMathOperator{\adj}{adj}
\usepackage{asymptote}
\usepackage{mathtools}
\usepackage{tikz-cd}
\usetikzlibrary{shapes.geometric, 
                patterns.meta,
                arrows.meta,
                intersections,
                shapes.arrows,
                decorations.markings,
                fadings, 
                shadows.blur,
                hobby,
                3d,
                backgrounds}
      
\usepackage{amssymb}
\usepackage{amsmath}
\usepackage[T2A,T1]{fontenc}
\DeclareSymbolFont{cyrillic}{T2A}{cmr}{m}{n}%{it}
\SetSymbolFont{cyrillic}{bold}{T2A}{cmr}{bx}{n}%{it}
\DeclareMathSymbol{\mTse}{\mathord}{cyrillic}{214}% Ц
\usepackage{pgfplots}

\usepackage{xr}

\pgfplotsset{compat = newest}
\usepgfplotslibrary{colormaps}

\newcommand{\R}{\mathbb{R}}
\newcommand{\Z}{\mathbb{Z}}
\newcommand{\sm}{\setminus}
\newcommand{\smm}{\!\sm\!}
\newcommand{\set}[1]{\left\{#1\right\}}
\renewcommand{\t}[1]{\text{#1}}
\newcommand{\mf}[1]{\mathfrak{#1}}
\newcommand{\Id}{\t{Id}}
\newcommand{\p}{\partial}
\newcommand{\ds}{\displaystyle}
\DeclareMathOperator{\codim}{codim}
\newcommand{\nsea}{\text{\scalebox{0.9}{\rotatebox[origin=c]{15}{$\searrow$}}}}

\usepackage{float}
\definecolor{purple1}{RGB}{130,45,92}
\definecolor{gold2}{RGB}{149,149,0}
\definecolor{gray1}{RGB}{147, 147, 147}
\definecolor{gray2}{RGB}{166, 166, 166}
\definecolor{gray3}{RGB}{185, 185, 185}
\definecolor{gray4}{RGB}{204, 204, 204}
\definecolor{MariaBlue}{RGB}{74, 144, 226}
\definecolor{MariaGreen}{RGB}{80, 227, 194}
\colorlet{MariaGreen2}{MariaGreen!75!black}

\newsavebox{\foobox}

\definecolor{rose}{RGB}{128,0,0}
\definecolor{roseyellow}{RGB}{222,205,99}
\definecolor{roseblue}{RGB}{167,188,214}
\definecolor{rosenavy}{RGB}{79,117,139}
\definecolor{roseorange}{RGB}{232,119,34}
\definecolor{rosegreen}{RGB}{61,68,30}
\definecolor{rosewhite}{RGB}{223,209,167}
\definecolor{rosegreen}{RGB}{61,68,30}

\ifpdf
  \DeclareGraphicsExtensions{.eps,.pdf,.png,.jpg}
\else
  \DeclareGraphicsExtensions{.eps}
\fi

\usepackage{enumitem}
\setlist[enumerate]{leftmargin=.5in}
\setlist[itemize]{leftmargin=.5in}

\newsiamremark{remark}{Remark}
\crefname{hypothesis}{Hypothesis}{Hypotheses}
\newsiamthm{claim}{Claim}
\headers{Hybrid Asymptotics}{M. Oprea, A. Shaw, R. Huq, K. Iwasaki, D. Kassabova, and W. Clark}

\title{A Study of the Long-Term Behavior of Hybrid Systems with Symmetries via Reduction and the Frobenius-Perron Operator\thanks{
\funding{This work was funded by the NSF grant DMS-1645643 and AFOSR Award No. MURI FA9550-32-1-0400.}}}

\author{Maria Oprea\thanks{Center of Applied Mathematics, Cornell University, Ithaca, NY 
  (\email{mao237@cornell.edu}).}
\and Aden Shaw\thanks{Department of Mathematics, Rose-Hulman Institute of Technology, Terre Haute, IN 
	(\email{shawap@rose-hulman.edu}).}
\and Robi Huq\thanks{Department of Mathematics, University of Minnesota Twin Cities, Minneapolis, MN 
  (\email{huq00011@umn.edu}).}
\and Kaito Iwasaki\thanks{Department of Mathematics, University of Michigan, Ann Arbor, MI
	(\email{kaitoi@umich.edu}).}
\and Dora Kassabova\thanks{Department of Mathematics, University of Washington, Seattle, WA
	(\email{dmk285@uw.edu}).}
 \and \newline William Clark\thanks{Department of Mathematics, Ohio University, Athens, OH (\email{clarkw3@ohio.edu})}
}

\usepackage{amsopn}

%% file: Introduction/Introduction.tex
A characteristic pursuit in dynamical systems is to study and understand their long-term behavior. One approach to understand a system's asymptotic properties is through a probabilistic lens. If the dynamics preserve a probability measure, then the celebrated Poincar\'{e} recurrence and ergodic theorems can be applied. These results are well-studied in the cases where the evolution is either discrete (an iterated map) or continuous (a flow induced by a differential equation) \cite{modern_theory_of_dyn_sys}; however, many real world phenomena are described by systems that fail to be exclusively discrete nor continuous \cite{HDS_magazine}. Consider the pedagogical example of a bouncing ball: as it flies through the air, its motion is continuous (ballistic motion is described by a differential equation), but when it hits the ground there is an instantaneous change in momentum (by applying the so-called impact map), see Figure \ref{fig:BouncingBalls}. We call such systems \textit{hybrid systems}---dynamical systems whose evolution are subject to \textit{both} continuous and discrete laws \cite{hybrid}.

Although hybrid systems are ubiquitous in real life applications, their rigorous study is made difficult due to combination of continuous and discrete dynamics, and is consequently less developed in the literature. As such, the overarching goal of this work is to extend preexisting theory to hybrid systems to (a) study such systems through a probabilistic lens and (b) leverage the geometry of such systems in order to reduce the difficulty/dimension of the problem. 

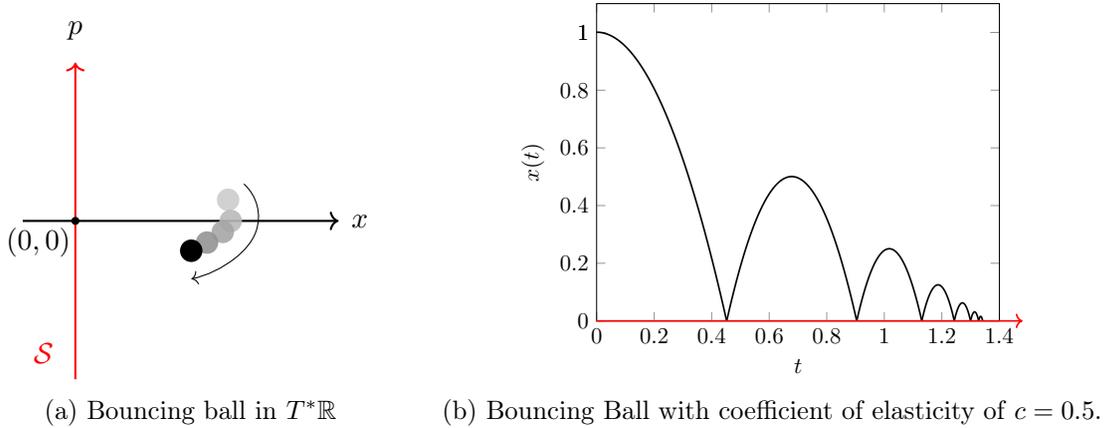
\begin{figure}[!ht]
    \centering
    \begin{subfigure}{.4\textwidth}
        \centering
        \input{TikZ_Code/BouncingBall}
        \caption{Bouncing ball in $T^*\mathbb{R}$}
        \label{subfig:BouncingBall}
    \end{subfigure}%
    \begin{subfigure}{.59\textwidth}
        \centering
        \input{TikZ_Code/1D_BallRevised}
        \caption{Bouncing Ball with coefficient of elasticity of $c = 0.5$.}
        \label{subfig:1Dball}
    \end{subfigure}
    \caption{Plots of the motion of the bouncing ball. Left: Phase plot on $T^*\mathbb{R}$ with the momentum and position of the ball. Right: Trajectory of the ball, which tracks the height $x(t)$ of the ball. The guard is given by $\mathcal{S} = \{(x,p)\in T^*\mathbb{R} \cong \mathbb{R}^2 : x=0, ~ p < 0\}$.}
   
    \label{fig:BouncingBalls}
\end{figure} \vspace*{-5mm}

Consider the continuous-time dynamical system $\dot{x} = X(x)$, where $X$ is a vector-field on the ambient manifold $M$. Let $\mathcal{S}$  be a distinguished codimension one (embedded) submanifold   \\
\noindent\begin{minipage}[b]{.42\linewidth}
    called the \textit{guard}; this encodes location of the discrete transitions. Finally, we introduce the \textit{reset} map, $\Delta \colon \mathcal{S} \to M$, which dictates the discrete transition.
    Such a system will be called \textit{hybrid} and its dynamics will be governed by the following rule
    \begin{equation}\label{eq:HDS}
        \begin{cases}
            \dot{x} = X(x), & x\not\in \mathcal{S}, \\
            x^+ = \Delta(x^-), & x\in\mathcal{S}.
        \end{cases}
    \end{equation}
    This system exhibits continuous behavior away from the guard and discrete when the state enters the guard. If the guard is in its image, i.e., $\Delta(\mathcal{S}) \cap \mathcal{S} \ne \varnothing$, then multiple resets can immediately occur---a phenomenon called \textit{beating}. To 
\end{minipage}\hfill
\begin{minipage}[b]{.556\linewidth}
    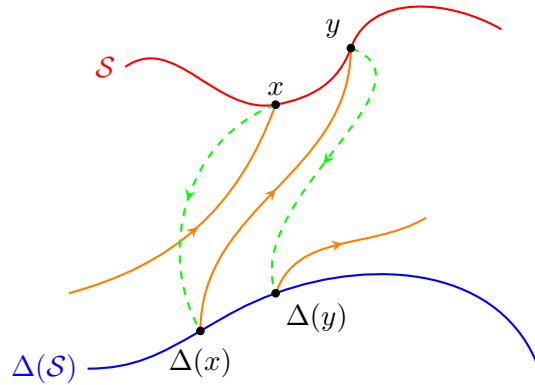
\begin{figure}[H]
        \centering
        \input{TikZ_Code/Tikz_hybrid_system.tex}
        \caption{Schematic drawing of a hybrid system.}
        \label{fig:HDS}
    \end{figure} 
    \vspace{0.2cm}
\end{minipage}
avoid this, we will assume that the guard is disjoint from its image under the reset map. More technical details for the hybrid system governed by \eqref{eq:HDS} will be discussed in Section \ref{sec:HDS}.

The Frobenius-Perron and Koopman operators are adjoint linear transfer operators that describe the evolution of scalar observables. Linear operator theory provides a framework for studying non-linear dynamical systems by obtaining infinite-dimensional linear representations of the system, which can be used for the spectral analysis of nonlinear dynamical systems \cite{mezic_koopman,mauroy2020koopman}. This allows for an elegant study of the asymptotic behaviour by looking at successive iterates of the Koopman and Frobenius-Perron operators, and their proprieties in the limit of infinitely many iterations. 

Despite the fact that Koopman's original ideas were published nearly a century ago, work on the subject was not popularized until fairly recently, due to a lack of efficient algorithms and computing power capable of calculating the Koopman operator. This theory is well-defined for both discrete- and continuous-time systems as right composition. However, current methods have shown that even when taking ino advantage the linearity of the Koopman operator, there are limitations when handling higher-dimensional systems \cite{koopman_ieee}. This leads us into our second approach---geometric dimension reduction.

The theory of reduction plays an important role in classical mechanics.
Lie-Poisson reduction concerns the case when the configuration space is a Lie group, $G$, and reduces the dynamics from the symplectic manifold $T^*G$ to the Poisson manifold $\mathfrak{g}^*$ effictively decreasing the dimension of the system from $2n$ to $n$ \cite[Chapter~13]{marsden1998introduction}. Discrete analogs have been proposed  \cite{discrete_reduction1}, where the Lie group is finite dimensional, and the Lagrangian is discrete. 
Hybrid systems pose a challenge due to the fact that the equations of motion come from a continuous Hamiltonian, whereas the discrete impacts are provided by an arbitrary map which might not come from a discrete Lagrangian as in \cite{discrete_reduction1}. Therefore, neither the continuous nor the discrete reduction can be directly applied. 

It is possible to leverage the symmetries of the continuous part of a hybrid system, but in order to do this, it is necessary to document the locations and effects of the discrete impacts. This is why, for hybrid systems, reduction from $2n $ to $n$ is not achievable. In this paper we propose a method (see \ref{thm:Impact_Reduction}) that allows reduction to $n + 1$ dimensions,  where the `extra' dimension is used to record the locations of the discrete impacts on the guard.

The goal of this work is to extend and combine these theories such that they can be used on more complex hybrid systems. For instance, the state space of a robotic device is often high dimensional, especially when one has to take into account a large number of components that work together to fulfill the functionality of the device. By reducing the dimension and studying the long-term behaviour of this system one could understand robustness and stability.

This manuscript is organized as follows: Section \ref{sec:prelim} contains the necessary preliminaries of hybrid systems (\ref{subsec:HDS}), the classical Frobenius-Perron and Koopman operators (\ref{subsec:FP}), and reduction (\ref{subsec:reduction}). Section \ref{sec:results} contains our results, specifically the hybrid Frobenius-Perron operator (\ref{subsec:hFP}) and hybrid Lie-Poisson reduction (\ref{subsec:hLPO}). Section \ref{sec:applications} contains numerical results for four examples: the bouncing ball (\ref{subsec:BB}), the Chaplygin sleigh (\ref{subsec:chapsleigh_angle}), matrix groups (\ref{subsec:matrix_groups}), and an SIR type disease model (\ref{sec:SIR}). Conclusions and future work are presented in Section \ref{sec:conclusions}. Detailed proofs for technical lemmas can be found in the Appendix. 
\newpage

%% file: TikZ_Code/BouncingBall.tex
% \resizebox{!}{5cm}{
\begin{tikzpicture}[x = 1.4cm, y = 1.4cm]
    \draw[thick,->] (0,1.5) -- (3,1.5);
    \draw[thick,->,red] (.5,0) -- (.5,3);
    \draw (3.2, 1.5) node [anchor=center][inner sep=0.75pt]{$x$};
    \draw (.5, 3.3) node [anchor=center][inner sep=0.75pt] {$p$};
    \draw (.2, .25) node [anchor=center,red][inner sep=0.75pt] {$\mathcal{S}$};
    % \draw (.3, 1.3) node [anchor=center][inner sep=0.75pt] {$\mathcal{O}$};
    \draw [fill] (.5,1.5) circle [radius=0.033];
    \draw (.15, 1.3) node [anchor=center][inner sep=0.75pt] {$(0,0)$};
    
    \node at (1.95, 1.7) [gray4,opacity=.9, circle, fill, inner sep=3pt]{}; 
    \node at (1.975, 1.5) [gray3,opacity=.9, circle, fill, inner sep=3pt]{}; 
    \node at (1.9, 1.396) [gray2,opacity=.9, circle, fill, inner sep=3pt]{}; 
    \node at (1.75, 1.295) [gray1,opacity=.9, circle, fill, inner sep=3pt]{}; 
    \node at (1.6, 1.2175) [opacity=1, circle, fill, inner sep=3pt]{}; 
    \draw[thin, ->] (2.1, 1.85) to[curve through={(2.215, 1.4)..(2.02, 1.15)}] (1.6, .95);
\end{tikzpicture}
% }

%% file: TikZ_Code/1D_BallRevised.tex
\resizebox{!}{5cm}{\begin{tikzpicture}
    \begin{axis}[
		    xlabel={$t$},
		    ylabel={$x(t)$},
		    xmin=0, xmax=1.4,
		    ymin=0, ymax=1.1,
		    xtick={0,.2,.4,.6,.8,1,1.2,1.4},
		    ytick={0,.2,.4,.6,.8,1,1},
		   axis equal image
		]
	\end{axis}
    \pgfmathsetmacro\rshift{1.9}
    
    \draw[thick, variable=\t,domain=0:.4517,samples=100]
        plot ({4.9*\t},{4.9*(1-(9.81/2)*(\t^2))})
        node[right] {} ;
    \draw[thick, variable=\t,domain=-.4517:.4517,samples=100]
        plot ({.5*4.9*(\t+3*.4517)},{.5*4.9*(1-(9.81/2)*((\t)^2))})
        node[right] {} ;
    \draw[thick, variable=\t,domain=-.4517:.4517,samples=100]
        plot ({.25*4.9*(\t+9*.4517)},{.25*4.9*(1-(9.81/2)*((\t)^2))})
        node[right] {} ;
    \draw[thick, variable=\t,domain=-.4517:.4517,samples=100]
        plot ({.125*4.9*(\t+21*.4517)},{.125*4.9*(1-(9.81/2)*((\t)^2))})
        node[right] {} ;
    \draw[thick, variable=\t,domain=-.4517:.4517,samples=100]
        plot ({.0625*4.9*(\t+45*.4517)},{.0625*4.9*(1-(9.81/2)*((\t)^2))})
        node[right] {} ;
    \draw[thick, variable=\t,domain=-.4517:.4517,samples=100]
        plot ({0.03125*4.9*(\t+93*.4517)},{0.03125*4.9*(1-(9.81/2)*((\t)^2))})
        node[right] {} ;
    \draw[thick, variable=\t,domain=-.4517:.4517,samples=100]
        plot ({0.015625*4.9*(\t+189*.4517)},{0.015625*4.9*(1-(9.81/2)*((\t)^2))})
			node[right] {} ;

    \draw[thick, red, ->](-.01,0) -- (7.25,0);
    %\draw (7.5,0) node [anchor=center, inner sep=0.75pt, red] {\resizebox{.35cm}{!}{$\Sigma$}};
\end{tikzpicture}}

%% file: TikZ_Code/Tikz_hybrid_system.tex
\begin{tikzpicture}%[x=0.75pt,y=0.75pt,yscale=-1,xscale=1]
    [decoration={markings, 
	    mark= at position 0.5 with {\arrow{stealth}}}] 
    % The guard
    \draw[red,thick] (-3,3) to [out=35,in=190] (-1,2.5) to [out=10,in=250] (0,3.25) to [out=70,in=150] (2,3.5);
    \node[red,left] at (-3,3) {$\mathcal{S}$};
    % The image
    \draw[blue,thick] (-3.5,-1) to [out=0,in=210] (-2,-0.5) to [out=30,in=200] (-1,0) to [out=20,in=110] (2.5,-1);
    \node[blue,left] at (-3.5,-1) {$\Delta(\mathcal{S})$};
    % The lines
    \draw [postaction={decorate},orange,thick] (-3.75,0) to [out=15,in=-110] (-1,2.5);
    \draw [postaction={decorate},orange,thick] (-2,-0.5) to [out=90,in=-90] (0,3.25);
    \draw [postaction={decorate},orange,thick] (-1,0) to [out=70,in=-150] (1,1);
    % The resets
    \draw [postaction={decorate},green,thick,dashed] (-1,2.5) to [out=200, in=120] (-2,-0.5);
    \draw [postaction={decorate},green,thick,dashed] (0,3.25) to [out=-20, in=110] (-1,0);
    % The nodes
    \draw [fill] (-1,2.5) circle [radius=0.05];
    \node [above] at (-1,2.5) {$x$};
    \draw [fill] (0,3.25) circle [radius=0.05];
    \node [above left] at (0,3.25) {$y$};
    \draw [fill] (-2,-0.5) circle [radius=0.05];
    \node [below] at (-2,-0.6) {$\Delta(x)$};
    \draw [fill] (-1,0) circle [radius=0.05];
    \node [below right] at (-1,0) {$\Delta(y)$};
\end{tikzpicture}

%% file: Preliminaries/preliminariesintro.tex
This section is devoted to defining common terminology related to our results and is broken up into three subsections: hybrid systems, Frobenius-Perron and Koopman operators, and reduction. An understanding of hybrid systems is necessary for understanding both of the main results, so hybrid systems are covered first, and each of the remaining subsections correspond to a main result.

%% file: Preliminaries/Hybrid_systems/General_definitions.tex
\label{sec:HDS}
The purpose of this section is to define the types of systems we are working with. First, we will focus on hybrid dynamical systems, which are continuous dynamical systems which exhibit occasional discrete transitions. The state space will always be a smooth manifold $M$. The continuous dynamics will be given in terms of a vector field, and the jumps happen when the trajectory of the system hits the guard denoted by $\mathcal{S}$ \cite{simple_hybrid}. Following hybrid systems, we will refine our discussion to impact systems which are a subset of hybrid systems.
\begin{definition}[Hybrid dynamical system]
    \label{def:Hybrid_Dynamical_System}
    A hybrid dynamical system $\mathcal{H}$, abbreviated by HDS, is 4-tuple $\mathcal{H}=(M,\mathcal{S},\Delta,X)$ with the following properties
    \begin{enumerate}
        \item[\normalfont{(H.1)}] $M$ is a finite-dimensional smooth manifold,
        \item[\normalfont{(H.2)}] $\mathcal{S} \subset M$ is an embedded smooth manifold where $\codim\mathcal{S} = 1$,
        \item[\normalfont{(H.3)}] $\Delta \colon \mathcal{S} \to M$ is a smooth map whose image is an embedded submanifold,
        \item[\normalfont{(H.4)}] $X \colon M \to TM$ is a smooth vector field, and
        \item[\normalfont{(H.5)}] $\mathcal{S} \cap \Delta(\mathcal{S}) = \varnothing$ and $\codim\!\big(\overline{\mathcal{S}} \cap \overline{\Delta(\mathcal{S})} \big) \geq 2$.
    \end{enumerate}
    As stated in the introduction, the submanifold $\mathcal{S}$ is called the guard and the map $\Delta$ is the reset.
\end{definition}

Away from the guard the trajectories of a hybrid dynamical system are integral curves of the vector field $X$, whereas, on the guard the system is equivalent to a discrete dynamical system with map $\Delta$. Putting the continuous and the discrete parts of the trajectory together generates the hybrid flow. 

\begin{definition}[Hybrid Flow]
    \label{def:Hybrid_Flow}
    Let $\mathcal{H} = (M, \mathcal{S}, \Delta, X)$ be an HDS. The map $\varphi^\mathcal{H} \colon \mathrm{dom}(\varphi^\mathcal{H})\subset \mathbb{R}\times M\to M$ is the hybrid flow if 
    \begin{equation*}
        \begin{array}{cl}
            \displaystyle
            \frac{d}{dt} \varphi^\mathcal{H}(t,x) = X_{\varphi^\mathcal{H}(t,x)},%\left( \right), 
            & \varphi^\mathcal{H}(t,x)\not\in\mathcal{S}; \\[2ex]
            \displaystyle
            \lim_{s\to t^+} \varphi^\mathcal{H}(s,x) = \Delta\left( \lim_{s\to t^-} \varphi^\mathcal{H}(s,x)\right), & \varphi^\mathcal{H}(t,x) \in \mathcal{S}.
        \end{array}
    \end{equation*}
\end{definition}

Hybrid systems can fail to be complete in two qualitatively different ways: finite time blow ups of the continuous trajectory, and the Zeno phenomenon. The latter is special to hybrid systems, and it happens when a trajectory undergoes infinitely many jumps in a finite amount of time. 

Zeno trajectories pose a difficult challenge. In order to compute the hybrid flow as time approaches Zeno one needs to be able to perform infinitely accurate event detection. From a theoretical perspective, what happens after the time becomes greater than $t_\infty$ is still unclear \cite{LifeAfterZeno}. 
It is generally difficult to determine whether or not a state will have a Zeno trajectory or not. However, for reasonable HDSs which possess an invariant volume-form, the set of points which are Zeno form a null set \cite{hybrid_forms}. As volume-forms are smooth objects, the HDS must satisfy some amount of smoothness conditions.

\begin{definition}[Quasi-smooth dependence property]\label{def:Quasi-smooth Dependence Property}
   
    Let $\mathcal{H} = (M, \mathcal{S}, \Delta, X)$ be a HDS with flow $\varphi^\mathcal{H}$. $\mathcal{H}$ has the quasi-smooth dependence property if for every $x \in M \smm \mathcal{S}$ and $t \in \mathbb{R}$ such that $\varphi^\mathcal{H}(t,x) \not\in \mathcal{S}$, there exists an open neighborhood $x\in U$ such that $U \cap \mathcal{S} = \varnothing$ and the map $\varphi^\mathcal{H}(t,\,\cdot\,) \colon U \to M$ is smooth. Furthermore, if $\mathcal{H}$ is a HDS with the quasi-smooth dependence property, then $\mathcal{H}$ is a smooth HDS.
\end{definition}

In this work we will assume that all the systems have the quasi-smooth dependence propriety (as well as completeness, unless otherwise stated); this is not a strong assumption cf. \cite{clarkthesis}. This assumption allows for a meaningful study of how volumes preserve across applications of the reset map, which is crucial when defining the Frobenius-Perron operator. 
For a volume form $\mu \in \Omega^n(M)$ to be be invariant under the hybrid flow within a smooth HDS, $\mu$ must be invariant under both the continuous-time dynamics and the reset. Invariance under the continuous dynamics follows when $\mathcal{L}_X\mu=0$ (equivalently, $\mathrm{div}_\mu(X)=0$), where $\mathcal{L}$ denotes the Lie derivative. Invariance across resets occurs when
\begin{equation*}
    \Delta^* i_X\mu = \iota^*_{\mathcal{S}}i_X\mu,
\end{equation*}
where $\iota_{\mathcal{S}}\colon\mathcal{S}\hookrightarrow M$ is the inclusion and $i_X$ is the interior product \cite{hybrid_forms}. This prompts the following definition.

\begin{definition}[Hybrid Jacobian of $\Delta$]
\label{def:hybrid_jacobian}
    Let $\mathcal{H} = (M,\mathcal{S},\Delta ,X)$ be a smooth HDS and $\mu\in \Omega^n(M)$ a volume-form. The unique function $\mathcal{J}_\mu^X(\Delta)\in C^\infty(\mathcal{S})$ such that
    \begin{equation}\label{eqn:hybrid Jacob definition}
        \Delta^*i_X\mu = \mathcal{J}_\mu^X(\Delta)\cdot \iota^*_\mathcal{S} i_X\mu,
    \end{equation}
    is called the hybrid Jacobian of $\Delta$.
\end{definition}

\begin{proposition}[Theorem 9 in \cite{hybrid_forms}]\label{prop:the hybrid jacobian is 1 most of the time}
Let $\mathcal{H} = (M, \mathcal{S}, \Delta ,X)$ be a smooth HDS and $\mu\in \Omega^n(M)$ a volume-form. Then $\mu$ is invariant under the flow of $\mathcal{H}$ if $\mathcal{L}_X\mu = 0$ and its hybrid Jacobian $\mathcal{J}_\mu^X (\Delta)$ is 1. % \mao{cite}
\end{proposition}

%%%%%%%%%%%%%%%%%%%%%%%%%%%%%%%%%%%%%%%%%%%%%
%%%%%%%%%%%%%%%%%%%%%%%%%%%%%%%%%%%%%%%%%%%%%
\input{Preliminaries/Hybrid_systems/Filipov_Box}
%%%%%%%%%%%%%%%%%%%%%%%%%%%%%%%%%%%%%%%%%%%%%
%%%%%%%%%%%%%%%%%%%%%%%%%%%%%%%%%%%%%%%%%%%%%

As will be seen in Section \ref{subsec:chikos}, the hybrid Jacobian will be central in the study of the hybrid Frobenius-Perron operator.

%% file: Preliminaries/Hybrid_systems/Filipov_Box.tex
\begin{example}[Filippov Systems]
    Consider the plane $M = \R^2$, with the guard given by 
    $$\mathcal{S} = \set{(x,y) \in \R^2 : xy = 0},$$ 
    and the identity reset map of $\Delta(x,y) = \Id_{\R^2}(x,y) = (x,y)$. Furthermore, consider the volume form $\mu = dx \wedge dy$ and the following discontinuous vector field.
    \begin{align*}
        X_{(x,y)} =
        \begin{cases}
            \left(\dfrac{y}{\alpha} - x\right) \dfrac{\p}{\p x} - (\alpha x + y) \dfrac{\p}{\p y}
            & \t{for }xy > 0; \\[7pt]
            (\alpha y - x) \dfrac{\p}{\p x} - \left(y + \dfrac{x}{\alpha}\right ) \dfrac{\p}{\p y}
            & \t{for }xy < 0.
        \end{cases}
    \end{align*}
    To compute the hybrid Jacobian, we notice that $\Delta$ is the identity so we must compute $\iota^*_\mathcal{S}i_X\mu$ on both sides of the guard
    \begin{align*}
        i_X dx\wedge dy = \begin{cases}
            \left( \dfrac{y}{\alpha}-x\right) dy + \left(\alpha x + y\right)dx & \t{for }xy>0; \\[5pt]
            \left(\alpha y - x\right)dy + \left(y + \dfrac{x}{a}\right) dx & \t{for } xy<0.
        \end{cases}
    \end{align*}
    Substituting either $x=0$ or $y=0$ shows that the hybrid Jacobian is
    \begin{align*}
        \mathcal{J}_\mu^X(\Id_{\R^2}) = \alpha^2.
    \end{align*}
    As a consequence of Theorem \ref{thm:general chikos} below, an invariant density will exist when $\alpha = \exp(\pi)$. In \\
    \begin{minipage}[b]{0.43\textwidth}
    this case, if $(x,y)$ lies within the first quadrant, then an invariant density is given by
    \begin{equation*}
        \rho(x,y) = e^{-2\tau}, \quad \tan\tau = \dfrac{\alpha x}{y}.
    \end{equation*}
    In the case that the points do not lie in the first quadrant, $\rho$ can be extended through the following map
    \begin{equation*}
        \tilde{\rho}(x,y) = 
        \begin{cases}
            \rho(x,y), & x > 0, \, y > 0; \\[.5ex]
            \rho(-y, x), & x > 0, \, y < 0; \\[.5ex]
            \rho(-x,-y), & x < 0, \, y < 0; \\[.5ex]
            \rho(y,-x), & x < 0,\, y > 0,
        \end{cases}
    \end{equation*}
    which is displayed in Figure \ref{fig:filippov}.
    % \begin{table}[H]
    %     \centering
    %     \begin{tabular}{c|c}
    %         Value of $\tilde{\rho}(x,y)$ & Quadrant \\ \hline
    %         $\rho(x,y)$ & I \\
    %         $\rho(y,-x)$ & II \\
    %         $\rho(-x,-y)$ & III \\
    %         $\rho(-y, x)$ & IV
    %     \end{tabular}
    %     \caption{This density is shown in Figure \ref{fig:filippov}.}
    %     \label{tab:Filippov Density}
    % \end{table}
    
    \end{minipage}\hfill
    \begin{minipage}[b]{0.5025\textwidth}
    \begin{figure}[H]
        \centering
        \includegraphics[width = .925\linewidth]{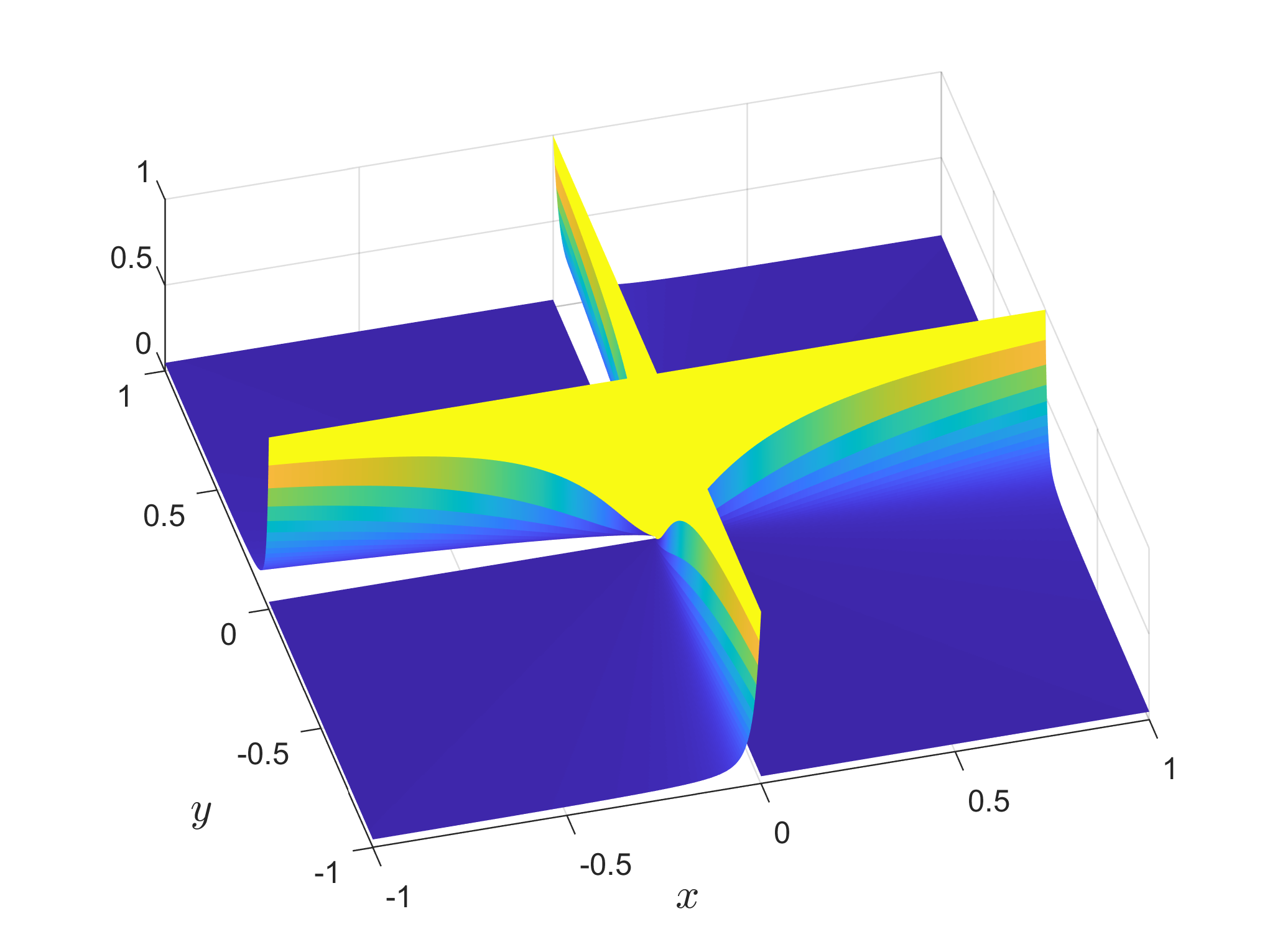}
        \caption{This is the invariant density for this example. The jump discontinuity between quadrants is precisely the hybrid Jacobian: $\exp(\pi)$.}
        \label{fig:filippov}
    \end{figure}
\end{minipage}

\end{example}

%% file: Preliminaries/Hybrid_systems/Mechanical_impact_systems.tex
A physically motivated class of hybrid systems are mechanical impact systems.
In these systems, the continuous dynamics obey Hamiltonian/Lagrangian mechanics while the reset map encodes the momentum/velocity jump at a physical impact. Throughout this work, we will be primarily in the Hamiltonian setting, but the results can be translated to the Lagrangian framework with little difficulty.

Let $Q$ be an $n$ dimensional smooth manifold representing the configuration space of the system of interest. To each point $q\in Q$, the possible momenta will be elements of the cotangent space $p\in T_q^*Q$. The union of all possible momenta attached to all possible points is called the cotangent bundle or the phase space, and is denoted by $T^*Q$. In a similar fashion, the possible velocities will be tangent vectors, $v \in T_qQ$, and the collection of all positions and velocities will be the tangent bundle $TQ$.

The cotangent bundle is equipped with the canonical symplectic form which we call $\omega \in \Omega^2(T^*Q)$, and in local coordinates is given by $\omega = dq^i \wedge dp_i$. For a given Hamiltonian function $H \colon T^*Q \to \mathbb{R}$, Hamilton's equations of motion are described by 
 \begin{equation}\label{eq:simple_eom}
   i_X \omega = dH,
 \end{equation}
 where $i_X$ denotes the interior product, $i_X\omega(Y) = \omega(X, Y)$ for any vector field $Y$, $dH$ is the differential of $H$, and $X$ is the resulting Hamiltonian vector field. There exist various levels of regularity for Hamilton functions; throughout this work, we will assume that all Hamiltonians are of mechanical type.

 \begin{definition}[Natural Hamiltonians]
    A Hamiltonian $H \colon T^*Q \to \R$ is called 
    natural/mechanical if there exists a Riemannian metric $\mTse$ on $Q$ and a function $V \colon Q \to \R$ such that
        \begin{equation*}
            H(q,p) = \frac{1}{2}(\mTse^{-1})_q(p,p) + V(q),
        \end{equation*}
        where $\mTse^{-1}$ is the induced co-metric.
\end{definition}

For a natural Hamiltonian we can explicitly write the write the Hamiltonian vector field from \eqref{eq:simple_eom} and obtain the following coordinate representation of Hamilton's equations of motion:
\begin{equation}\label{eqn:Induced Hamilton Vector Field}
    \begin{split}
        \dot{q}^k 
        & = \frac{\partial H}{\partial p_k} 
        = (\mTse^{-1})^{kj}p_j; \\
        \dot{p}_k 
        & = -\frac{\partial H}{\partial q^k} 
        = -\frac{1}{2}\frac{\partial (\mTse^{-1})^{ij}}{\partial q^k}p_ip_j - \frac{\partial V}{\partial q^k}.
    \end{split}
\end{equation}
The metric underlying a mechanical Hamiltonian has the additional feature of being able to isomorphically transform vectors to covectors. 
\begin{definition}[Musical Isomorphisms% {\cite[Chapter~13]{Lee2003}}
]
\label{def:musical_isomorphisms} 
    Given a Riemannian manifold $(Q, \mTse)$ the musical isomorphisms are the maps
    \begin{align*}
        \flat \colon &  \mathfrak{X}(Q) \to \Omega^1(Q) : X \mapsto \mTse(X, \cdot\,),\\
        \sharp \colon &  \Omega^1(Q) \to \mathfrak{X}(Q) : \alpha \mapsto \mTse^{-1}(\alpha, \cdot\,),
    \end{align*}
    which are defined so that $\sharp = \flat^{-1}$. In coordinates the musical isomorphisms can be written as
    \begin{align*}
        (X^\flat)_i = \mTse_{ij}X^j, \qquad \t{and}\qquad (\alpha^\sharp)^i = (\mTse^{-1})^{ij}\alpha_j.
    \end{align*}
\end{definition}

\begin{remark}\label{remark:Fiber Derivative}
If the Hamiltonian is natural then one can isomorphically switch between the Hamiltonian and Lagrangian formalism using the fiber derivative $\mathbb{F}L \colon TQ \to T^*Q$, where $L:TQ\to\mathbb{R}$ is the Lagrangian, and the inverse fiber derivative $\mathbb{F}H \colon T^*Q \to TQ$. In the mechanical case, the fiber derivatives agree with the musical isomorphisms.

\end{remark}

Impacts induce instantaneous changes in the momenta in $T^*Q$ and are triggered by the positions in $Q$. We assume that the impact surface $\Sigma$ is a smooth codimension 1 submanifold of the configuration space $Q$, which can locally be written as a regular level set of a smooth function $s \colon Q \to \R$. Away from $\Sigma$, the dynamics are integral curves of the Hamiltonian vector field generated by $H$. We put these pieces together in the following definition:

\begin{definition}[Impact system]
\label{def:Impact System}
An impact system is a tuple $\mathcal{I} = (Q, H, \Sigma)$ such that
    \begin{enumerate}
        \item[\normalfont{(I.1)}] $Q$ is a finite-dimensional smooth manifold,
        \item[\normalfont{(I.2)}] $H \colon T^*Q \to \R$ is a natural Hamiltonian,
        \item[\normalfont{(I.3)}] $\Sigma \subset Q$ is an orientable embedded smooth manifold where $\codim\Sigma = 1$.
        % \item[\normalfont{(I.3)}] $s \in C^\infty(Q)$ with zero as a regular value which defines the impact surface via $\Sigma = s^{-1}(\set{0})$.
    \end{enumerate}
\end{definition}

The three pieces of information in Definition \ref{def:Impact System} are enough to completely specify a hybrid dynamical system $(M, X, \mathcal{S}, \widetilde{\Delta})$ through the following process.
\begin{itemize}
    \item[(h.1)] $M = T^*Q$ is the phase space containing positions and momenta $(q, p)$.
    \item[(h.4)] $X$ is given by $X_H$, the Hamiltonian vector field generated by $H$ through equation \eqref{eq:simple_eom}.
    \item[(h.2)] The guard is the set of all outward pointing momenta attached to positions in $\Sigma$, so
    \begin{align*}
        \mathcal{S} = \left\{(q, p) \in T^*Q|_\Sigma : \pi^*_Q ds(X_H) > 0\right\},
    \end{align*}
    where $\pi_Q \colon T^*Q \to Q$ is the canonical projection. 
  
    \item[(h.3)] The impact map is identity in the first $n$ components $\widetilde{\Delta}(q, p) = (q, \Delta(p))$. In the last $n$ components, $\Delta(p)$ can be computed from the Weierstrass-Erdmann corner conditions \cite[Section 15]{CoV_1991}. Let $p^+$ and $p^-$ denote be the values of the momenta before and after the impact, respectively, with $H^+$ and $H^-$ their corresponding values of the Hamiltonian. Then the corner conditions are given by
    \begin{gather}\label{Weierstrass–Erdmann corner conditions}
        H^- = H^+ %\varepsilon \, \frac{\partial s}{ \partial t}
        \qquad \text{and} \qquad
        p^+ - p^- =  \varepsilon \, ds.
        % \label{Weierstrass–Erdmann corner conditions 2}
        %\qquad \text{ or }  \qquad p^+_i - p^-_i = \lambda_\alpha {\eta^\alpha}_i + \varepsilon \, \frac{\partial}{\partial x^i} s
    \end{gather}
\end{itemize}
\begin{remark}
    Throughout this paper, the impact surface $\Sigma \subset Q$ will be independent of time% as implied by $s \in C^\infty(Q)$
    , but in general, the local defining function $s$ is an element of $C^\infty(\R \times Q)$, which locally defines $\Sigma_t \subset \R \times Q$. Then, the left side of \eqref{Weierstrass–Erdmann corner conditions} becomes $H^- - H^+ = \varepsilon \p s/\p t$, which can be physically intuited as the impact surface adding to/taking away from the energy of the system.
\end{remark}

We can rephrase Equation \eqref{Weierstrass–Erdmann corner conditions} in a coordinate-free manner with 
\begin{equation}\label{eqn:CordFreeWECornerConditions}
    (\Id \times \Delta)^*\vartheta_H = \iota^* \vartheta_H \qquad \text{and}\qquad  \vartheta_H = p_i dx^i - H dt
\end{equation}
where $\iota \colon \R \times Q \hookrightarrow \R \times T^*Q$ is the inclusion map and $\vartheta_H$ is called the action form. Note that \eqref{eqn:CordFreeWECornerConditions} works for both time-dependent and time-independent impact surfaces, but will be used only for the latter.

In the case where the Hamiltonian in \eqref{Weierstrass–Erdmann corner conditions} is natural, there exists a unique solution \cite{hybrid_forms,cog_optimal_control} given by
\begin{equation}
    \tilde\Delta(q, p) = (q, R(q)p),
    \qquad \text{where} \qquad  
    R(q)p = p - 2\frac{\mTse^{-1}(ds, p)}{\mTse^{-1}(ds, ds)}ds.
\end{equation}

%% file: Preliminaries/Hybrid_systems/Constraints.tex
Constraints are ubiquitous in mechanical systems and dictate either allowable positions or velocities. There are two distinguished classes of constraints: integral or holonomic and non-integrable/nonholonomic.
While holonomic systems maintain their Hamiltonian structure, nonholonomic systems are \textit{not} Hamiltonian as they do not come from variational principles. Fortunately, holonomic systems can be viewed as a special case of nonholonomic systems and the theory for the latter still describes the former. As a result, we will treat all constraints as nonholonomic and all constraints will be \textit{linear} in the velocities. For a more in-depth treatment of nonholonomic mechanics see \cite[Chapter~2]{bloch2004nonholonomic}. 

To obtain the nonholonomic equations of motion, one must use the Lagrange-d'Alembert principle \cite[Chapter~7]{marsden1998introduction}. In the Hamiltonian formalism these give raise to the constraint Hamiltonian vector field. 

\begin{definition}[Constrained Hamiltonian vector field]
    Let $(Q,\omega,H)$ be a natural Hamiltonian system and $\mathcal{D}\subset TQ$ be a regular distribution. Let $\eta^\alpha\in\Omega^1(Q)$ be a collection of $k$ 1-forms which locally describe $\mathcal{D}$ via annihilation. Then the nonholonomic Hamiltonian vector field $X_H^\mathcal{D}$, is the unique vector field such that
    \begin{equation*}
        i_{X_H^\mathcal{D}}\omega|_{\mathcal{D}^*} = dH|_{\mathcal{D}^*} + \lambda_\alpha \pi_Q^*\eta^\alpha|_{\mathcal{D}^*},
    \end{equation*}
    where $\lambda_\alpha$ are multipliers to enforce the constraints,
    \begin{equation*}
        \eta^\alpha\big(\mathbb{F}H(q,p)\big) = 0, \quad 1 \leq \alpha \leq k,
    \end{equation*}
    and $\mathbb{F}H\left(\mathcal{D}^*\right) = \mathcal{D}$ is the induced 
    $($co$)$distribution on $T^*Q$, i.e., 
    $\mathcal{D}^* = \{(q, p)\in T^*Q : p = g(\dot{q}, \cdot) \text{ for some } (q, \dot{q}) \in \mathcal{D}\}.$
\end{definition}

\input{Preliminaries/Hybrid_systems/Chaplygin_Sleigh}

We now combine the constraint continuous dynamics with the discrete dynamics given by impacts. The object that stores all the information about the constrained dynamics and the discrete transitions is called a constrained impact system, and is defined below.

\begin{definition}[Constrained impact system]\label{def:Constrained Impact System}
A constrained impact system is a 4-tuple $\mathcal{I} = (Q, H, \Sigma, \mathcal{D})$ such that
    \begin{enumerate}
        \item[\normalfont{(I.1)}] $Q$ is a finite-dimensional smooth manifold,
        \item[\normalfont{(I.2)}] $H \colon T^*Q \to \mathbb{R}$ is a natural Hamiltonian,
        \item[\normalfont{(I.3)}] $\mathcal{D} \subset TQ$ is a regular distribution, 
        \item[\normalfont{(I.4)}] $\Sigma \subset M$ is an embedded smooth manifold where $\codim\Sigma = 1$.
    \end{enumerate}
\end{definition}

\noindent{\color{header1} \textit{Example} \ref{example:Chaplygin sleigh} \ {\sffamily{(Continued).}}} In addition to the to the previous properties of the Chaplygin sleigh, we now define an impact surface for the system
\begin{align*}
    s \colon \text{SE}_2 \to \mathbb{R} : (x,y,\theta) \mapsto (\theta - \theta_0)(\theta + \theta_0).
\end{align*}
Given this map, $\Sigma = s^{-1}(\{0\}) = \{(x,y,\theta) \in \text{SE}_2: (\theta - \theta_0)(\theta + \theta_0) = 0\}$.% which implies that $\Sigma = \mathbb{S}^1 \subset \mathbb{R}^2$. 
With the impact surface, the Chaplygin sleigh becomes a constrained impact system in contrast to previously being a constrained mechanical system.

Next, the corresponding reset map for this system is given by $\Delta(v,\omega,\theta) = (v,-\omega,\theta),$ which represents the sleigh automatically changing direction when the critical angle $\theta_0$ is reached. This example will be further continued in Section \ref{subsec:chapsleigh_angle}.

%% file: Preliminaries/Hybrid_systems/Chaplygin_Sleigh.tex
\begin{example}[The Chaplygin sleigh]\label{example:Chaplygin sleigh} The Chaplygin sleigh is a pedagogical example of a (nonholonomically) constrained system. The configuration space of the Chaplygin sleigh is $\text{SE}_2$, which tracks the $x$ and $y$ position of the sleigh along with the direction $\theta$ its facing. The Lagrangian of this system is
\begin{align*}
    L = \frac{1}{2}\left(m(\dot{x}^2 + \dot{y}^2) + (I + ma^2)\dot{\theta}^2 - 2ma\dot{\theta}\left(\dot{x} \sin\theta - \dot{y} \cos\theta\right)\right)
\end{align*}
where $m$ is the mass of the sleigh, $a$ is the distance between the center of mass and the front of the sleigh, $I$ is the moment of inertia about the center of mass. To disallow for sliding in transverse directions, the sleigh is given the knife edge constraint,
\begin{align}\label{eq:sleigh_constraint}
    \dot{x} \sin\theta - \dot{y}\cos\theta = 0,
    \qquad \t{or} \qquad
    \eta = \sin(\theta)\,dx -\cos(\theta)\,dy.
\end{align}
Given the constraint \eqref{eq:sleigh_constraint}, the corresponding distribution is
\begin{align*}
    \mathcal{D} = \left\{ (x,y,\theta,\dot{x},\dot{y},\dot{\theta})\in TQ : \dot{x} \sin\theta - \dot{y}\cos\theta = 0\right\},
\end{align*}
which shows the relation between the allowable configurations and the constraints. More information can be found within \cite[Chapter~1.7]{bloch2004nonholonomic}.%, as mentioned below Equation \eqref{eqn:distribution}.
% \aps{Maria/Will, what is the name for the lagrangian }\wac{The Lagrangian has a name?}
% Applying the Lagrange-d’Alembert principle with the substitution $v = \dot{x}\cos\theta + \dot{y}\sin\theta$ and $\omega = \dot\theta$, the equations of motion for this Lagrangian are given by
% \begin{equation}\label{eq:sleigh_continuous}
% 	\begin{split}
% 		\dot{v} & = a\omega^2, \\
% 		\dot{\omega} & = -\frac{ma}{I+ma^2}v\omega, \\
% 		\dot{\theta} & = \omega.
% 	\end{split}
% \end{equation}
% Information regarding the derivation of the following equations can be found in .
% \wac{This example should be much shorter}\aps{It's shorter}
\end{example}

%% file: Preliminaries/Frobenius_Perron/Definitions_FP.tex
When studying the long time behaviour of dynamical systems through the probabilistic lens, one is interested in the evolution of densities instead of single points. Consider a large number of initial points $x_1, \dots , x_n \in X$, sampled randomly from some probability density function $f_0$, and assume that the dynamics are given by a 1-parameter group action $\phi_t \colon X \rightarrow X$. Each of these points are going to move around under the dynamics to points $x_1(t), \dots, x_n(t)$ after some time $t$. If $n$ is large and $t \to \infty$, computing, for each individual trajectory, its value at time $t$ becomes very expensive. 
Now, assume that there exists some other probability density $f_t$ such that $x_1(t), \dots, x_n(t)$ are roughly obtained by sampling $n$ points from this distribution. Then in order to analyze the long term behavior of the system it is sufficient to look at how $f_t$ evolves over time, and in particular, at $\lim_{t \to \infty}f_t$. The Frobenius-Perron operator dictates the evolution of $f_t$. 

In the following, we provide the formal definitions of the Frobenius-Perron operator for discrete and continuous time systems as well as its infinitesimal generators \cite{chaosfractalsnoise}.
\begin{definition}[Discrete Frobenius-Perron Operator]
    \label{def: Discrete_FP_Operator}
    Let $\varphi \colon M \to M$ be a nonsingular measurable transformation on a measure space $(M, \mathcal{A}, \mu)$. The (discrete) Frobenius-Perron operator corresponding to $\varphi$ is the unique linear operator  $P \colon L^1(M, \mathcal{A}, \mu) \to L^1(M, \mathcal{A}, \mu)$ defined by
    \begin{equation}\label{eq:Discrete_FP_def}
        \int_E Pf(x)\,d\mu = \int_{\varphi^{-1}(E)} f(x)\,d\mu, \qquad \text{for all } E\in\mathcal{A}.
    \end{equation}
\end{definition}
\begin{remark}[Measure-theoretic considerations]\label{sigma algebra remark}
    Throughout this paper, all sigma algebras are assumed to be the Borel algebra generated by the topology on the manifold ($Q$ or $M$, depending on the context), so measure spaces will be denoted by doubles $(M,\mu)$ rather than triples $(M, \mathcal{A}, \mu)$. Furthermore, all sets are assumed to be measurable unless stated otherwise. In addition, we interchangeably use our volume-forms as measures. Formally, $\mu(E) \coloneqq \int_E \mu$ for $\mu \in \Omega^n(M)$ and a measurable set $E \subseteq M$.
  
\end{remark}

Given an initial probability distribution $f \in L^1$, the Frobenius-Perron operator gives an evolution of $f$ by the transformation $\varphi$. Furthermore, for any $f \in L^1$, if $\varphi$ is an \textit{invertible} non-singular transformation, then  \vspace*{-3mm}
\begin{equation*}
    Pf(x) = f(\varphi^{-1}(x))\mathcal{J}^{-1}(x),
\end{equation*}
where $\mathcal{J}^{-1}$ is the determinant of the inverse of the Jacobian matrix. If $\varphi$ fails to be invertible, %then
\begin{equation}\label{eq:discFP}
    Pf(x) =\sum_{y\in \varphi^{-1}(\set{x})} f(y)\mathcal{J}^{-1}(y).
\end{equation}

The discrete Frobenius-Perron operator has a natural continuous-time system analog.
\begin{definition}[Continuous Frobenius-Perron Operator]
\label{def: Continuous_FP_Operator}
    Let $\varphi \colon M \times \mathbb{R}\to M$ be a flow and denote $\varphi_t(x) = \varphi(x,t)$. Suppose that the flow is nonsingular and measurable on the measure space $(M,\mu)$.
    For $E \subseteq M$, the (continuous) Frobenius-Perron operator $P_t \colon L^1(M,\mu) \to L^1(M,\mu)$ for each $t \in \mathbb{R}$ corresponding to $\varphi_t$ is defined by
    \begin{equation*}
        \int_E P_tf(x) \, d\mu = \int_{\varphi_{-t}(E)}f(x) \, d\mu \qquad \t{ for all } E\subseteq M
    \end{equation*}
\end{definition}

\begin{remark}
    The Frobenius-Perron and Koopman operators are adjoints of one another. Letting $K_t \colon L^\infty(M,\mu) \to L^\infty(M,\mu)$ be the Koopman operator, which is defined as
     \vspace*{-3mm}\begin{equation*}
        K_tg(x) = g(\varphi_t(x))
    \end{equation*}
    and $\langle\cdot,\cdot\rangle \colon L^1\times L^\infty \to \mathbb{R}$ denote the duality pairing (see \cite[Page~212]{chaosfractalsnoise} for duality pairings in this context) between $L^1$ and $L^\infty$, then $\langle P_t f, g \rangle = \langle f,K_tg\rangle$. 
\end{remark}

%% file: Preliminaries/Frobenius_Perron/infinitesimal_generators.tex
Computing the Frobenius-Perron operator directly from the definition requires knowledge of the flow $\varphi_t$ which is the solution to the differential equation $\dot{x} = X(x)$. This is often infeasible, especially if the analytical expression for $\varphi_t$ is not known. This problem can be sidestepped by considering the infinitesimal generator of the Frobenius-Perron operator, which allows us to describe the evolution of an initial probability density function as the solution to a linear partial differential equation.

\begin{definition}[Infinitesimal Generator of Frobenius-Perron Operator]
    Let $\{P_t : t \in \mathbb{R}\}$ be a family of Frobenius-Perron operators  corresponding to a the flow $\{\varphi_t\}$ arising from the differential equation $\dot{x} = X(x)$. The infinitesimal generator for the Frobenius-Perron operator, $A \colon \dom(A)\subset L^1(M,\mu) \to L^1(M,\mu)$ is defined as
    \begin{equation}
        Af = \lim_{\tau \to 0} \frac{P_{t+\tau}f-P_tf}{\tau}.
    \end{equation}
\end{definition}

We now state a theorem which allows for the computation of both $A$ and $P_t$, assuming that $f$ is differentiable.

\begin{theorem}[{\cite[Theorem~7.6.11]{chaosfractalsnoise}}]\label{thm:inf_gen}
    Let $X\in\mathfrak{X}(M)$ be a smooth vector field inducing the flow $\varphi_t$. Suppose that the measure space $(M, \mu)$ has a smooth measure, i.e., 
    $\mu\in\Omega^n(M)$ is a differentiable volume-form. 
    Then the infinitesimal generator $A$ for the Frobenius-Perron operator is given by
    \begin{equation}\label{eq:Cont_inf_gen}
        Af = -df(X) - f\cdot \mathrm{div}_\mu X.
    \end{equation}
    Moreover, if $u(t,x) = P_tf(x)$, then $u$ solves the partial differential equation
    \begin{equation}\label{eq:Cont_FP}
       \frac{\partial u}{\partial t} + du(X) = -u\cdot\mathrm{div}_\mu(X), \qquad u(0,x) = f(x),
    \end{equation}
    where $\mathrm{div}_\mu(X)$ is the divergence of the vector field $X$ with respect to the volume form $\mu$.
\end{theorem}

\begin{proof}
    Note that \eqref{eq:Cont_FP} holds as $AP_t = P_tA$, so it remains to show \eqref{eq:Cont_inf_gen}. Let $f \in L^1$ be an initial probability distribution function, then
    \begin{equation*}
        \int_E (P_t f) \mu = \int_{\varphi_{-t}(E)} f\mu
        = \int_{E} (\varphi_{-t})^* f \cdot  (\varphi_{-t})^* \mu,
    \end{equation*}
    which provides $(P_tf) \mu = (\varphi_{-t})^* f \cdot \! (\varphi_{-t})^* \mu$. Differentiating both sides with respect to $t$ and applying the product rule yields  \vspace*{-3mm}
    \begin{equation*}
        (Af)\mu = -df(X)\mu - f\cdot\mathrm{div}_\mu(X)\mu.
    \end{equation*}
    This immediately implies \eqref{eq:Cont_inf_gen} as $\mu\ne 0$.
\end{proof}

%% file: Preliminaries/Reduction/reduction_general.tex
Numerical computations of the Frobenius-Perron operator become costly with increasing dimension. This is why reduction techniques should be employed to make numerical simulations more tractable. At its core, the theory of reduction tries to deal with the following question: {How can we reduce the dimensionality of a system, if there is access to additional information?} There are different approaches to answering this question, depending on the type of insight available. In this work, we deal with systems that exhibit symmetry. Hybrid systems with symmetry are present in a myriad of examples, ranging from multi-legged locomotion \cite{animal_gaits_1999,robot_symmetry_ZR}, image recognition and robotic perception \cite{robotic_perception} to climate dynamics \cite{el_nino}. 

Mathematically, symmetries are expressed as actions of a Lie group on the state space $Q$, which is assumed to be a manifold. From here, depending on what kind of structure is available on the configuration manifold there are different reduction techniques that can be applied. For instance, if the manifold is the cotangent bundle equipped with a symplectic form then we are in the realm of cotangent reduction \cite{ctg_reduction1}. On the other hand, if we are dealing with a Poisson manifold, then techniques of Poisson reduction can be applied. A particular case that fits in both of these scenarios is Lie-Poisson reduction, which deals with situations when the state space is  a Lie group, acting on itself by left or right translation \cite{marsden_ratiu}. For example, this is the case for systems whose coordinates are Cartesian ($Q = \mathbb{R}^n$), polar ($Q = \mathbb{S}^1$, $\t{SO}_n$, etc.), or a combination 
% of angles and real valued coordinates 
($Q = \t{SE}_n$). Classical Lie-Poisson reduction is capable of reducing the dimension from $2n$ down to $n$.

Recent efforts have been made towards creating a theory of cotangent reduction for hybrid systems \cite{Ames2006HybridCB}. Other versions of reduction have been carried out for impact systems, e.g. \cite{ames2006,leo_reduction,EYREAIRAZU202194} and the references therein; however, there is still need for an analysis of Lie-Poisson reduction for hybrid systems. As will be discussed in Section \ref{sec:reduction} for hybrid systems on Lie groups, the best reduction we can attain is from $2n$ to $n + 1$ dimensions.

%% file: Preliminaries/Reduction/reduction_Lie_Poisson.tex
In the following we present the continuous version of Lie Poisson reduction. Building upon this prerequisite, we will develop hybrid Lie Poisson reduction in Section \ref{subsec:chikos}. 
Let $G$ be a Lie group and denote by $\mathfrak{g}$ and $\mathfrak{g}^*$ its Lie algebra and dual, respectively. 
\begin{minipage}{.44\linewidth}
    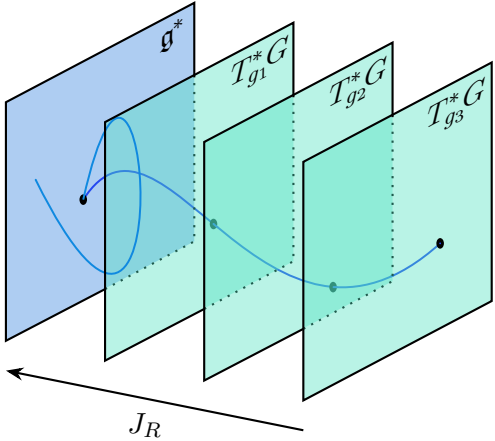
\begin{figure}[H]
        \centering
        \input{TikZ_Code/Maria_reduction}
        \caption{Schematic drawing of continuous Lie-Poisson reduction.}% 
        \label{fig:reduction scheme}
    \end{figure}
\end{minipage} \hfill
\begin{minipage}{.525\linewidth}
    Consider the action of $G$ on $T^*G$ defined by the lift of left translations: $g \cdot \alpha = \ell_{g^{-1}}^*\alpha$. The momentum map of this action and the momentum map for the lift of right translations are% given by%\vspace{-.25\abovedisplayskip}
    \begin{align*}
        J_L \colon & T^*G \to \mathfrak{g}^* : \alpha_g \mapsto r_g^*\alpha_g, \\
        J_R \colon & T^*G \to \mathfrak{g}^* : \alpha_g \mapsto  \ell_g^*\alpha_g, %\vspace{-.25\belowdisplayskip}
    \end{align*}
    respectively. 
    % and for right translation, the momentum map is given by
    The cotangent bundle is equipped with the canonical symplectic form $\omega = dg\wedge dp_g$, which gives raise to the usual Poisson bracket %\vspace{-.25\abovedisplayskip}\vspace{-.25\belowdisplayskip}
    \begin{align*}
        \{\cdot, \cdot\} 
        \colon & C^\infty(T^*G)\times C^\infty(T^*G) \to C^\infty(T^*G) \\
        \colon & (f, k) \mapsto \omega(X_f, X_k),
    \end{align*}
    where $X_f$ and $X_k$ are the Hamiltonian vector fields of \hfill $f$ \hfill and \hfill $k$, \hfill respectively. \hfill We \hfill assume \hfill that \hfill the 
\end{minipage} \\
functions in $C^\infty(T^*G)$ are left invariant. Then, using the right momentum map, one can project the dynamics on $T^*G$ to $\mathfrak{g}^*$ without losing essential information about the full trajectories (Figure \ref{fig:reduction scheme}).

The quotient manifold $T^*G/G$ is isomorphic to the dual of the Lie algebra $\mathfrak{g}^*$ \cite[Chapter~13.3]{marsden_ratiu} and can be showed to be a Poisson manifold with Poisson bracket given by  \cite{Marsden1983}
\begin{equation*}
   \{f, k\} (\mu) = \langle \mu, [df, dk]\rangle, \qquad \text{for }  f, k \in C^\infty(\mathfrak{g}^*).
\end{equation*}
where the bracket $[df, dk]$ denotes the Lie bracket in $\mathfrak{g}$. Note that the differentials $dk$ and $df$ can indeed be considered elements of the Lie algebra, due to the identification of $\mathfrak{g}^{**}$ with $\mathfrak{g}$.

Now, assume that the dynamics on the full cotangent bundle is given by a left invariant Hamiltonian $H \colon T^*G \to \mathbb{R}$. Using the available Poisson structure, it is possible to define a reduced dynamical system on $\mathfrak{g}^*$. This is summarized in the following theorem.

\begin{theorem}\label{thm:cont reduction}
    Let $H \colon T^*G \to \mathbb{R}$ be a left invariant Hamiltonian, so $H(q, p_q) = H(gq, \ell_{g}^*p_q)$ for all $g \in G$ and $(q, p_q) \in T^*G$. Then, the restriction of $H$ to $\mathfrak{g}^*$, denoted by $h = H|_{\mathfrak{g}^*}$, satisfies
    \begin{align*}
        {\color{white}H(q, p_q) = h (J_R(q, p_q)),}H = h \circ J_R 
        \qquad \t{if and only if} \qquad
        H(q, p_q) = h (J_R(q, p_q)),{\color{white}H = h \circ J_R }
    \end{align*}
    and the flow $\Phi_t$ of $H$ on $T^*G$ is related to the flow $\phi_t$ of $h$ on $\mathfrak{g}^*$ by
    \begin{align*}
        J_R \circ \Phi_t = \phi_t \circ J_R.
    \end{align*}
    Moreover, the reduced equations of motion on $\mathfrak{g}^*$ can be written as
    \begin{equation}\label{eq:Lie_poisson_cont}
        \frac{d}{dt}\phi(t) (\mu_0) = \dot{\mu} = \ad^*_{dh}\mu,
    \end{equation}
    where $\ad_{dh}^* \colon \mathfrak{g}^*\to \mathfrak{g}^*$ and $\ad^*_{dh}\mu (\xi) = \mu([dh, \xi])$ for all $\xi \in \mathfrak{g}^*$.
\end{theorem}
\begin{proof}
    See Theorem 13.4.1 in \cite{marsden_ratiu}.
\end{proof}
% A proof of this theorem can be found in \cite{marsden_ratiu}.
\begin{example}
    Consider the action of $\t{SO}_3$ on itself. The Lie algebra of $\t{SO}_3$ is $\mathfrak{so}_3 \cong \R^3$ with Lie bracket $[u, v] = u \times v$ and its dual $\mathfrak{so}_3^* \cong \R^3$ consists of row vectors $p$. The right momentum map is $J_R(A, p) = pA$, and the reduced Poisson bracket on $\mathfrak{so}_3$ is: $\{f, k\}(p) = p (df \times dk) = p^T \cdot (df \times dk)$. The reduced equations of motion \eqref{eq:Lie_poisson_cont} become
    $\dot{\mu}^T = \ad^*_{dh} \mu = \mu^T \times dh$ or $\dot{\mu} = \mu \times dh$ if  $\mu$ is viewed here as a column vector.
\end{example}

In Section \ref{sec:reduction}, we will extend the reduction presented above to hybrid systems.

%% file: TikZ_Code/Maria_reduction.tex
\begin{tikzpicture}[x=0.75pt,y=0.75pt,yscale=-1,xscale=1]
    \fill[fill=MariaBlue, fill opacity=0.45] (100,220) -- (100,100) -- (195,50) -- (195,170);
    \useasboundingbox(95,47.5) rectangle (350,270);
    % Purple
    %blue!50!MariaBlue
    \draw [blue!50!MariaBlue, line width=0.75] (139,149) .. controls (179,87) and (231,250) .. (319,171);
    \fill[blue!50!MariaBlue] (139,149) circle [x radius = .12mm, y radius = .12mm];
    % Second dot
    \fill (204.7,161.25) circle [x radius = .5mm, y radius = .7mm];
    \fill[blue!50!MariaBlue] (204.7,161.25) circle [x radius = .5, y radius = .5];
    \draw[blue!50!MariaBlue, line width=0.75] (204.7,161.25) -- (206.9,163.1);
    
    % Third dot
    \fill (265,193) circle [x radius = .5mm, y radius = .7mm];
    \fill[blue!50!MariaBlue] (265.2,193) circle [x radius = .12mm, y radius = .12mm];
    \draw[blue!50!MariaBlue, line width=0.75] (265.2,193) -- (267.2,193);
    
    \draw[blue!50!MariaBlue, line width=0.75] (294,187.725) -- (296,186.85);

    \draw[thick, dash pattern={on 0.84pt off 2.51pt}] (150,194) -- (195,170) -- (195,87.5);
    \fill[fill=MariaGreen, fill opacity=0.40] (150,230) -- (150,110) -- (245,60) -- (245,180);

    \draw[thick, dash pattern={on 0.84pt off 2.51pt}] (200,204) -- (245,180) -- (245,97.5);
    \fill[fill=MariaGreen, fill opacity=0.40] (200,240) -- (200,120) -- (295,70) -- (295,190);

    \draw[thick, dash pattern={on 0.84pt off 2.51pt}] (250,214) -- (295,190) -- (295,107.5);
    \fill[fill=MariaGreen, fill opacity=0.40] (250,250) -- (250,130) -- (345,80) -- (345,200);
    % First dot
    \fill (139,149) circle [x radius = .5mm, y radius = .7mm];
    
    % Blue
    \draw [color={rgb, 255:red, 16; green, 137; blue, 224}, line width=0.75] (139,149) .. controls (171,6) and (192,294) .. (115,139);
    \fill[color={rgb, 255:red, 16; green, 137; blue, 224}] (115,139) circle [x radius = .12mm, y radius = .12mm];
    
    % Tangent planes
    \draw[thick] (150,194) -- (100,220) -- (100,100) -- (195,50) -- (195,87.5);
    \draw[shift={(50,10)}, thick] (150,194) -- (100,220) -- (100,100) -- (195,50) -- (195,87.5);
    \draw[shift={(100,20)}, thick] (150,194) -- (100,220) -- (100,100) -- (195,50) -- (195,87.5);
    \draw[shift={(150,30)}, thick] (100,220) -- (100,100) -- (195,50) -- (195,170) -- cycle;
    
    % Final dot
    \fill (319,171) circle [x radius = .5mm, y radius = .7mm];
    \fill[blue!50!MariaBlue] (319,171) circle [x radius = .12mm, y radius = .12mm];
    
    \draw (184,66) node [anchor = center, inner sep=0.75pt, yslant = .531, xslant = .025]{$\mathfrak{g}^{\hspace{-.3mm}*}$};
    \draw (184,66) node [anchor = center, inner sep=0.75pt, yslant = .531, xslant = .025, shift={(45,-39)}]{$T_{g_1}^*\!G$};
    \draw (184,66) node [anchor = center, inner sep=0.75pt, yslant = .531, xslant = .025, shift={(95.5,-76)}]{$T_{g_2}^*\!G$};
    \draw (184,66) node [anchor = center, inner sep=0.75pt, yslant = .531, xslant = .025, shift={(146,-113)}]{$T_{g_3}^*\!G$};
    
    \draw[thick, Stealth-](100,235) -- (250,265);
    \draw (160,255) node [anchor=north west, inner sep=0.75pt]{$J_R$};
\end{tikzpicture}

%% file: Results/introduction_results.tex
Our results are twofold. First, we define the hybrid Frobenius-Perron operator, which is the analogous operator to the continuous case, extended for when the underlying dynamics is hybrid. Then, we extend Lie-Poisson reduction to hybrid systems whose state space is a Lie group.

%% file: Results/Hybrid_FP/definition_of_hybrid_FP.tex
\label{subsec:chikos}
In this section, we derive the infinitesimal generator for the Frobenius-Perron operator of hybrid systems. As hybrid systems are a mixture of both discrete- and continuous-time systems, the infinitesimal generator will be a combination of both \eqref{eq:discFP} and \eqref{eq:Cont_FP}. 

\begin{definition}\label{def:hybrid_FP}
    Let $\varphi_t^\mathcal{H} \colon M \to M$ be a hybrid flow on the differentiable manifold $M$. Assume there exists a differentiable volume form $\mu \in \Omega^n(M)$ such that $(M, \mu)$ is a measure space and that the flow $\varphi_t^\mathcal{H}$ is complete. Then, the hybrid Frobenius-Perron operator corresponding to $\varphi_t^\mathcal{H}$ is the unique linear operator $P_t^\mathcal{H} \colon L^1(M, \mu) \to L^1(M, \mu)$ implicitly defined by
    \begin{equation}\label{eq:hybrid_FP_def}
        {\color{white} \t{for all } E \subseteq M.}\qquad
        \int_E P_t^\mathcal{H}f \, d\mu = \int_{\varphi_{-t}^\mathcal{H} (E)}\!\!\!f \, d\mu \qquad \t{for all } E \subseteq M.\qquad
    \end{equation}
    %\mao{In order to say $\varphi_{-t}^\mathcal{H}$ do we need to assume no Zeno?} \wac{We need more than that, we need that $-t$ is in the domain of $x$ (which implies nonZeno)}
    When $P_t^\mathcal{H}$ is continuous, the hybrid infinitesimal generator is defined to be
    % Its infinitesimal generator is defined as in the continuous case to be:
    \begin{equation}\label{eq:hybrid_inf_def}
        A^\mathcal{H}f = \lim_{\Delta t \to 0}\frac{P_{t + \Delta t}^\mathcal{H}f - P_{t}^\mathcal{H}f}{\Delta t}.
    \end{equation}
\end{definition}

As in the continuous case, our goal is to find an explicit formula for the infinitesimal generator. Moreover, as in Theorem \ref{thm:inf_gen} we want to find a partial differential equation for $u(t, x) \coloneqq P_t^\mathcal{H} f(x)$.

\begin{remark}\label{remark:dimensionality}
    Before we proceed, we will further examine \eqref{eq:hybrid_inf_def}. Assume $E$ is a sufficiently small ball about some point $x \in M \smm \mathcal{S}$, and that $\varphi_{-t}^\mathcal{H}(E) \cap \mathcal{S} = \varnothing$ for some time $t \in \R$. Then, within $E$, the dynamics of point $x$ follows the continuous part of $\varphi_t^\mathcal{H}$, which is smooth by the quasi-smooth dependence propriety. Since we are taking the limit as $\Delta t \to 0$, we can assume that $\varphi_{s}^\mathcal{H}(E)\cap \mathcal{S} = \varnothing $ for $-t - \Delta t \leq s \leq -t $. Hence, the hybrid infinitesimal generator should be the same as in the continuous case. 

    Now consider what happens when $E \cap \mathcal{S} \neq \varnothing$. At a first glance, it is tempting to define the transfer operator as the one corresponding to the discrete reset map $\Delta$; however, the guard $\mathcal{S}$ as well as its tangent space is $n- 1$ dimensional, whereas the image of the reset map is $M$ is $n$ dimensional. This is an issue since both sides of \eqref{eq:hybrid_FP_def} integrate $\mu$, which is a volume form on $M$. In order to deal with this mismatch in dimensionality, we will extend the differential of $\Delta$, so that it acts on vectors in the tangent space of $M$ restricted to base points in $\mathcal{S}$. We will do this by taking advantage of the fact that any vector in the tangent space at $x \in \mathcal{S}$ can be decomposed into its component along the flow $v_x$, and $v_s \in T_x\mathcal{S}$ as in Figure \ref{fig:extended_differential}. 
    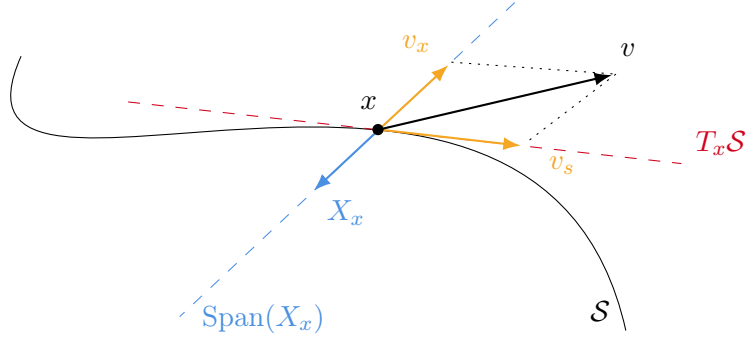
\begin{figure}[!ht]
        \centering
        \input{TikZ_Code/MariaTangentSpaceDecomp}
        \caption{Decomposition of vectors in the tangent space of the ambient manifold $M$, with base point on the guard.}
        \label{fig:extended_differential}
    \end{figure}
 
\end{remark}

In the following, we will make the decomposition from the previous remark more precise. 

\begin{definition}[Extended differential]\label{def:Extended differential}
    Let $\mathcal{H} = (M, \mathcal{S}, \Delta, X)$ be a smooth HDS that satisfies propriety \ref{def:Quasi-smooth Dependence Property} and let $x \in \mathcal{S}$. The augmented differential of the reset map is the linear map
    \begin{align*}
        \Delta_*^X \colon T_xM \to T_{\Delta(x)}M : v \mapsto 
        \begin{cases}
            \Delta_* v  & \text{if }  v \in T_x\mathcal{S}; \\
            c X_{\Delta(x)} & \text{if } v = cX_x \in \t{\textup{Span}}(X_x).
        \end{cases}
    \end{align*}
\end{definition}

When defining the discrete part of the Frobenius-Perron operator, we will use the Jacobian of the augmented differential rather than the Jacobian of the reset map. This fixes the dimensionality issue, and is indeed the correct procedure as will become apparent in the proof of Theorem \ref{thm:restricted chikos} below.

The remainder of this section is dedicated to finding a hybrid analog of  formula \eqref{eq:discFP}. The result is the summarized in Theorem \ref{thm:general chikos}, and is the main contribution of this section. As in the discrete case, the goal is to perform a change of coordinates in the integral \eqref{eq:hybrid_FP_def} so as to obtain an explicit formula for the Frobenius-Perron operator. Due to the dimensionality issues presented in Remark \ref{remark:dimensionality}, this is not straightforward, and will be carefully analysed in the following. 

We begin by giving a precise definition of the determinant for general manifolds. The determinant will be important later on, when extending the change of variables formula from the discrete case \eqref{eq:discFP}.

%% file: TikZ_Code/MariaTangentSpaceDecomp.tex
\begin{tikzpicture}[x=0.75pt,y=0.75pt,yscale=-1,xscale=1]
    %Curve Lines [id:da31919837529231854] 
    \draw (124,48.1) .. controls (75,158.1) and (385,-10.9) .. (429,186.1);
    
    % T_x S red dashed line
    \draw [color={rgb, 255:red, 208; green, 2; blue, 27 }, dash pattern={on 4.5pt off 4.5pt}]  (178,71.1) -- (457,102.1) ;
    
    % X(x) arrow
    \draw [thick, -Latex, MariaBlue] (304.5,85) -- (271.51,115.79);
    
    % Span(X(x)) line 
    \draw[MariaBlue] [dash pattern={on 4.5pt off 4.5pt}]  (373,21.1) -- (204,179.1);
    
    % v_s arrow
    \draw [thick,-Latex,color={rgb, 255:red, 245; green, 166; blue, 35}] (304.5,85) -- (376.01,93.1);

    % v arrow
    \draw[thick,-Latex] (304.5,85) -- (422.05,57.56);
    
    % v_x arrow
    \draw [thick,-Latex,color={rgb, 255:red, 245; green, 166; blue, 35 }] (304.5,85) -- (339.58,52.5);
    
    % black dotted lines
    \draw  [dash pattern={on 0.84pt off 2.51pt}] (424,57.1) -- (378,92.1);
    \draw  [dash pattern={on 0.84pt off 2.51pt}] (341,51.1) -- (424,57.1);
    
    % Dot:
    \node at (304,85) [circle,fill,inner sep=1.5pt]{};
    
    % Text:
    \draw (410,170) node [anchor=north west, inner sep=0.75pt] {$\mathcal{S}$};
    \draw (294,67) node [anchor=north west, inner sep=0.75pt] {$x$};
    \draw (463,84.4) node [anchor=north west, inner sep=0.75pt] {$\textcolor[rgb]{0.82,0.01,0.11}{T_{x}\mathcal{S}}$};
    \draw (277,119.4) node [anchor=north west, inner sep=0.75pt,MariaBlue] {$X_x$};
    \draw (214,171.4) node [anchor=north west, inner sep=0.75pt,MariaBlue] {$\text{Span}(X_x)$};
    \draw (425,40.4) node [anchor=north west, inner sep=0.75pt] {$v$};
    \draw (389,97.4) node [anchor=north west, inner sep=0.75pt] {$\textcolor[rgb]{0.96,0.65,0.14}{v_{s}}$};
    \draw (315,36.4) node [anchor=north west, inner sep=0.75pt] {$\textcolor[rgb]{0.96,0.65,0.14}{v_{x}}$};
\end{tikzpicture}

%% file: Results/Hybrid_FP/intermediary_lemmas.tex
\begin{definition}[Form determinant \cite{sternberg1980ralph}]\label{def:Form Determinant}
    Let $M$ and $N$ be $n$-dimensional smooth manifolds with volume-forms $\mu \in \Omega^n(M)$ and $\eta \in \Omega^n(N)$. Let $F \colon TM \to TN$ be a smooth map which is linear in the fibers. Then the determinant of $F$ with respect to $\mu$ and $\eta$ is defined to be the unique $\mathcal{C}^\infty(M)$ function such that
    \begin{equation*}
        \det_{\mu \to \eta }(F)\cdot\mu%(\,\cdot\,,\ldots,\,\cdot\,) = \eta(F\,\cdot\,,\ldots,F\,\cdot\,) := 
        = F^*\eta% (\,\cdot\,, \ldots,\, \cdot\,)
    \end{equation*}
    Moreover, if $F \colon TM\to TM$ then we will write $\displaystyle{\det_{\mu \to \mu}} (F)$ as $\det_\mu (F)$.
\end{definition}
\begin{remark}
    If $f \colon M \to N$ is a smooth map, $\displaystyle{\det_{\mu\to\eta}} (f_*)$ is the usual Jacobian determinant.% $\det Df$.
    % If $F = f_*$ for smooth $f \colon M \to N$, then $\displaystyle{\det_{\mu\to\eta}} (f_*)$.  
\end{remark}

\begin{remark}
    The existence and uniqueness of the determinant comes from the fact that both $\mu$ and $F^*\eta$ are top-forms on $M$. The determinant does not depend on the vectors that $\mu$ and $F^*\eta$ act on. In particular, 
    \begin{align*}
        \quad \quad  \det_{\mu \to \eta }(F) \cdot \mu (v^1, \dots v^n) = F^*\eta(v^1, \dots v^n) \qquad \text{for all }v^1, \ldots, v^n \in T_xM.
    \end{align*}
\end{remark}

When $F = f_*$, there is an elegant formula for the determinant of the inverse transformation.
\begin{lemma}
    \label{lemma:Relation Between Form Determinant and Inverse Form Determinant}
    For a diffeomorphism $f \colon M \to N$ and volume-forms $\mu \in \Omega^n(M)$ and $\eta \in \Omega^n(N)$,% we have
    \begin{equation*}  
        \det_{\mu \to \eta }(f_*^{-1}) = \frac{1}{ \underset{\eta \to \mu}{\det} (f_*) \circ f^{-1}} .
    \end{equation*}
\end{lemma}
\begin{proof}
    See Appendix \ref{lemma:Relation Between Form Determinant and Inverse Form Determinant Proof}.
\end{proof}
We now turn to the case where $f \colon \mathcal{S} \to M$ is the reset map $\Delta$. This is precisely where the augmented differential becomes important since it offers a way to relate the determinant of the reset map with respect to the induced measures on $\mathcal{S}$ and $\Delta(\mathcal{S})$ to the determinant of a linear map from $TM$ to $TM$ with respect to just $\mu$. We formalize this through the following lemma.

\begin{lemma}[Transverse property]
    \label{lemma:Transverse Property}
    Let $\Delta_*^X \colon TM|_\mathcal{S}\to TM|_{\Delta(\mathcal{S})}$ be the augmented differential and let $\iota_\mathcal{S} \colon \mathcal{S} \hookrightarrow M$ and $\iota_{\Delta(\mathcal{S})} \colon \Delta(\mathcal{S}) \hookrightarrow M$ denote the respective inclusion maps.
    If $X$ is transverse to $\mathcal{S}$, then
    \begin{equation*}
        \det _{\mu\to\mu} \Delta^X_* = \det _{\alpha \to \beta } \Delta_*
    \end{equation*}
    where $\alpha = \iota^*_\mathcal{S}i_X\mu$, $\beta = \iota^*_{\Delta(\mathcal{S})}i_X\mu$.
\end{lemma}
\begin{proof}
    See Appendix \ref{lemma:Transverse Property Proof}.
\end{proof}

Keeping the same notation as in Lemma \ref{lemma:Transverse Property}, the definition of the determinant gives us
\begin{equation*}
    \det_{\alpha \to \beta}(\Delta_*) \cdot \alpha = \Delta^*\beta
\end{equation*}
and plugging in the expressions for $\alpha$ and $\beta$, we see that
\begin{equation}\label{eqn:this equation}
    \det_{\alpha \to \beta}(\Delta_*) \cdot \iota^*_Si_X\mu = \Delta^*\iota^*_{\Delta(S)} i_X\mu,
\end{equation}
which is precisely the defining equation for the hybrid Jacobian $\mathcal{J}_\mu^X(\Delta)$. Thus, by the uniqueness of the determinant we have
\begin{equation*}
    \det_{\alpha\to\beta}\Delta_* = \mathcal{J}_\mu^X(\Delta),
    \qquad \t{and finally}, \qquad
    \det_{\mu\to\mu} \Delta_*^X = \mathcal{J}_\mu^X(\Delta).
\end{equation*}
In other words, the hybrid Jacobian is equal to the determinant of the augmented differential, as stated precisely in the following lemma.

\begin{lemma}[Relation between determinant and Jacobian]
    \label{lemma:Relation Between Determinant and Jacobian}
    The determinant of the augmented differential is the hybrid Jacobian, i.e.,
    \begin{equation*}
        \det_{\mu\to\mu} \Delta_*^X = \mathcal{J}^X_\mu(\Delta).
    \end{equation*}
\end{lemma}

Combining lemmas \ref{lemma:Relation Between Determinant and Jacobian} and \ref{lemma:Relation Between Form Determinant and Inverse Form Determinant}, we obtain a formula for how the measure changes under the coordinate transformations $x \mapsto \Delta(x)$.
\begin{lemma}\label{lemma:Hybrid Frobenius Perron Operator}
    Suppose that $\Delta \colon \mathcal{S}\to\Delta(\mathcal{S})$ is invertible and the linear transformation $\Delta_*^X$ is non-singular. Then,
    \begin{equation*}
        \mu_{\Delta^{-1}(x)} \Big(\!\!\left(\Delta_*^X\right)^{-1}\!v^1,\ldots,\left(\Delta_*^X\right)^{-1}\!v^n\Big) = 
        \frac{1}{\big(\mathcal{J}_\mu^X(\Delta) \circ \Delta^{-1}\big)(x)}\cdot \mu_x(v^1, \ldots, v^n),
    \end{equation*}
    for all $v^1,\ldots, v^n \in T_xM|_{\mathcal{S}}$.
\end{lemma}
\begin{proof} 
    By the definition of the form determinant,
    \begin{equation}\label{eqn:proof 1}
        \mu_{\Delta^{-1}(x)} \Big(\!\!\left(\Delta_*^X\right)^{-1}\!v^1,\ldots,\left(\Delta_*^X\right)^{-1}\!v^n\Big) = \det_{\mu}\left((\Delta_*^X)^{-1}\right)\cdot\mu_x(v^1, \ldots, v^n).
    \end{equation}
    Utilizing Lemma \ref{lemma:Relation Between Form Determinant and Inverse Form Determinant} to rewrite \eqref{eqn:proof 1}, we obtain
    \begin{equation}
        \mu_{\Delta^{-1}(x)} \Big(\!\!\left(\Delta_*^X\right)^{-1}\!v^1,\ldots,\left(\Delta_*^X\right)^{-1}\!v^n\Big) = \frac{1}{\big(\!\det_{\mu}(\Delta_*^X) \circ \Delta^{-1}\big)(x)}\cdot\mu_x(v^1, \ldots, v^n),
    \end{equation}
    and Lemma \ref{lemma:Relation Between Determinant and Jacobian} gives us the desired result.
\end{proof}

We are now ready to state the main theorem.

%% file: Results/Hybrid_FP/reduction_proofs.tex
\begin{theorem}[Hybrid infinitesimal generator for invertible $\Delta$]
    \label{thm:restricted chikos}
    Let $\mathcal{H}$ be a smooth hybrid dynamical system with an invertible reset map $\Delta$ and dynamics given by \eqref{eq:HDS}. Let $\mu \in \Omega^n(M)$ be a reference volume-form and suppose that $\mathcal{J}_\mu^X(\Delta) \ne 0$. Then, the hybrid transfer operator $u(t,x) \coloneqq P_t^\mathcal{H}f(x)$ satisfies the following
    \begin{equation}\label{eq:chikos}
        \begin{cases}
            \ds \frac{\partial u}{\partial t} + du(X) = -u \cdot \mathrm{div}_\mu(X) & \t{for }x \notin \Delta(\mathcal{S}); \\[1ex]
            \ds u(t^+,x) = \frac{u(t^-, \Delta^{-1}(x))}{\big(\mathcal{J}_\mu^X(\Delta) \circ \Delta^{-1}\big)(x)}  & \t{for } x \in \Delta(\mathcal{S}).
        \end{cases}
    \end{equation}
\end{theorem}
\begin{proof}
    We will first derive the continuous case of \eqref{eq:chikos}; let $x \notin \Delta(\mathcal{S})$. To do so, perform a change of coordinates in Equation \eqref{eq:hybrid_FP_def} to obtain
    \begin{align}\label{eq:change_of_coords}
        \int_E P_t^\mathcal{H} f(x) \, d\mu 
        = \int_{\varphi_{-t}^\mathcal{H}(E)}f(x) \, d\mu 
        = \int_E \big(f \circ \varphi_{-t}^\mathcal{H}\big)(x) \left(\varphi_{-t}^\mathcal{H}\right)^* \!\mu,
    \end{align}
    which holds for any measurable $E \subseteq M$. Pointwise, we have 
    \begin{equation}\label{eq:proof_chikos_1}
        P_t^\mathcal{H} f(x) \mu = \big(f \circ \varphi_{-t}^\mathcal{H}\big)(x) \left(\varphi_{-t}^\mathcal{H}\right)^* \!\mu.
    \end{equation}
    Since $x \notin \Delta(\mathcal{S})$, the quasi-smooth dependence propriety guarantees the existence of a small interval $(-\varepsilon, \varepsilon)$ such that $\varphi_{-\tau}^\mathcal{H}(x) \notin \Delta(\mathcal{S})$ for all $\tau \in (-\varepsilon, \varepsilon)$, which, in turn, implies the map $\tau \mapsto \varphi_{-\tau}^\mathcal{H}$ is smooth in a neighborhood of $x$. We can then differentiate \eqref{eq:proof_chikos_1} with respect to $t$ to obtain 
    \begin{equation*}
        \frac{\partial u}{\partial t} \mu = -du(X) \mu - u \cdot \mathrm{div}_\mu(X) \mu.
    \end{equation*}
    Finally, dividing out by $\mu \neq 0$ yields the $x\not\in \Delta(\mathcal{S})$ case of \eqref{eq:chikos}. 

    Now, suppose that $x \in \Delta(\mathcal{S})$ at $t = 0$, i.e., the starting point at $x$ is right after an impact has occurred. Yet again using the quasi-smooth dependence property, let $\varepsilon \in \R^+$ be such that $\varphi_{-t}^\mathcal{H}(x) \notin \mathcal{S}$ for all $t \in (0,\varepsilon)$ so that $\varphi_{-t}^\mathcal{H}(x) = \varphi_{-t}(x)$ and $\lim_{t \nsea 0} \varphi_{-t}^\mathcal{H}(x) = \Delta^{-1}(x)$. The goal will be to use equation \eqref{eq:proof_chikos_1} and take the limit as $t\nsea 0$.
    
    % By \cite[Theorem~4.2]{hybrid_forms}
    First, we analyze the behavior of the pushforward of vectors tangent to $\mathcal{S}$ by the flow $\varphi_{-t}^\mathcal{H}$, and we claim that
    \begin{equation}\label{eqn:claim}
        \lim_{t \nsea 0} (\varphi_{-t}^\mathcal{H})_*v = (\Delta_*^X)^{-1}v, \qquad \t{for all } v \in T_xM.
    \end{equation}
    To show this, we consider two cases based on the decomposition $T_xM = T_x\Delta(\mathcal{S}) \oplus \t{Span}(X_x)$. First, let $v \in T_x\Delta(\mathcal{S})$. We use an argument similar to that of Theorem 4.2 in \cite{hybrid_forms}, which is depicted in Figure \ref{fig:proof chikos discrete vectors}.
    \begin{align*}\label{eqn:other split}
        \begin{split}
            \lim_{t \nsea 0}(\varphi_{-t}^\mathcal{H})_*v
            % = & \, \lim_{t \nsea 0}(\varphi_{-t}^\mathcal{H})_* \gamma'(0)
            % && \t{} \\
            = & \, \Delta_*^{-1} v
            \hspace{3cm} \t{Since $-t < 0$ for all $t \in (0, \varepsilon)$} \\
            = & \, \big(\Delta_*^X\big)^{-1} (v)
            \hspace{2.08cm} \t{By the definition of $\big(\Delta_*^X\big)^{-1}$}
        \end{split}
    \end{align*}
    \begin{figure}
        \centering
        \begin{subfigure}[H]{0.47\textwidth}
            \centering
            \begingroup
                \tikzset{every picture/.style={scale=0.75}}%
                \input{TikZ_Code/MariaPushforward}
            \endgroup
            \caption{Pushforward of the flow on vectors in the tangent space of the impact surface. For a curve $\gamma(s) \in \Delta(\mathcal{S})$ that passes through $x$, the curve $\Delta^{-1}(\gamma(s)) \in \mathcal{S}$ passes through $y = \Delta^{-1}(x)$. Since $\lim_{t \nsea 0}(\varphi_{-t}^\mathcal{H}) = \Delta^{-1}\gamma$, in the limit the pushforward $v$ under the flow is the tangent vector at $y$ of $\Delta^{-1}(\gamma)$ depicted in red.}
            \label{fig:proof chikos discrete vectors}
        \end{subfigure}
        \hfill
        \begin{subfigure}[H]{0.47\textwidth}
            \centering
            \begingroup
                \tikzset{every picture/.style={scale=0.73}}
                \input{TikZ_Code/MariaTangentFlow}
            \endgroup
            \caption{The effect of the pushforward of the flow on vectors which are tangent to the flow. The representative curve of $X_x$ is the flow itself $\varphi_s^\mathcal{H}$ depicted in green. By the semigroup propriety, when we compose it with $\varphi_{-t}^\mathcal{H}$ we get the same flow, but with a time shift of $-t$ depicted in orange. Differentiating and taking the limit as $t\nsea 0$ gives $X_{\Delta^{-1}(x)}$. }
            \label{fig:proof_chikos_flow}
        \end{subfigure}
        
        \caption{Schematic drawing of the two cases in the proof of Theorem \ref{thm:restricted chikos}.}
        \label{fig:three graphs}
    \end{figure}
    \!Now, let $v = c X_x\in \text{Span}(X_x)$ as illustrated in Figure \ref{fig:proof_chikos_flow}. By the semigroup propriety of $\varphi_{-t}^\mathcal{H}$ and the definition of $\big(\Delta_*^X\big)^{-1}$, we have that 
    \begin{equation*}
        \lim_{t \nsea 0}(\varphi_{-t}^\mathcal{H})_*cX_x 
        = c X_{\Delta^{-1}(x)} %\hspace{2.5cm} \text{By the semigroup propriety of $\varphi_{-t}^\mathcal{H}$}  \\
        = \big(\Delta_*^X\big)^{-1}(cX_x)%\hspace{1.8cm} \text{By the definition of } \big(\Delta_*^X\big)^{-1}
    \end{equation*}
    Next, we take the limit of Equation \eqref{eq:proof_chikos_1} as $t \nsea 0$ and get
    \begin{align}\label{eq:split}
        % \begin{split}
            u(0^+,x)\mu
            =  \lim_{t \nsea 0} u(t,x)\mu 
            % \hspace{4.25cm} \t{For sufficiently small $t$} \\
            = \lim_{t \nsea 0} P_t^\mathcal{H} f(x)\mu
            % \hspace{4cm} \t{By the definition of $u$} \\
            =  \lim_{t \nsea 0}\big(f \circ \varphi_{-t}^\mathcal{H}\big)(x) \left(\varphi_{-t}^\mathcal{H}\right)^* \!\mu.
            % \hspace{2.1cm} \t{By Equation \eqref{eq:proof_chikos_1}} 
        % \end{split}
    \end{align}     \vspace*{-5mm}

    As in the continuous case, we want to cancel $\mu_x$ from both sides. To do so, Let $v^1, \ldots, v^n \in T_xM$ be arbitrary. Plugging $v^1, \ldots, v^n$ into  Equation \eqref{eq:split}, we see that      \vspace*{-0mm}
\begin{align}
        & u(0^+,x)\mu_x(v^1, \ldots, v^n) \nonumber \\
        & \qquad \quad = \lim_{t \nsea 0}\big(f \circ \varphi_{-t}^\mathcal{H}(x)\big) \!\cdot\!\big(\varphi_{-t}^\mathcal{H}\big)^* \mu_x(v^1, \ldots, v^n) 
        && \t{By Equation \eqref{eq:split}} \nonumber \\
        & \qquad \quad = \lim_{t \nsea 0}\big(f \circ \varphi_{-t}^\mathcal{H}(x)\big)\!\cdot\!\mu_{\varphi_{-t}^\mathcal{H}(x)} \big((\varphi_{-t}^\mathcal{H})_*v^1, \cdots, (\varphi_{-t}^\mathcal{H})_*v^n\big) 
        && \!\!\!\begin{tabular}{l}
            By the definition of \\
            the pushforward
        \end{tabular} \nonumber \\
        & \qquad \quad = f(\Delta^{-1}(x)) \!\cdot\! \mu_{\Delta^{-1}(x)} \big(\big(\Delta_*^X\big)^{-1} (v^1), \dots, \big(\Delta_*^X\big)^{-1} (v^n)\big)
        && \t{By Equation \eqref{eqn:claim}} \nonumber \\
        & \qquad \quad = \frac{f(\Delta^{-1}(x))}{\big(\mathcal{J}_\mu^X(\Delta)\big)(\Delta^{-1}(x))} \mu_x(v^1, \ldots, v^n) 
        && \t{By Lemma 3.10} \nonumber \\
        & \qquad \quad = \frac{u(0^-,\Delta^{-1}(x))}{\big(\mathcal{J}_\mu^X(\Delta) \circ \Delta^{-1}\big)(x)}\mu_x(v^1, \ldots, v^n)
        && \t{By the definition of $u$}\label{eqn:follows}
    \end{align}   
    Since  $v^1, \ldots, v^n$ are arbitrary and $\mu_x \neq 0$, we have obtained the discrete case of \eqref{eq:chikos} for $t = 0$. 
   
    Now, consider $\varphi_{-t}^\mathcal{H}(x) \in \Delta(\mathcal{S})$ for $t \neq 0$ and without loss of generality, consider $t > 0$ since the $t < 0$ case follows from applying the same steps in the opposite direction. 
    \begin{align*}
        u(t^+, x)
        = & \, P^{\mathcal{H}}_t u(0^+, x)
        && \!\!\! \begin{tabular}{l}
            By definition of $u$  \\
            and since $P_{t + 0}^\mathcal{H} = P_t^\mathcal{H}P_0^\mathcal{H}$
        \end{tabular}\\
        = & \, P^{\mathcal{H}}_t \frac{u(0^-,\Delta^{-1}(x))}{\big(\mathcal{J}_\mu^X(\Delta) \circ \Delta^{-1}\big)(x)}
        && \t{By Equation \eqref{eqn:follows}} \\
        = & \, \frac{1}{\big(\mathcal{J}_\mu^X(\Delta) \circ \Delta^{-1}\big)(x)} P^{\mathcal{H}}_t u(0^-,\Delta^{-1}(x))
        && \!\!\! \begin{tabular}{l}
            Since $\big(\mathcal{J}_\mu^X(\Delta) \circ \Delta^{-1}\big)(x) \in \R$  \\
            and $P^{\mathcal{H}}_t$ is $\R$-linear
        \end{tabular} \\
        = & \, \frac{1}{\big(\mathcal{J}_\mu^X(\Delta) \circ \Delta^{-1}\big)(x)} u(t^-,\Delta^{-1}(x))
        && \t{Since $t > 0$}
    \end{align*}
    which concludes the proof. 
\end{proof}

In general, the reset map need not be invertible. Assuming that $\Delta^{-1}$ has finitely many preimages, Theorem \ref{thm:restricted chikos} still holds. The precise statement is as follows.

\begin{theorem}[General hybrid infinitesimal generator]
    \label{thm:general chikos}
    Let $\mathcal{H}$ be a smooth hybrid dynamical system with dynamics given by \eqref{eq:HDS} and suppose $\Delta^{-1}(\set{x})$ is finite. Additionally, let $\mu \in \Omega^n(M)$ be a reference volume-form and suppose that $\mathcal{J}_\mu^X(\Delta) \ne 0$. Then, the hybrid transfer operator $u(t,x) \coloneqq P_t^\mathcal{H}f(x)$ satisfies the following
    \begin{equation}\label{eqn:general chikos}
        \begin{cases}
            \displaystyle \frac{\partial u}{\partial t} + du(X) = -u\cdot\mathrm{div}_\mu(X) & \t{for } x \notin \Delta(\mathcal{S}); \\[2ex]
            \displaystyle u(t^+,x) = \!\! \sum_{y \in \Delta^{-1}(\set{x})} \frac{1}{\big(\mathcal{J}_\mu^X(\Delta)\big)(y)} u(t^-,y) & \t{for } x \in \Delta(\mathcal{S}).
        \end{cases}
    \end{equation}
\end{theorem}
\begin{proof}As the continuous case of Theorem \ref{thm:general chikos} is identical to that of \ref{thm:restricted chikos}, it remains to show
    \begin{minipage}{.42\linewidth}
        the $x \in \Delta(\mathcal{S})$ case of \eqref{eqn:general chikos}. Doing so, let $\Delta^{-1}(\set{x}) = \set{y_1, \ldots, y_n}$ and define $E \subset M$ to be a half-neighborhood about $x \in \Delta(S)$ such that 
        \begin{align*}
            \lim_{\tau \nsea t}\varphi_{-\tau}^\mathcal{H}(E) = \bigcup_{i = 1}^n E_i,
        \end{align*}
        where each $E_i$ is a half-neighborhood about $y_i$ and $\set{E_i}$ is pairwise disjoint as in Figure \ref{fig:multiple_preimages}. Let $\varphi_{t}^i = \varphi_{t}^{\mathcal{H}}|_{E_i}$ denote the hybrid flow restricted to each ball $E_i$ so that each $\varphi_{t}^i$ is invertible. Now, we will apply the proof of \ref{thm:restricted chikos} to each $\varphi_{t}^i$ and then combine the results to account for the \hfill multiple \hfill preimages. \hfill Generalizing
    \end{minipage}\hfill
    \begin{minipage}{.56\linewidth}
        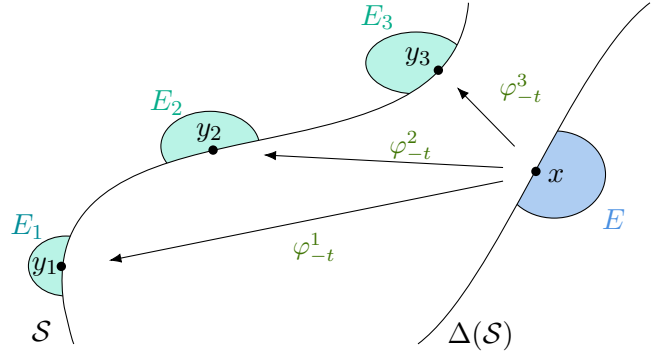
\begin{figure}[H]
            \centering
            \input{TikZ_Code/ModdedMariaGeneralHFP}
            \caption{Diagrammatic construction of $E_i$ and its corresponding $\varphi^i_{-t}$ when $n = 3$.}%\wac{BEAUTIFUL} \aps{That's what I'm sayin !!!}}
            \label{fig:multiple_preimages}
        \end{figure}
    \end{minipage}\\
    the change of coordinates in \eqref{eq:change_of_coords}, becomes a sum over the $\varphi_{-t}^i$'s as follows
    \begin{align*}
        \int_E P_t^\mathcal{H} f(x) \, d\mu 
        = \sum_{i = 1}^n\int_{\varphi_{-t}^i(E)}f(x) \, d\mu 
        = \sum_{i = 1}^n \int_{E} \big(f \circ \varphi_{-t}^i\big)(x) \left(\varphi_{-t}^i\right)^* \!\mu,
    \end{align*}
    which by similar reasoning, implies that
    \begin{equation}\label{eqn:multiple_preimages1}
        P_t^\mathcal{H}f(x) \mu = \sum_{i = 1}^n (f\circ\varphi_{-t}^i)(x) (\varphi_{-t}^i)^*\mu.
    \end{equation}
    Following the same steps that imply Equation \eqref{eqn:follows} gives us
    \begin{align}\label{eqn:final split}
        \begin{split}
            \lim_{t \nsea 0} \big(f \circ \varphi_{-t}^i\big)(x) \big(\varphi_{-t}^i\big)^* \mu_{x}(v^1,\ldots, v^n)
            = & \, \frac{f(y_i)}{\big(\mathcal{J}^X_\mu(\Delta)\big)(y_i)} \mu_x(v^1,\ldots, v^n).
        \end{split}
    \end{align} since $\Delta$ is invertible when restricted to $E_i$.
    Summing \eqref{eqn:final split} over each $y_i \in \Delta^{-1}(\set{x})$ and equating with the left-hand side of \eqref{eqn:multiple_preimages1} gives
    \begin{align*}
        u(t^+, x) 
        = \sum_{i = 1}^n \frac{1}{\big(\mathcal{J}_\mu^X(\Delta)\big)(y_i)}u(t^-, y_i) 
        = \sum_{y \in \Delta^{-1}(\set{x})} \frac{1}{\big(\mathcal{J}_\mu^X(\Delta)\big)(y)}u(t^-, y),
    \end{align*}
    which is precisely the discrete part of Equation \eqref{eqn:general chikos}.
\end{proof}
\begin{remark}[$\Delta^{-1}(\set{x})$ with infinite cardinality]
    Consider the smooth HDS 
    \begin{equation*}
        (M,\mathcal{S}, \Delta, X) = (\R, \Z \smm \set{0}, n \mapsto 0, \p/\p x),
    \end{equation*}
    which implies $\Delta(\mathcal{S}) = \set{0}$. This is an HDS that does not undergo beating or any sort of Zeno properties, but $\Delta^{-1}(\set{0}) = \Z \sm \set{0}$, which is infinite. This implies that we necessarily cannot assume that the preimage of a point in $\Delta(\mathcal{S})$ has finite cardinality. Suppose $\mu = dx$. Then,
    \begin{gather*}
        \Delta^*i_X\mu = \Delta^* dx(X) = \Delta^* 1 = 1 \circ \Delta = 1 \\
        \iota^*_\mathcal{S} i_X\mu = \iota^*_\mathcal{S} dx(X) = 1 \circ \iota_\mathcal{S} = 1
    \end{gather*}
    which implies that $\mathcal{J}_\mu^X(\Delta) = 1$. Then, 
    \begin{equation}
        % \lim_{t \nsea 0} u(t^+,0) 
        u(0^+,0) 
        = \sum_{y\in\Delta^{-1}(\set{0})} \frac{1}{\mathcal{J}_\mu^X(\Delta)(y)} u(0^-,y)
        = \sum_{y\in \Z \sm \set{0}} u(0^-,y)
    \end{equation}
    which can blowup for certain choices of $f$. 
\end{remark}

%% file: TikZ_Code/MariaPushforward.tex
\begin{tikzpicture}[x=0.75pt,y=0.75pt,yscale=-1,xscale=1]
    \useasboundingbox (152,10) rectangle (520,210);
    % ∆(S)
    \draw[thick](248,13) .. controls (209,157.1) and (503,9.1) .. (470,152.1);
    
    % S
    \draw[thick](152,99) .. controls (173,29.1) and (372,308.1) .. (425,154.1);
    %308.1
    %Curve Lines
    \draw[thick, MariaGreen!75!black, line width=1.5] (244.65,41.5) .. controls (246,71.1) and (278,84.1) .. (327,83.2);
    \draw[-stealth, thick, MariaBlue] (270,74.8) -- (207.76,44.05);
    
    %Curve Lines
    \draw[thick, color={rgb, 255:red, 245; green, 166; blue, 35}, line width=1.5] (288,166.1) .. controls (319,184.7) and (363,216.5) .. (398,193.5);
    \draw [-stealth, thick, color={rgb, 255:red, 208; green, 2; blue, 27}] (337,193.7) -- (403,225);

    \draw[->, thick] (355,94) .. controls (355,115) and (325,118) .. (309,160);

    % dots
    \node at (270,74.8) [circle,fill,inner sep=1.25pt]{};
    \node at (337,193.7) [circle,fill,inner sep=1.25pt]{};
    
    % Text Node
    \draw (166, 68.4) node [anchor=north west, inner sep=0.75pt] {$\mathcal{S}$};
    \draw (474,143.4) node [anchor=north west, inner sep=0.75pt] {$\Delta(\mathcal{S})$};
    \draw (271, 72) node [anchor=south west, inner sep=0.75pt] {$x$};
    \draw (335,175.0) node [anchor=north west, inner sep=0.75pt] {$y$};
    \draw (300,107) node [anchor=north west, inner sep=0.75pt] {$\Delta^{-1}$};
    \draw (248, 32) node [anchor=north west, inner sep=0.75pt, MariaGreen!75!black]  {$\gamma(s)$};
    \draw (254,180) node [anchor=north, inner sep=0.75pt, color={rgb, 255:red, 245; green, 166; blue, 35}] {$\Delta^{-1}(\gamma(s))$};
    % 
    % \draw (current bounding box.north east) -- (current bounding box.north west) -- (current bounding box.south west) -- (current bounding box.south east) -- cycle;
\end{tikzpicture}

%% file: TikZ_Code/MariaTangentFlow.tex
\begin{tikzpicture}[x=0.75pt,y=0.75pt,yscale=-1,xscale=1]
    
    % Delta(S)
    \draw[thick] (203,44) .. controls (308,131.1) and (497,100.1) .. (470,152.1);
    % S
    \draw[thick] (152,99) .. controls (192,69) and (370,237.1) .. (425,154.1);
    
    %Curve Lines 
    \draw [thick, color={rgb, 255:red, 245; green, 166; blue, 35}] (167,228.1) .. controls (207,198.1) and (171,140.1) .. (210,110);
    \draw [thick, MariaGreen!75!black] (335.5,101.5) .. controls (376,71) and (485,125.1) .. (512,55.1);

    \draw [thick, -Stealth, color={rgb, 255:red, 208; green, 2; blue, 27 }] (210,110) -- (239.7,82);
    \draw [thick, -Stealth, MariaBlue] (335.5,101.5) -- (387.39,63.28);
    
    % Text 
    \draw (399,180) node [anchor=north west, inner sep=0.75pt] {$\mathcal{S}$};
    \draw (474,143) node [anchor=north west, inner sep=0.75pt] {$\Delta(\mathcal{S})$};
    \draw (340,55) node [anchor=north west, inner sep=0.75pt, MariaBlue] {$X_x$};
    \draw (195, 70) node [anchor=north west, inner sep=0.75pt, color={rgb, 255:red, 208; green, 2; blue, 27}] {$X_y$};
    \draw (188,193) node [anchor=north west, inner sep=0.75pt, color={rgb, 255:red, 245; green, 166; blue, 35}]  {$\varphi_{-t + s}^{\mathcal{H}}(x)$};
    \draw (416, 64) node [anchor=north west, inner sep=0.75pt, MariaGreen!75!black] {$\varphi_{s}^{\mathcal{H}}(x)$};
    
    \draw (335,106) node [anchor=north west, inner sep=0.75pt] {$x$};
    \draw (207,116) node [anchor=north west, inner sep=0.75pt] {$y$};
    
    \node at (210,110) [circle,fill,inner sep=1.25pt]{};
    \node at (335.5,101.5) [circle,fill,inner sep=1.25pt]{};
    % \draw (current bounding box.north east) -- (current bounding box.north west) -- (current bounding box.south west) -- (current bounding box.south east) -- cycle;
\end{tikzpicture}

%% file: TikZ_Code/ModdedMariaGeneralHFP.tex
\begin{tikzpicture}[x=0.75pt,y=0.75pt,yscale=-.975,xscale=.975]
    
    % E and E_i sets for i = 1,2,3. {rgb, 255:red, 74; green, 144; blue, 226}
    \draw [fill={rgb, 255:red, 74; green, 144; blue, 226}, fill opacity=0.45] (304,117.6) .. controls (304,105.17) and (315.86,95.1) .. (330.5,95.1) .. controls (345.14,95.1) and (357,105.17) .. (357,117.6) .. controls (357,130.03) and (345.14,140.1) .. (330.5,140.1) .. controls (315.86,140.1) and (304,130.03) .. (304,117.6) -- cycle;
    \draw [fill={rgb, 255:red, 80; green, 227; blue, 194}, fill opacity=0.4] (233,60.05) .. controls (233,51.24) and (244.64,44.1) .. (259,44.1) .. controls (273.36,44.1) and (285,51.24) .. (285,60.05) .. controls (285,68.86) and (273.36,76) .. (259,76) .. controls (244.64,76) and (233,68.86) .. (233,60.05) -- cycle;
    \draw [fill={rgb, 255:red, 80; green, 227; blue, 194}, fill opacity=0.4] (58.48,164.8) .. controls (58.48,156.27) and (67.21,149.35) .. (77.98,149.35) .. controls (88.75,149.35) and (97.48,156.27) .. (97.48,164.8) .. controls (97.48,173.33) and (88.75,180.25) .. (77.98,180.25) .. controls (67.21,180.25) and (58.48,173.33) .. (58.48,164.8) -- cycle;
    \draw [fill={rgb, 255:red, 80; green, 227; blue, 194}, fill opacity=0.4] (127,103.05) .. controls (127,93.08) and (138.42,85) .. (152.5,85) .. controls (166.58,85) and (178,93.08) .. (178,103.05) .. controls (178,113.02) and (166.58,121.1) .. (152.5,121.1) .. controls (138.42,121.1) and (127,113.02) .. (127,103.05) -- cycle;
    
    % S
    \fill[white](82,205) .. controls (81,204) and (77.25,187) .. (77,186) .. controls (75.77,178.23) and (75.56,171.17) .. (76.25,164.8) .. controls (85.67,78.1) and (261.64,121.22) .. (284,41) .. controls (284.5,40) and (286.5,29.5) .. (286,29) -- (325,29) -- (260,90) -- (180,120) -- (110,150) -- cycle;
    \draw (82,205) .. controls (81,204) and (77.25,187) .. (77,186) .. controls (75.77,178.23) and (75.56,171.17) .. (76.25,164.8) .. controls (85.67,78.1) and (261.64,121.22) .. (284,41) .. controls (284.5,40) and (286.5,29.5) .. (286,29);

    % Delta(S)
    \fill[white] (260,205) .. controls (300,175.1) and (340,59.1) .. (380,29) -- (360,29) -- (235,205) -- cycle;
    \draw (260,205) .. controls (300,175.1) and (340,59.1) .. (380,29);

    % Arrows Top to bottom
    \draw[-Latex] (310,103.1) -- (280,72.5);
    \draw[-Latex] (304,114) -- (180,107.5);
    \draw[-Latex] (304,121) -- (100,162);
    
    % x and y_i points for i = 1,2,3.
    \node at (321,116) [circle,fill,inner sep=1.25pt]{};
    \node at (75.5,165) [circle,fill,inner sep=1.25pt]{};
    \node at (154,105) [circle,fill,inner sep=1.25pt]{};
    \node at (270.53,63.8) [circle,fill,inner sep=1.25pt]{};
    
    % Text
    \draw (55,188)  node [anchor=north west] {$\mathcal{S}$};
    \draw (270,188) node [anchor=north west] {$\Delta(\mathcal{S})$};
    
    \draw (350,130) node [anchor=north west, color={rgb, 255:red, 74; green, 144; blue, 226}] {$E$};
    \draw (44,133)  node [anchor=north west, color={rgb, 255:red,  2; green, 141; blue, 153}] {$E_1$};
    \draw (116,69)  node [anchor=north west, color={rgb, 255:red,  2; green, 141; blue, 153}] {$\textcolor[rgb]{0.02,0.68,0.54}{E_2}$};
    \draw (225,25)  node [anchor=north west, color={rgb, 255:red,  2; green, 141; blue, 153}] {$\textcolor[rgb]{0.02,0.68,0.54}{E_3}$};

    \draw (322,110) node [anchor=north west] {$x$};
    \draw (55,155)  node [anchor=north west] {$y_1$};
    \draw (140,86)  node [anchor=north west] {$y_2$};
    \draw (248,47)  node [anchor=north west] {$y_3$};
    
    \draw (190,142.5) node [anchor=north west, font=\small] {$\textcolor[rgb]{0.25,0.46,0.02}{\varphi_{-t}^1}$};
    \draw (240,90)    node [anchor=north west, font=\small] {$\textcolor[rgb]{0.25,0.46,0.02}{\varphi_{-t}^2}$};
    \draw (297,60.4)  node [anchor=north west, font=\small] {$\textcolor[rgb]{0.25,0.46,0.02}{\varphi_{-t}^3}$};
\end{tikzpicture} 

%% file: Results/Reduction/intro.tex
The primary goal of this work is to understand and compute the Frobenius-Perron operator for mechanical impact systems which results in solving \eqref{eq:chikos}. Unfortunately, the dimension in many examples makes direct numerical study difficult. The goal of this section is to extend the idea of Lie-Poisson reduction (Theorems \ref{thm:cont reduction} and its nonholonomic counterpart) %\ref{thm: NH Lie-Poisson}) 
to impact systems. 
Other versions of reduction have been carried out for impact systems, e.g. \cite{ames2006,leo_reduction,EYREAIRAZU202194} and the references therein.
Classical Lie-Poisson reduction is capable of reducing the dimension from $2n$ down to $n$. When impacts are considered, it will be shown that the dimension can only be reduced to $n+1$.

%% file: Results/Reduction/impact_reduction_theorems.tex
\label{sec:reduction}

Lie-Poisson reduction is possible when the Hamiltonian is left-invariant as the dynamics can be translated to the identity element. The natural extension to impact systems is for the impact surface, $\Sigma \subset G$, to be left-invariant as well. As will be shown below, this is true precisely when $\Sigma$ is the \textit{right} coset of a normal subgroup.

Before we proceed, we provide a summary of the notation used in this section. 
\begin{table}[!ht]
    \centering
    \begin{tabular}{c|c|c}
        Term & Notation & Element(s) \\\hline
        Lie Group & $G$ & $g, g_0$ \\
        Lie Subgroup & $K$ & $k,\tilde{k}$ \\
        Lie Algebra & $\mf{g}$ & $\xi$ \\
        Lie Algebra of $K$ & $\mf{K}$ & $\delta k$, $\widetilde{\delta k}$ \\
        Dual Lie Algebra & $\ \mf{g}^*$ & $\zeta$ \\
        Impact Surface & $\Sigma = Kg_0$ & $kg_0$ \\
        Metric on $G$ & $\mTse$ & N/A
    \end{tabular}
    \caption{Notation for Section \ref{sec:reduction}}
    \label{tab:Table of Notation}
\end{table}

Throughout this section we consider mechanical impact systems as defined in \ref{def:Impact System}, which, by definition, means that the natural Hamiltonian can be written as 
\begin{equation*}
    H(q, p) = \frac{1}{2}\mTse^{-1} (p, p) + V(q).
\end{equation*}

\begin{remark}\label{remark:musical_eta_xi}
    We can go between $G$'s Lie algebra $\mathfrak{g}$ and its dual $\mathfrak{g}^*$ using the musical isomorphisms. In particular $\xi \coloneqq \zeta^\sharp = (\mTse^{-1})^{ij}\zeta_j$.
    % \wac{huh?}\aps{ah, ``huh'' is a good point.}
\end{remark}

Assume now that we are in the setting of Lie-Poisson reduction, i.e., $H$ is left invariant and the state space is a Lie group $G$. Then the full equations of motion for the mechanical impact system $(G, H, \Sigma)$ are given by the following Proposition.
\begin{proposition}\label{prop:eom for Lie algebras}
    Let $H \colon T^*G \to \mathbb{R}$ be a natural, left-invariant Hamiltonian and $h \colon \mathfrak{g}^* \to \mathbb{R}$ be its restriction to the identity. For an impact surface, $\Sigma \subset G$, let $s \colon G \to \mathbb{R}$ be such that zero is a regular value and (locally) $\Sigma = s^{-1}(\{0\})$. Then, between impacts, the usual Lie-Poisson equations hold
    \begin{equation}\label{eq:LP cont}
        \begin{cases}
            \dot{\zeta} = \mathrm{ad}^*_{dh}\zeta; \\
            \dot{g} = \left(\ell_g\right)_*\zeta^\sharp,
        \end{cases} \quad \t{for }s(g)\ne 0,
    \end{equation}
    and upon impact, the left-translation of the corner conditions \eqref{Weierstrass–Erdmann corner conditions} becomes
    \begin{equation}\label{eq:LP impact}
        \begin{cases}
            \zeta\mapsto \zeta - 2\dfrac{\big(\ell_g\big)^*ds_g\big(\zeta^\sharp\big)}{\ell_g^* ds_g\big((\ell_{g^{-1}})_*ds_g^\sharp\big)}\big(\ell_g\big)^*ds_g; \\
            g\mapsto g,
        \end{cases} \quad \t{for } s(g) = 0,
    \end{equation}
    where $\sharp \colon \Omega^1(G) \to \mathfrak{X}(G)$ is the musical isomorphism, as defined in \ref{def:musical_isomorphisms}.% induced by the metric from the Hamiltonian as in Remark \ref{remark:musical_eta_xi}
\end{proposition}

\begin{proof}
    For an arbitrary manifold $Q$ with natural Hamiltonian and non-moving impact surface, the variational impact equations are given by \cite{hybrid_forms}:
    \begin{equation}\label{general_impact_eqn}
        \begin{cases}  
            \ds p^+ = p^- - \frac{2p^-(\nabla s_q)}{ds_q(\nabla s_q)}ds_q;  \\
            \ds q^+ = q^-,
        \end{cases} \qquad \text{for } q \in \Sigma,
    \end{equation}
    where $\nabla s_q = (ds_q)^\sharp$. By \cite{marsden_ratiu}, the system will follow the continuous Lie-Poisson equations of motion as follows when away from impacts.
    \begin{align*}
        \begin{cases}
            \dot{\zeta} = \mathrm{ad}^*_{dh}\zeta; \\
            \dot{g} = (\ell_{g^{-1}})_* \zeta^\sharp.
        \end{cases}
    \end{align*}
    In order to get the impact equations on the dual of the Lie algebra we need to left translate \eqref{general_impact_eqn} to the origin. For $g \in G$,
    \begin{align*}
        p^+ 
        = & \ p^- - \frac{2p^-(\nabla s_g)}{ds_g(\nabla s_g)}ds_g 
        && \t{Starting with \eqref{general_impact_eqn}} \\
        \iff (\ell_g)^* p^+ 
        = & \ (\ell_g)^*p^- -(\ell_g)^* \frac{2p^-(\nabla s_g)}{ds_g(\nabla s_g)}ds_g 
        && \t{After left translating} \\ 
        \iff \zeta^+ 
        = & \ \zeta^- - \frac{2\zeta^-((\ell_{g^{-1}})_*\nabla s_g)}{\ell_g^* ds_g\big((\ell_{g^{-1}})_*\nabla s_g\big)}(\ell_g)^*d s_g 
        && \t{By the definition of $\zeta$}
    \end{align*}
    where $\zeta^+ = (\ell_g)^*p^+$ and $\zeta^- = (\ell_g)^*p^-$ are elements of $\mathfrak{g}^*$. The last step is to note that 
    \begin{align*}
        \zeta (\ell_{g^{-1}})_*\nabla s_g 
        = & \ g^{-1}\cdot \zeta_i (\mTse^{-1}){ij}(ds_g)_{j} 
        && \t{Expanding in local coordinate representaion} \\
        = & \ (ds_g)_{j}g^{-1}\cdot (\mTse^{-1})^{ij}\zeta_i 
        && \t{Since $(ds_g)_{j} \in \R$ for each $j$} \\
        = & \ ds_g (\ell_{g^{-1}})_* \zeta^\sharp 
        && \t{Contracting indices} \\
        = & \ \ell_{g^{-1}}^* ds_g (\zeta^\sharp)
        && \t{By the definition of the pullback} 
    \end{align*}
\end{proof}

The goal now is to provide an analog of continuous Lie-Poisson reduction \ref{thm:cont reduction} for the hybrid equations of motion presented above \ref{prop:eom for Lie algebras}. This is not straightforward due to the reasons presented in the following: 
 
\begin{remark}\label{ref:HLP_issues}
    The power of the Lie-Poisson equations \eqref{eq:Lie_poisson_cont} is that it decouples the momentum from the position dynamics. This fails to be the case for the impact case for two main reasons: 1) the impact map for $\zeta$ in \eqref{eq:LP impact} depends on $g$ as well as $\zeta$, and 2) switching between \eqref{eq:LP cont} and \eqref{eq:LP impact} depends on $g$ (i.e. on the position on the impact surface). As will be shown, by requiring the impact set to be left-invariant, the first issue can be resolved. The second issue requires more structure i.e. that of a right coset of a normal subgroup.
\end{remark}

The overarching idea is the following: we want to be able to detect when the impact occurs without having to keep track of the entire $n$ dimensional trajectory in $G$. We try to do so in the least number of dimensions possible. One crucial piece of information is \textit{whether} the impact is happening or not. Hence, the question to be asked is: what propriety does $\Sigma$ need to have so that \textit{where} the impact happens does not affect the reduced hybrid equations of motion? From the following definitions and lemmas it will become apparent that $\Sigma$ needs to be a right coset.

\begin{minipage}{0.45\textwidth}
    \begin{definition}[Impact $u$-stabilizer]\label{Gs}\\
    Given some $u \in \Sigma \subseteq G$, define the \\ impact $u-$ stabilizer to be the set
    \begin{equation*}
        G_u(\Sigma) = \{g \in G: gu \in \Sigma\}.
    \end{equation*}

\end{definition}
\end{minipage}
\begin{minipage}{0.45\textwidth}
    \begin{figure}[H]
    \centering
    \input{TikZ_Code/Invariant_Set}
    \caption{Depiction of $\Sigma$ elements remaining in $\Sigma$ after multiplication by $G_u(\Sigma)$ elements.}
    \label{fig:Gs}
\end{figure}
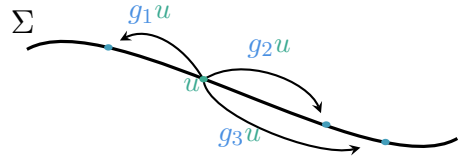
\end{minipage}

In other words $G_u(\Sigma)$, consists of all group elements that, when acting on $u$, keep $\Sigma$ invariant. 
A special property holds when $\Sigma$ is a right coset.
\begin{lemma}\label{prop:Gs_is_H}
    Let $K$ be a subgroup of $G$, let $\Sigma = Kg_0$ for some $g_0 \in G$, and define $G_u(\Sigma)$ as in Definition \ref{Gs}. Then, $G_u(\Sigma) = K$.
\end{lemma}
\begin{proof}
    We will show $G_u(\Sigma) \subseteq K$ and then $K \subseteq G_u(\Sigma)$. Let $g \in G_u(\Sigma)$ be arbitrary so $gu \in \Sigma$ by definition. Since $\Sigma = Kg_0$, there exists a $k \in K$ such that $gu = kg_0$, which implies $g = kg_0u^{\!-1}$. Since $u \in \Sigma = Kg_0$, we know that $u = \tilde{k}g_0$ for some $\tilde{k} \in K$. Substituting this in, $  g = kg_0u^{\!-1} = kg_0(\tilde{k}g_0)^{-1} = kg_0g_0^{-1}\tilde{k}^{-1} = k\tilde{k}^{-1}.$ Since $K \le G$, we know $g = k\tilde{k}^{-1} \in K$, so $G_u(\Sigma) \subseteq K$. Now, let $k \in K$ be arbitrary. Using similar arguments as in the first part, $ku = k\tilde{k}g_0$ for some $\tilde{k} \in K$, and since $K \le G$, we know that $k\tilde{k} \in K$, which implies $ku \in Kg_0 = \Sigma$. Thus, $G_u(\Sigma) = K$.
\end{proof} \vspace*{-2mm}

In order to reduce the dimension to $n + 1$, it should not matter whether the impact happens at $u$ or at another point in $G_u(\Sigma)$. The previous lemma implies that for right cosets $G_u(\Sigma) = K$ for any $u$. Assuming that points in $G_u(\Sigma)$ are equivalent for the reset map, this implies that for right cosets, the only thing we would need to keep track of is the direction normal to $K$. This is indeed one dimensional if $K$ is codimension 1. Now we need to see when the equivalence assumption actually holds. For this we define the tangent preserving propriety.\\
\begin{minipage}{.5\linewidth}
    \begin{center}
        \begin{minipage}{.76\linewidth}
            \begin{definition}[Tangent preserving]\label{Tangent Preserving}
                \!\!\!\!Let $\Sigma$ be a submanifold of $G$ (not necessarily a subgroup). Then, $\Sigma$ is said to be (left) tangent preserving if for $u \in \Sigma$ and $g \in G_u(\Sigma)$, then 
                \begin{align*}
                    (\ell_g)_* T_u\Sigma = T_{gu}\Sigma.
                \end{align*}
            \end{definition}
        \end{minipage}
    \end{center}
\end{minipage}\hfill
\begin{minipage}{.5\linewidth}
    \begin{figure}[H]
    \centering
    \input{TikZ_Code/Tangent_Preserving}
    \caption{Geometric depiction of the tangent preserving propriety, where $\Sigma$ has the tangent preserving propriety and $\Sigma'$ does not, so $(\ell_g)_*v \in T\Sigma$ and $(\ell_g)_*v \notin T\Sigma'$.}
    \label{fig:tangent preserving comparison}
\end{figure}
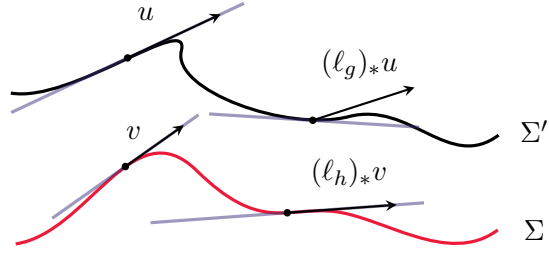
\end{minipage}

\begin{lemma}\label{lemma:cosetProp}
    If $\Sigma = Kg_0$ for a subgroup $K \leq G$, then $\Sigma$ is (left) tangent preserving.
\end{lemma}
\begin{proof}
    Let $g \in G_u = K$ be arbitrary. By Lemma \ref{prop:Gs_is_H}, we have $g\Sigma = gKg_0 = Kg_0 = \Sigma$, so $\ell_g \colon \Sigma \to \Sigma$ is a diffeomorphism, which implies that $(\ell_g)_* T_u\Sigma = T_{gu}\Sigma$. Thus, $\Sigma$ is tangent preserving.
\end{proof}

The tangent preserving propriety tells us that we can left translate a vector along $\Sigma$, and still stay in the tangent space of $\Sigma$. See Figure \ref{fig:tangent preserving comparison} for a schematic drawing. 
\begin{lemma}\label{lemma:left_invariant_surface}
    Assume, as before, that $K \leq G$ is a subgroup and let $\Sigma = Kg_0$ for some $g_0 \in G$. If there exists a smooth function $s \in C^\infty(G)$ such that $s^{-1}(\set{0}) = \Sigma$, then for all $u \in S$, $g \in G_u(\Sigma)$, there exists an $\alpha_{u,g} \in \R$ such that $\alpha_{u, g} ds_u = (\ell_g)^*ds_{ug}$.
\end{lemma}
\begin{proof}
    Since $\Sigma = Kg_0$, we use Lemma \ref{Tangent Preserving} states that $(\ell_g)_*T_u\Sigma = T_{gu}\Sigma$ for any $g \in K$. By the definition of $s \in C^\infty(G)$, we know $T_u\Sigma = \ker ds_u$ for all $u \in \Sigma$. Then applying the tangent preserving property, 
    \vspace*{-5mm}
    \begin{equation*}
        \ker ds_u = \big(\ell_{g^{-1}}\big)_*\ker ds_{gu} = \ker((\ell_g)^*ds_{gu}).
    \end{equation*}
    As $T_u\Sigma = \ker ds_u = \ker((\ell_g)^*ds_{gu})\subset T_uG$ is a codimension 1 subspace, it follows that $ds_u$ and $(\ell_g)^*ds_{gu}$ are linearly dependent.
    Since $T_u\Sigma = \ker ds_u = \ker((\ell_g)^* ds_{gu})$ is an codimension 1 vector field, it follows that the range of both $ds_u$ and $ds_{gu}$ is one dimensional and thus $ds_u$ and $ds_{gu}$ must differ by a constant i.e. there exists an $\alpha_{u, g} \in \mathbb{R}$ such that $\alpha_{u, g}  ds_u = ds_{ug}$.
\end{proof}

Hence, if the tangent preserving property is fulfilled, the only information that we lose by translating vectors along $\Sigma$ is contained within a constant.

%% file: TikZ_Code/Invariant_Set.tex
\begin{tikzpicture}[x=1pt,y=0.75pt,xscale=1,yscale=-1, scale=0.8]

% Sigma
\draw[very thick,] (94.5, 74.75) .. controls (134.5, 44.75) and (258, 160.75) .. (298, 130.75);

% Multiplication arrows
\draw[-stealth, thick] (178, 93.75) .. controls (167.87, 70.59) and (151.68, 60) .. (138.43, 69.17);
\draw[-stealth, thick] (178, 93.75) .. controls (196.43, 77.74) and (227.56, 94.67) .. (233.52, 114.87);
\draw[-stealth, thick] (178, 93.75) .. controls (178.97, 115.58) and (231.69, 143.05) .. (252.19, 136.46);

\fill[MariaGreen2] (178, 93.5) circle [radius=2];

\fill[MariaGreen2!50!MariaBlue] (133, 73.5) circle [radius=2];
\fill[MariaGreen2!50!MariaBlue] (236, 122) circle [radius=2];
\fill[MariaGreen2!50!MariaBlue] (264, 133) circle [radius=2];

% Text Node
\draw (92.5, 57.5) node [anchor = center, inner sep=0.75pt] {\resizebox{!}{3mm}{$\Sigma$}};
\draw (172.5, 97) node [anchor = center, inner sep=0.75pt, MariaGreen2] {$u$};
\draw (152.5, 53) node [anchor = center, inner sep=0.75pt] {${\color{MariaBlue}g_1}{\color{MariaGreen2}u}$};
\draw (209, 74) node [anchor = center, inner sep=0.75pt] {${\color{MariaBlue}g_2} {\color{MariaGreen2}u}$};
\draw (195, 130) node [anchor = center, inner sep=0.75pt] {${\color{MariaBlue}g_3} {\color{MariaGreen2}u}$};
\end{tikzpicture}

%% file: TikZ_Code/Tangent_Preserving.tex
\begin{tikzpicture}[x=1pt,y=1pt, yscale=-1,xscale=1.5]
    \draw[very thick,red1] (25.22,136.78) .. controls (43.41,132.83) and (54.78,85.53) .. (70.78,109.03) .. controls (86.78,132.53) and (93.13,122.8) .. (105.28,124.53) .. controls (117.43,126.26) and (134.28,145.53) .. (146.78,131.53);
    \draw[very thick] (24,79.88) .. controls (39.25,83.88) and (69.5,48.89) .. (67,63.89) .. controls (64.5,78.89) and (93.75,94.75) .. (109.25,88.75) .. controls (124.75,82.75) and (134.5,109.75) .. (147,95.75);

    \draw[very thick, blue2, opacity = .5] (24.09,87.23) -- (82.79,46.03);
    \draw[very thick, blue2, opacity = .5] (74,87.75) -- (127,92.75);
    \draw[very thick, blue2, opacity = .5] (34.4425,127.11) -- (71.3175,88.11);
    \draw[very thick, blue2, opacity = .5] (59.235,128.79) -- (128.185,121.29);
    
    \draw[thick, -stealth] (53.44,66.63) -- (76.92,50.15);
    \draw[thick, -stealth] (100.13,90.16) -- (125.8,77.77);
    \draw[thick, -stealth] (52.88,107.61) -- (67.63,92.01);
    \draw[thick, -stealth] (93.71,125.04) -- (121.29,122.04);
    
    \draw (156,131.53) node {$\Sigma$};
    \draw (156,94.25)  node {$\Sigma'$};
    \draw (55,95)      node {$v$};
    \draw (58,50)      node {$u$};
    \draw (109.26,110) node {$(\ell_h)_* v$};
    \draw (111.93,69)  node {$(\ell_g)_* u$};
    
    \node at (53.440,66.63) [circle, fill, inner sep = 1pt]{};
    \node at (100.13,90.16) [circle, fill, inner sep = 1pt]{};
    \node at (52.88,107.61) [circle, fill, inner sep = 1pt]{};
    \node at (93.71,125.04) [circle, fill, inner sep = 1pt]{};
    
\end{tikzpicture}

%% file: Results/Reduction/impact_reduction_proofs.tex
With the help of the previous lemmas, we can address the issues in Remark \ref{ref:HLP_issues} in the theorem below. 
\begin{theorem}[Impact Lie-Poisson Reduction]\label{thm:Impact_Reduction}
    Let $G$ be a Lie group, $g_0 \in G$ be arbitrary, and $\Sigma = Kg_0 \subseteq G$ be a right coset of codimension 1 for a closed, normal, codimension 1 Lie subgroup $K$ of $G$. Denote the natural projection map by $\pi \colon G\to K \backslash G$, and the Lie algebra of $K$ by $\mathfrak{K}$.

    Let $H \colon T^*G \to \mathbb{R}$ be a natural left-invariant Hamiltonian and let $h = H|_{\mathfrak{g}^*}$ be its restriction to the identity.
    Suppose $(g(t), p(t))$ follows the hybrid flow $\varphi_t^\mathcal{H}$ and let $\zeta(t) = (\ell_{g(t)})^*p(t)$.
    Let $\sigma \colon K \backslash G\to G$ be a local section, let $q \in K \backslash G$, and let $\Delta\zeta \in \Ann(\mathfrak{K})$ be such that $h(\zeta) = h(\zeta + \Delta\zeta)$.
    Then, the equations of motion can be written as
    \begin{gather}
        \begin{cases}
            \dot{\zeta} = \ad_{dh}^*\zeta, \\
            \dot{q} = d\pi_{\sigma(q)} \big(\ell_{\sigma(q)}\big)_*\zeta^\sharp,
        \end{cases} \quad q \not\in \pi(\Sigma),\label{eq:ILP_cont}\\
        \begin{cases}
            \zeta \mapsto \zeta + \Delta\zeta, \\
            q\mapsto q,
        \end{cases}\qquad\qquad\ q \in \pi(\Sigma),\label{eq:LP_impact}
    \end{gather}
\end{theorem}

\begin{remark}
    In classical Lie-Poisson reduction, the space can be reduced from $T^*G$ to $\mathfrak{g}^*$. In the hybrid case, the reduction stops at $\mathfrak{g}^* \times \left(K \backslash G\right)$. This extra term is used to determine whether or not an impact occurs.
\end{remark}

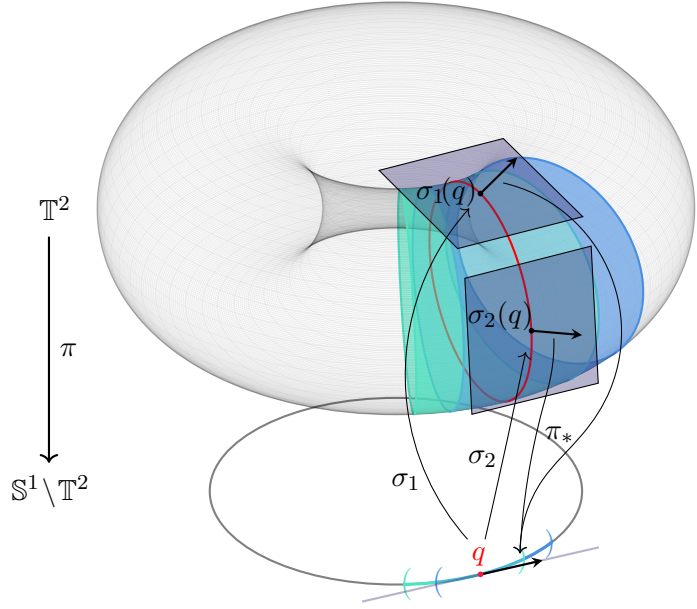
\begin{figure}
    \centering
    \input{TikZ_Code/Torus_Super_Cool}
    \caption{Depiction of independence of local section as guaranteed by \ref{thm:Impact_Reduction}. The lift of $q = [g_0]$ in the case of $\mathbb{S}^1 \subset \mathbb{T}^2$ is the circle depicted in red. The tangent vectors to any section project down to the same $\dot{q}$ in the reduced space $\mathbb{S}^1\backslash \mathbb{T}^2.$
    }
    \label{fig:independence of local section}
\end{figure}
\begin{proof}
    What needs to be shown is that 1) the dynamics on $q \in K \backslash G$ are well-defined, i.e. the choice of section does not matter and 2) the reduced hybrid equations \eqref{eq:ILP_cont} and \eqref{eq:LP_impact} do not depend on $g\in G$. 
    
    We begin with the first part. There are two components: the discrete and the continuous equations. For the continuous part, the equations of motion are given as in Proposition \ref{prop:eom for Lie algebras}
    \begin{align*}
        \begin{cases}
            \dot{\zeta} = \ad^*_{dh}\zeta; \\
            \dot{g} = (\ell_{g})_* \zeta^\sharp,
        \end{cases} 
        \text{ if } g \notin \Sigma.
    \end{align*}

    As $K \leq G$ is a subgroup that is closed in the topology of $G$, there exists a unique manifold structure such that $\pi \colon G \to K \sm G$ is a smooth surjective submersion \cite{tu2017differential}. Therefore, for a smooth curve $g(t) \in G$, we have that $q(t) \coloneqq \pi(g(t)) \in K \sm G$ is a smooth curve of equivalence classes. Additionally, as $K \leq G$ has codimension 1, the quotient manifold is a 1-dimensional manifold.

    For a curve in $g(t)\in G$ satisfying $\dot{g} = \left(\ell_g\right)_*\zeta^\sharp$, the projected curve, $q = \pi(g)$, satisfies
    \begin{equation}\label{eq: q_dynamics}
        \frac{d}{dt} q(t) 
        = \frac{d}{dt} \pi\left(g(t)\right) 
        = d\pi_{g(t)} \left( \frac{d}{dt}g(t) \right) 
        = d\pi_{g(t)} \left( \ell_{g(t)}\right)_*\zeta^\sharp.
    \end{equation}
    For all $k\in K$ and $g\in G$, $\pi(kg) = \pi(g)$, so differentiating $\pi(g)$ yields
    \begin{equation*}
        \left(\ell_k\right)^* d\pi_{kg} = d\pi_g.
    \end{equation*}
    Let $\sigma \colon K \backslash G \to G$ be a smooth (local) section. Choosing a local section is equivalent to choosing a representative $g_0(t)$ of the equivalence class $q(t)$. Any curve $g(t)\in G$ can be written as $g(t) = k(t)\sigma(\pi(g(t)))$ for some curve $k(t)\in K$ as in Figure \ref{fig:LP proof}. Calling $g_0(t) \coloneqq \sigma(\pi(g(t)))$, the dynamics on $q = \pi(g)$ in equation \eqref{eq: q_dynamics} are
    \begin{equation*}
        \dot{q} 
        = d\pi_{k(t)g_0(t)}\big( \ell_{k(t)g_0(t)}\big)_*\zeta^\sharp 
        = \left(\ell_{k(t)}\right)^* d\pi_{k(t)g_0(t)} \left(\ell_{g_0(t)}\right)_* \zeta^\sharp 
        = d\pi_{g_0(t)}\left(\ell_{g_0(t)}\right)_*\zeta^\sharp.
    \end{equation*}
    Note that the previous equation does not depend on $g(t)$ anymore, but only on $q(t)$ and on the choice of local section $\sigma$. 
    \begin{figure}[!ht]
        \centering
        \input{TikZ_Code/MariaLPO}
        \caption{Schematic drawing of Impact Lie-Poisson reduction, where $g(t)$ represents the true dynamics of the system, which differs from $g_0(t)$ by some subgroup element $k(t)$ for every $t$. Since $\Sigma$ projects down to one point $[g_0] \in K\backslash G$,it is sufficient to keep track of $q(t)$ in order to figure out when impacts happen.}
        \label{fig:LP proof}
    \end{figure}
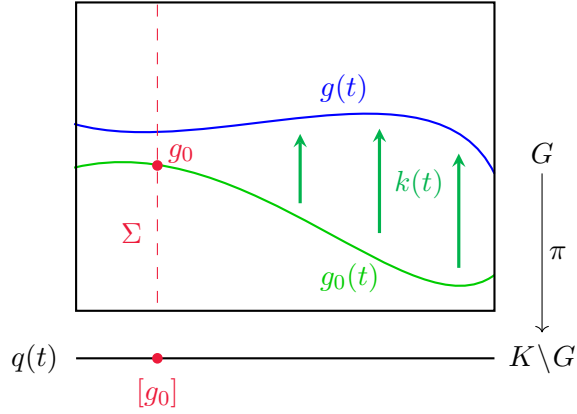
% \vspace*{-5mm}
    Now, suppose we choose a different section $\tilde{g}_0(t) = \tilde{\sigma}(\pi(g(t)))$. Then, there exists $\tilde{k}(t)\in K$ such that $\tilde{g}_0(t) = \tilde{k}(t)g_0(t)$ (see Figure \ref{fig:independence of local section}). This is because both $g_0$ and $\tilde{g}_0$ are lifts of $q(t)$. Comparing these two different representatives yields
    \begin{equation*}
        \begin{split}
            d\pi_{\tilde{g}_0(t)}\left(\ell_{\tilde{g}_0(t)}\right)_* \zeta^\sharp 
            = & \, d\pi_{\tilde{k}(t)g_0(t)}\big(\ell_{\tilde{k}(t)g_0(t)}\big)_* \zeta^\sharp \\
            = & \, \big(\ell_{\tilde{k}(t)}\big)^* d\pi_{g_0(t)}\left(\ell_{g_0(t)}\right)_* \zeta^\sharp \\
            = & \, d\pi_{g_0(t)}\left(\ell_{g_0(t)}\right)_*\zeta^\sharp.
        \end{split}
    \end{equation*}
    where the last equality comes from the fact that $\pi(kg) = \pi(g)$ for all $k \in K$ and $g \in G$.
    Therefore, the equations of motion are independent on the choice of section.

    In order to find the impact map we use the corner conditions \eqref{Weierstrass–Erdmann corner conditions} and adapt them for the case of Lie groups. 

    Since the impact surface is time independent $\delta t = \mathbb{R}$ and we get the energy conservation condition: $H^+ = H^-$. In particular, $H^+|_{\mathfrak{g}} = H^-|_{\mathfrak{g}}$ and $h^+|_{\mathfrak{g}} = h^-|_{\mathfrak{g}}$.
     For the momentum equation, we left translate to the identity
    \begin{equation*}
        ((\ell_g)^* (p^+ - p^-))((\ell_{g^{-1}})\delta g) = 0 \implies \left( \zeta^+ - \zeta^- \right)(\ell_{g^{-1}})_*\delta g = 0
    \end{equation*}
    where  $\zeta^\pm := \ell_g^*p^\pm$.

    From the corner conditions we know that $\delta g$ has to lie in $T_g\Sigma$ and since $\Sigma$ is a right coset $\delta g = (r_{g_0})_* \dot{k}$ for $\dot{k} \in T_kK$. Therefore, 
    \begin{equation*}
        (\ell_{g^{-1}})_*\delta g = (\ell_{g_0^{-1}})_*(\ell_{k^{-1}})_*(r_{g_0})_* \dot{k} = (\t{Ad}_{g_0^{-1}})_* (\ell_{k^{-1}})_* \dot{k}.
    \end{equation*}
    Since $K $ is a subgroup $(\ell_{k^{-1}})_*\dot{k} = \widetilde{\delta k} \in \mathfrak{K}$. Additionally, $K$ is normal, so its lie algebra is closed under the adjoint, i.e., $\mathrm{Ad}_g\xi\in \mathfrak{K}$ for any $g\in G$ and $\xi\in\mathfrak{K}$.
    In particular this holds for $g = g_0$ and for $\xi = \widetilde{\delta k}$ so $(\t{Ad}_{g_0^{-1}}) (\ell_{k^{-1}})_* \dot{k} \coloneqq \delta k \in \mathfrak{K}$. Since $\delta g$ can be \textit{any} vector in $T_g\Sigma$, $\delta k $ spans $\mathfrak{K}$. 
    
    For the $q$ variable $g^+ = g^-$ implies that $\pi(g^+) = \pi(g^-)$, so $q^+ = q^-$. Hence we obtain the impact map
    \begin{align*}
        \begin{cases}
            \zeta^+ = \zeta^- + \Delta \zeta\\
            q^+ = q^-
        \end{cases} \text{for all } \Delta\zeta \in \Ann(\mathfrak{K}) \text{ such that }h(\zeta^+) = h(\zeta^-).
    \end{align*}
 
    The impact equation on $K \backslash G$ are identity, independent of the local section chosen for the continuous part. Moreover, the trajectory in $q(t)$ is continuous, as shown in Figure \ref{fig:reduction scheme}. Impacts occur when $g(t)g_0(t)^{-1} \in K$ i.e. when the projection $q(t)$ passes through the equivalence class $[g_0]$. This process is also independent on $g(t)$.
\end{proof}
\begin{remark}
If the Hamiltonian is natural, $\zeta^\sharp$ is defined such that $\mTse(\zeta^\sharp, \, \cdot \, ) = \zeta$. However, $\zeta^\sharp$ is well-defined even if the Hamiltonian is not natural as long as the fiber derivative is invertible. The formula is given by $\zeta^\sharp \coloneqq \mathbb{F}L^{-1}(\zeta)$. 
\end{remark}

\begin{remark}
     The proof requires the existence of a \textit{local} section, and the choice of the section $\sigma:K\backslash G\to G$ does not matter. As such, no global section needs to exist.
\end{remark}

In the following we give an explicit expression for $\Delta \zeta$.

\begin{corollary}\label{cor:4.10}
If $H \colon T^*G \to \mathbb{R}$ is a natural left invariant Hamiltonian of the form $H(g, p) = \frac{1}{2}\mTse^{-1}(p,p)$ where $\mTse$ is a Riemanninan metric, and if $K$ is a connected, closed, and codimension 1 subgroup of $G$, then there exists a smooth function $s \colon G \to \mathbb{R}$ such that $\Sigma = s^{-1}(\set{0})$ and $\Delta \zeta$ is given by
\begin{gather}\label{eq:Delta_zeta}
   \Delta \zeta = - \frac{2ds_{g_0}((\ell_{g_0^{-1}})_*(\zeta^-)^\sharp)}{(\ell_{g_0}^*)ds_{g_0}((\ell_{g_0^{-1}})_* \nabla s_{g_0})}(\ell_{g_0})^*ds_{g_0}.
\end{gather}
\end{corollary}
\begin{proof}
    Since $K$ is closed and connected, $K \sm G$ is a 1-dimensional connected manifold, and $\pi \colon G \to K \sm G$ is a surjective submersion, and hence, differentiable. Thus $K\backslash G$ must be isomorphic to either $\mathbb{R}$ or $\mathbb{S}^1$. Without loss of generality, assume $K \sm G \cong \mathbb{R}$ and let $j \colon K \sm G \to \mathbb{R}$ be the isomorphism. Define $s$ to be 
    \begin{equation*}
        s \colon G \to \mathbb{R} : g \mapsto j(\pi(g)) - j(\pi(g_0))
    \end{equation*}
    Note that $g \in \Sigma$ if and only if $[g] = [g_0]$, which implies that $s(g) = j(\pi(g)) - j(\pi(g_0)) = 0$. Moreover, since both $j$ and $\pi$ are differentiable, $s$ is also differentiable.

    From Proposition \ref{prop:eom for Lie algebras} we get 
    \begin{align*}
        \Delta \zeta = -\frac{2\zeta^-((\ell_{g^{-1}})_*\nabla s_g)}{ ds_g (\nabla s_g)}(\ell_g)^*d s_g.
    \end{align*}
    This formula is already close to the desired result, but it depends on $g$. Our goal is to use the right coset property in order to eliminate the $g$ dependency. 
    
    From Lemma \ref{lemma:left_invariant_surface} there exists an $\alpha_{u, g} \in \R$ such that $\ell_g^* ds_{ug} = \alpha_{u, g}ds_u$. Let $g = kg_0 \in \Sigma$, $u = k \in G_{g_0}$ and $g = g_0 \in \Sigma$. Then, $\ell_{k}^*ds_{kg_0} = \alpha_{k, g_0}ds_g$ and similarly $(\ell_{k^{-1}})_*\nabla s_{kg_0} = \alpha_{k, g_0}^{-1}\nabla s_{g_0}$. Moreover, $\ell_{kg_0}^* = (\ell_k \circ \ell_{g_0})^* = \ell_{g_0}^*\ell_k^*$, and analogously $(\ell_{(kg_0)^{-1}})_* =( \ell_{g_0^{-1}})_*(\ell_{k^{-1}})_*$. Putting these two observations together into the formula for $\Delta \zeta$ we obtain
    \begin{align*}
        \Delta \zeta = - \frac{2\zeta^-((\ell_{g_0^{-1}})_* \alpha_{k, g_0}^{-1}\nabla s_g)}{ \ell_{g_0}^* \alpha_{k, g_0} ds_{g_0} (\alpha_{k, g_0}^{-1}(\ell_{g_0^{-1}})_*\nabla s_{g_0})}(\ell_{g_0})^* \alpha_{k, g_0} d s_{g_0}.
    \end{align*}
    Since $\alpha_{k, g_0}$ is a nonzero constant, we can pull it out of the pullback and pushforward, and since it appears in both the numerator and the denominator, we are left with
    \begin{align*}
        \Delta \zeta = -  \frac{  2\zeta^-((\ell_{g_0^{-1}})_* \nabla s_{g_0})}{ \ell_{g_0}^*  ds_{g_0} ((\ell_{g_0^{-1}})_*\nabla s_{g_0})}(\ell_{g_0})^*  d s_{g_0}.
    \end{align*}
    The final result follows from the same derivation as the end of the proof of Proposition \ref{prop:eom for Lie algebras}.
\end{proof}

\begin{remark}
    The normality of the subgroup $K$ is crucial for Hybrid Lie-Poisson reduction. If dropped, $\Delta \zeta$ might not be independent of $g$, and hence, the reset equations cannot be reduced. This is illustrated in the following counterexample: Consider the Lie group $G = \t{Aff}(1)$ and the subgroup $K$ of $G$, which are the following sets closed under matrix multiplication 
    \begin{align*}
        G = \left\{ 
        \begin{pmatrix}
            a & b \\ 0 & 1
        \end{pmatrix}
        : a \in \R\sm\set{0} \t{ and } b \in \R\right\}
        \qquad \t{and} \qquad
        K = \left\{ 
        \begin{pmatrix}
            a & 0\\ 0 & 1
        \end{pmatrix}: a \in \R\right\}.
    \end{align*}
    $K$ a codimension 1 closed subgroup of $G$; however, $K$ is not closed under the adjoint, and hence, not normal. Next, choose $g_0 = \begin{psmallmatrix} a_0  & b_0 \\ 0 & 1 \end{psmallmatrix}$ so then for $g \in \Sigma$ and $\delta g \in T_g\Sigma$, we have
    \begin{align*}
        \delta g = \begin{pmatrix} \dot\alpha a_0 & \dot \alpha b_0 \\ 0  & 0\end{pmatrix}
        \qquad \t{for }
        \begin{pmatrix}
            \dot\alpha & 0 \\ 0 & 0
        \end{pmatrix} \in \mathfrak{K}.
    \end{align*}
    However, 
    \begin{align*}
        (\ell_{g^{-1}})_*\delta g = 
        \frac{\dot\alpha}{\alpha} 
        \begin{pmatrix}
            1 & b_0/a_0 \\
            0 & 0
        \end{pmatrix}
    \notin \mathfrak{K}.
    \end{align*}
    The dual of the Lie algebra $\mathfrak{g}$ consists of vectors $P = (p_a, p_b)$, and the dual pairing $P(\dot A) = p_a \dot a + p_b \dot b$ is the usual dot product between $P$ and $(\dot a, \dot b)$. Left translation of $P$ is $\ell_C^*P = (cp_a, \ cp_b)$ where $C = \begin{psmallmatrix}
        c & d \\ 0 & 1
    \end{psmallmatrix}$. With $\zeta = \ell_A P $, the left translated corner conditions become
    \begin{align*}
        \begin{pmatrix}
        a(p_a^+ - p_a^-) & a(p_b^+ - p_b^-) 
    \end{pmatrix}\begin{pmatrix}
        \dot \alpha/\alpha \\
        \dot{\alpha}b_0/\alpha a_0 
    \end{pmatrix} = a (p_a^+ - p_a^-) \frac{\dot \alpha}{\alpha} + \frac{\dot{\alpha}b_0}{\alpha a_0} a(p_b^+ - p_b^-).
    \end{align*}
    Now use the fact that $a = \alpha a_0$ and that the above has to be $0$ for any $\alpha$ due to the corner conditions to obtain
    \begin{equation*}
        a_0(p_a^+ - p_a^-) + b_0(p_b^+ - p_b^-) = 0
    \end{equation*}
    This equation does not depend on $\alpha $ anymore which is good, but it also implies that $(p_a^+ - p_a^-)/(p_b^+ - p_b^-) = -b_0/a_0$. With $p_a^+ - p_a^- = \beta$, $\Delta \zeta = (\beta, -b_0 a\beta/a_0)$ which depends on $a$. So the reset map is not independent of the group element. 
\end{remark}

%% file: TikZ_Code/Torus_Super_Cool.tex
\begin{tikzpicture}[xscale = 1.5, yscale = 1.5]
    \def\a{60}
    \draw[thick, opacity = .5] (0,-2.5) ellipse (1.65 and {1.65*cos(\a)});
    %%% Technical Stuff %%%%%%%%%%%%%%%%%%%%%%%%%%%%%%%%%%%%%%%%%%%%%%%%%%%%%%%%%%%%%%%%%%%%%%%
    \fill[MariaGreen, opacity = .6](.25919,.18986) to[out = 186.65, in = 3.4748] (.03939,.17468) .. controls (-0.03622,0.16876) and (0.0902,-1.8322) .. (.1601,-1.8216) to[out = 3.4748, in = 186.65] (.4908, -1.7967) .. controls (0.2491,-1.8334) and (-0.01162,0.19972) .. cycle;
    \fill[MariaBlue!50!MariaGreen, opacity = .6] (1.5395, -1.5136) to[out = 204.3, in = 6.65] (.4908, -1.7967) .. controls (0.2491,-1.83346) and (-0.01162,0.19972) .. (.25919,.18986) to[out = 6.65, in = 204.3] (.7005,.30197) .. controls (0.024,0.0145) and (0.8,-1.82) .. cycle;
    \fill[MariaBlue, opacity = .6](.7754,.3442) to[in = 24.8, out = 216.2] (.7005,.30197) .. controls (0.015,0.02) and (0.8,-1.82) .. (1.5395, -1.5136) to[out = 24.8, in = 216.2] (1.9566,-1.2697)  .. controls (1.185,-1.7875) and (0.029,-0.3) .. cycle;
    %%%%%%%%%%%%%%%%%%%%%%%%%%%%%%%%%%%%%%%%%%%%%% Teehee %%%%%%%%%%%%%%%%%%%%%%%%%%%%%%%%%%%%%
    \draw[thick, shift={(1.36600887032, -0.462741765518)}, MariaBlue, rotate = 36.2] (0,1) arc[start angle=90, end angle=270,x radius= 0.7151, y radius = 1];
    \draw[thick, shift={(1.12, -0.605825844221)}, MariaBlue!50!MariaGreen, rotate = 24.8] (0,1) arc[start angle=90,end angle=270,x radius= 0.587895,y radius = 1];
    \draw[thick, shift={(0.74998,-0.734851450533)}, red, rotate = 14.3] (0,1)arc[start angle=90,end angle=270,x radius = 0.392, y radius = 1];
    \draw[thick, shift={(0.37499835136, -0.80341089059)}, MariaBlue!50!MariaGreen, rotate = 6.65] (0,1) arc[start angle=90,end angle=270,x radius = 0.1967, y radius = 1];
    \draw[thick, shift={(0.1,-0.823483280754)}, MariaGreen, rotate = {180+3.4748}] (0,-1) arc[start angle=-90,end angle=90,x radius = 0.052882 ,y radius = 1];
    %%%%%%%%%%%%%%%%%%%%%%%%%%%%%%%%%%%%%%%%%%%%%%%%%%%%%%%%%%%%%%%%%%%%%%%%%%%%%%%%%%%%%%%%%%%
    % {-180, ..., 195}{-43, ..., 0}
    \foreach \i in {-180, ..., 195}{%
        \def\ya{cos(deg(\i))*sin(\a)}
        \def\xr{1.65-sin(deg(\i))}
        \def\yr{cos(\a)*(1.65-sin(deg(\i)))}
        \draw[opacity = .05] (0,{\ya}) ellipse ({\xr} and {\yr});
        }
    %%%%%%%%%%%%%%%%%%%%%%%%%%%%%%%%%%%%%%%%%%%%%%%%%%%%%%%%%%%%%%%%%%%%%%%%%%%%%%%%%%%%%%%%%%%
    \fill[MariaGreen, opacity = .6](.25919,.18986) to[out = 186.65, in = 3.4748] (.03939,.17468) .. controls (.115,.1806) and (.23,-1.811) .. (.1601,-1.8216) to[out = 3.4748, in = 186.65] (.4908, -1.7967) .. controls (.7325,-1.76) and (.53,.18) .. cycle;
    \fill[MariaBlue!50!MariaGreen, opacity = .6] (1.5395, -1.5136) to[out = 204.3, in = 6.65] (.4908, -1.7967) .. controls (.7325,-1.76) and (.53,.18) .. (.25919,.18986) to[out = 6.65, in = 204.3] (.7005,.30197) .. controls (1.375,.59) and (2.25,-1.1) .. cycle;
    % nice looking open set. dont delete just in case
    % \draw[very thin, blue]  (.25919,.18986) to[out = 6.65, in = 204.3] (.7005,.30197) .. controls (1.375,.59) and (2.25,-1.1) .. cycle;
    \fill[MariaBlue, opacity = .6](.7754,.3442) to[in = 24.8, out = 216.2] (.7005,.30197) .. controls (1.375,.59) and (2.25,-1.1) .. (1.5395, -1.5136) to[out = 24.8, in = 216.2] (1.9566,-1.2697)  .. controls (2.71,-.62) and (1.54,.85) .. cycle;
    %%%%%%%%%%%%%%%%%%%%%%%%%%%%%%%%%%%%%%%%%%%%%%%%%%%%%%%%%%%%%%%%%%%%%%%%%%%%%%%%%%%%%%%%%%%
    % \draw[shift={(1.36600887032, -0.462741765518)}, MariaBlue, rotate = 36.2] (0,0) ellipse (0.7151 and 1);
    \draw[thick, shift={(1.36600887032, -0.462741765518)}, MariaBlue, rotate = 36.2] (0,-1) arc[start angle=-90, end angle=90,x radius= 0.7151, y radius = 1];
    % \draw[shift={(1.12000008299, -0.605825844221)}, MariaBlue!50!MariaGreen, rotate = 24.8] (0,0) ellipse (0.587895 and 1);
    \draw[thick, shift={(1.12000008299, -0.605825844221)}, MariaBlue!50!MariaGreen, rotate = 24.8] (0,-1) arc[start angle=-90,end angle=90,x radius= 0.587895,y radius = 1];
    
    \draw[thick, shift={(0.749982254857,-0.734851450533)}, red, rotate = 14.3] (0,-1)arc[start angle=-90,end angle=90,x radius = 0.392, y radius = 1];
    \fill[red] (.987,-1.711) -- (1.012,-1.708) -- (.997,-1.698) -- cycle;
    % \draw[shift={(0.37499835136, -0.80341089059)}, MariaBlue!50!MariaGreen, rotate = 6.65] (0,0) ellipse (0.1967 and 1);
    \draw[thick, shift={(0.37499835136, -0.80341089059)}, MariaBlue!50!MariaGreen, rotate = 6.65] (0,-1) arc[start angle=-90,end angle=90,x radius = 0.1967, y radius = 1];
    % \draw[shift={(0.1,-0.823483280754)}, MariaGreen, rotate = {3.4748}] (0,0) ellipse (0.052882 and 1);
    \draw[thick, shift={(0.1,-0.823483280754)}, MariaGreen, rotate = 3.4748] (0,-1) arc[start angle=-90,end angle=90,x radius = 0.052882, y radius = 1];
    %%% Less Technical Stuff %%%%%%%%%%%%%%%%%%%%%%%%%%%%%%%%%%%%%%%%%%%%%%%%%%%%%%%%%%%%%%%%%%
    \draw[thick, ->] (-3.075, -.25) -- (-3.075,-2.25);
    \draw (-2.9,-1.25) node [anchor = center] {$\pi$};
    
    \draw (-3.075, 0) node [anchor = center]{{\color{white}${}^2$}$\mathbb{T}^2{}$};
    \draw (-3.075, -2.5) node [anchor = center]{{\color{white}$\mathbb{T}^2$}$\mathbb{S}^1 \smm \mathbb{T}^2${\color{white}$\mathbb{S}^1$}};

    \filldraw[fill = blue2, fill opacity = .375](-.149,.3356) -- (.9517,.6174) -- (1.6562, -.0736) -- (.5469,-.3554) -- cycle;
    \filldraw[fill = blue2, fill opacity = .375](1.7305, -.3335) -- (1.7932, -1.5455) -- (.6735, -1.82) -- (.6108, -.6179) -- cycle;

    \begin{scope}
        \clip (.07,-3.285) rectangle (.35,-3.35);
        \draw[very thick, MariaGreen](0,-2.5) ellipse (1.65 and {1.65*cos(\a)});
    \end{scope}
    \begin{scope}
        \clip(.35,-3.35) rectangle(1.145,-3.075);
        \draw[very thick, MariaBlue!50!MariaGreen](0,-2.5) ellipse (1.65 and {1.65*cos(\a)});
    \end{scope}
    \begin{scope}
        \clip(1.145,-3.1) rectangle (1.39,-2.9);
        \draw[very thick, MariaBlue](0,-2.5) ellipse (1.65 and {1.65*cos(\a)});
    \end{scope}
    
    \draw (0.095, -3.323)  node [MariaGreen, anchor=center]{\resizebox{1.6mm}{3.2mm}{$($}};
    \draw (0.375, -3.303)  node [MariaBlue,  anchor=center]{\resizebox{1.5mm}{3.05mm}{$($}};
    \draw (1.120, -3.106)  node [MariaGreen, anchor=center]{\resizebox{1.475mm}{2.925mm}{$)$}};
    \draw (1.366, -2.9627) node [MariaBlue,  anchor=center]{\resizebox{1.45mm}{2.875mm}{$)$}};

    % Bottom tangent line
    \draw[shift={(0,.184)},blue2, thick, draw opacity = .375] (1.8, -3.18) -- (-.3, -3.66);

    \draw[shift={(0,.184)},-stealth, thick](.75, -3.42) -- (1.3, -3.295);
    \node at (.75, -3.235) [red1,circle,fill,inner sep=.75pt]{};
    
    % \draw[->] (.625, -3.15) to[curve through={(.1,-1.75)}] (.65, 0);
    % \draw[->] (.8, -3.125) to[curve through={(1,-1.75)}] (1.15, -1.27);
    \draw[->] (.635, -2.925) to[curve through={(.1,-1.75)}] (.65, 0);
    \draw[->] (.79, -2.925) to[curve through={(1.05,-1.75)}] (1.15, -1.27);
    \draw (-.13, -2.25) node [anchor=north west] {$\sigma_1$};
    \draw (.55, -2.05) node [anchor=north west] {$\sigma_2$};
    
    % \draw (1, -3.6) node [anchor=center] {\resizebox{!}{2mm}{$[\xi]$}};
    % \draw (.96, -.27) node [anchor=center] {\resizebox{!}{2mm}{$\sigma_1([\xi])$}};
    % \draw (.825, -1) node [anchor=center] {\resizebox{!}{2mm}{$\sigma_2([\xi])$}};

    \node at (.7493, .131) [circle,fill,inner sep=.75pt]{};
    \node at (1.202, -1.0817) [circle,fill,inner sep=.75pt]{};
    
    \draw[thick, -stealth](.7493, .131) -- (1.075, .45);
    \draw[thick, -stealth](1.202, -1.0817) -- (1.65, -1.12);

    % \draw (.75, -3.4) node [anchor = center, red] {\scalebox{1}{$q$}};
    \draw (.735, -3.075) node [anchor = center, red] {$q$};

    \draw (.45, .15) node [anchor = center] {$\sigma_{1}\!(q)$};
    \draw (.9275, -.95) node [anchor = center] {$\sigma_{2}(q)$};
    % \draw (1.1, -3.35) node [anchor = center] {\scalebox{.85}{$\dot{q}$}};

    % \draw[->] (.65, 0) to[curve through={(.1,-1.75)}] (.65, -1);
    \draw[->] (1.4, -1.15) to[curve through={(1.375, -1.47) .. (1.3, -1.77)}] (1.1, -3.05);
    \draw[->] (.95, .225) to[curve through={(1.9, -1.6) .. (1.65, -2) .. (1.3, -2.4)}] (1.1, -3.05);
    \draw (1.45, -2.) node [anchor = center] {${\scalebox{1}{$\pi_*$}}$};
\end{tikzpicture}

%% file: TikZ_Code/MariaLPO.tex
\begin{tikzpicture}[x=0.75pt,y=0.75pt,yscale=-1,xscale=1]
    \draw [thick, color=green!80!black] (96,107.1) .. controls (184,86.1) and (267,191.1) .. (308,161.1);

    \draw[thick,blue] (96,85.1) .. controls (163,104.1) and (278,48.1) .. (308,110.1);

    % G
    \draw[thick] (97,24) -- (308,24) -- (308,179.1) -- (97,179.1) -- cycle;

    \draw[thick] (97,203) -- (308,203);
    
    % green!80!black Arrows
    \draw [very thick, -stealth, color=green!70!blue] (210, 125) -- (210, 90.1);
    \draw [very thick, -stealth, color=green!70!blue] (250, 140) -- (250, 87.5);
    \draw [very thick, -stealth, color=green!70!blue] (290, 157.5) -- (290, 100);
    
    \draw (310, 203) node (KG) [anchor = west] {$K\! \sm\! G$};
    % \draw[->] (KG.center) -- (332,150);
    \draw[->] (332,110) -- (332,190);
    
    \fill[fill=red1] (138, 106) circle (2pt);
    \fill[fill=red1] (138,203) circle (2pt);
    
    \draw [color=red1, dash pattern={on 4.5pt off 4.5pt}] (138,25.1) -- (138,179.1);

    \draw (270,115) node [anchor = center][color=green!70!blue] {$k(t)$};
    \draw (215,55) node [anchor=north west,blue] {$g(t)$};
    \draw (215,150) node [anchor=north west, color=green!80!black] {$g_{0}(t)$};
    \draw (93, 203) node [anchor = east] {$q(t)$};
    
    \draw (332, 100) node [anchor= center] {$G$};
    \draw (150,100) node [anchor = center, red1] {$g_0$};
    \draw (125,140) node [anchor = center, red1] {$\Sigma$};
    \draw (138,208) node [anchor=north,red1] {$[g_0]$};
    \draw (340,150) node [anchor=center] {$\pi$};
    
    % Aden: Center alignment
    \draw[white] (45, 198) -- (45, 208);
\end{tikzpicture}

%% file: Applications/Numerical_experiments_intro.tex
The theory developed above enables us to speed up the computation of the Frobenius-Perron operator for hybrid systems. First, solving the partial differential equation \eqref{eq:chikos} is faster than computing the Frobenius-Perron operator using the definition \eqref{def:hybrid_FP}. For the latter one would need to sample a large number of initial points from the given distribution, and compute their trajectories for a long period of time. This endeavour suffers greatly under curse of dimensionality. Second, by doing impact reduction on the hybrid system at hand we can reduce the dimensionality, which further decreases the computation memory and time. 

In the following we illustrate the theory, and its applicability on a number of diverse examples. First, we consider the classical examples of the bouncing ball and the nonholonomic Chaplygin sleigh. Next, we provide analytical results in the case where the configuration space is $\t{GL}_n(\R)$ and the impact surface is the normal subgroup $\t{SL}_n(\R)$. Finally, we employ our method to analyse an SIR model of an epidemic, with event-based human intervention.

The numerical simulations are done using a semi-Lagrangian discretization scheme. At each grid point a characteristic is run backward for some fixed $\Delta t$, and the value of the solution is updated by using the nearest grid point to the foot of the characteristic. Although this method is quite crude, the simulations are much faster than the ones that use the definition based approach to computing the Frobenius-Perron operator (see baseline in \cite{hybrid_forms}). All simulations are performed in \texttt{matlab} on a MacBook Pro (10 cores and 32 GB memory). The code can be found at \texttt{https://github.com/Mary199810/Hybrid-transfer-operators.git}.

%% file: Applications/The_bouncing_ball.tex
The canonical pedagogical example is the bouncing ball with elastic impacts. The equations of motion in the Hamiltonian framework are
\begin{equation*}
    \dot{x} = \frac{p}{m}, \quad
    \dot{p} = -mg,
\end{equation*}
impacts occur when the ball strikes the ground ($x = 0$), and the reset map is $\Delta(0, p) = (0, -p)$. Thus when the ball hits the ground, the sign momentum is reversed since no dissipation occurs during impact. The hybrid continuity equation, \eqref{eq:chikos}, is
\begin{align}\label{eq:BB}
    \begin{cases}
        \dfrac{\partial u}{\partial t} + \dfrac{p}{m}\dfrac{\partial u}{\partial x} - mg \dfrac{\partial u}{\partial p} = 0, & \t{for } x \ne 0; \\[2pt]
        u(t^+, 0, p) = u(t^-, 0, -p), &  \t{otherwise.} %x = 0.
    \end{cases}
\end{align}

Consider an initial distribution $\rho(x, p) = e^{ - (x - 1)^2 - p^2}$ and let $m = 1$ and $g = 1$. The results of the simulations are shown in Figure \ref{fig:BB simulations}. The simulation time was \texttt{252.213084 seconds}.
\begin{figure}[!ht]
    \centering
     \begin{subfigure}[t]{0.45\textwidth}
         \centering
         \includegraphics[width=0.8\textwidth]{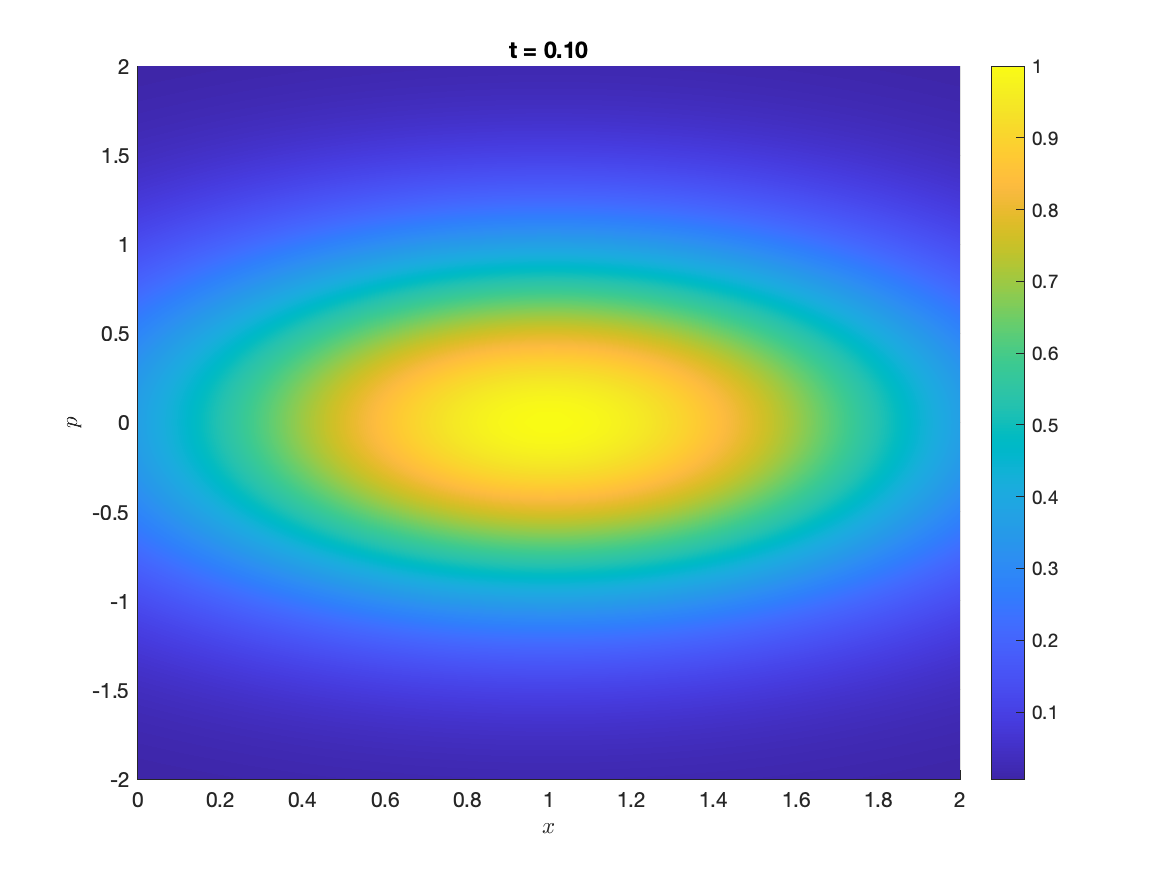}
         \caption{The initial density is centered around $(1, 0)$.}
         \label{subfig:1a}
         % \label{fig:y equals x}
     \end{subfigure}
     \hfill
     \begin{subfigure}[t]{0.45\textwidth}
         \centering
         \includegraphics[width=0.8\textwidth]{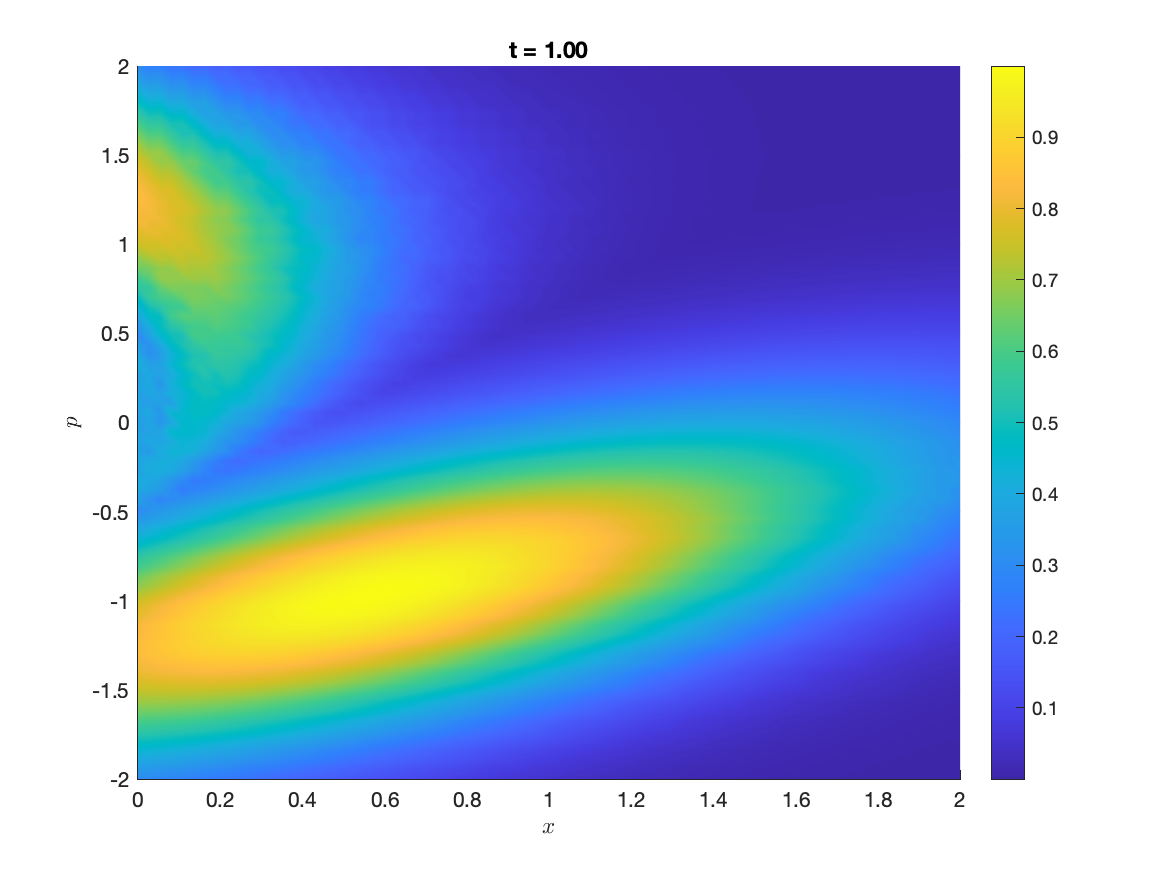}
         \caption{Immediately after the start of the simulation, at $t = 1$, the mass is attracted towards the impact surface $x = 0$.}
         \label{subfig:2a}
     \end{subfigure}
     \hfill
     
      \begin{subfigure}[t]{0.45\textwidth}
         \centering
         \includegraphics[width=0.8\textwidth]{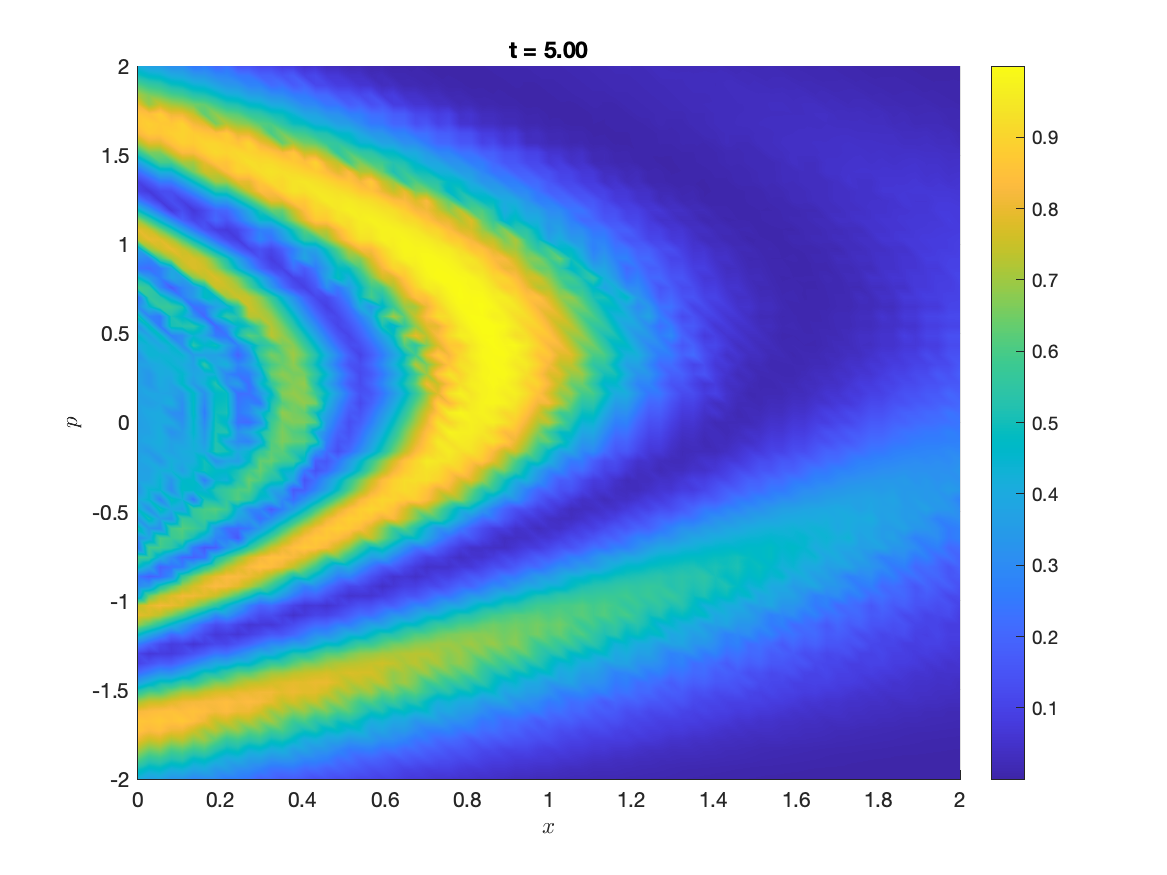}
         \caption{Around $t = 5$  the mass starts to concentrate in stripes around the Zeno point $(0, 0)$.}
         \label{subfig:3a}
     \end{subfigure}
     \hfill
      \begin{subfigure}[t]{0.45\textwidth}
         \centering
         \includegraphics[width=0.8\textwidth]{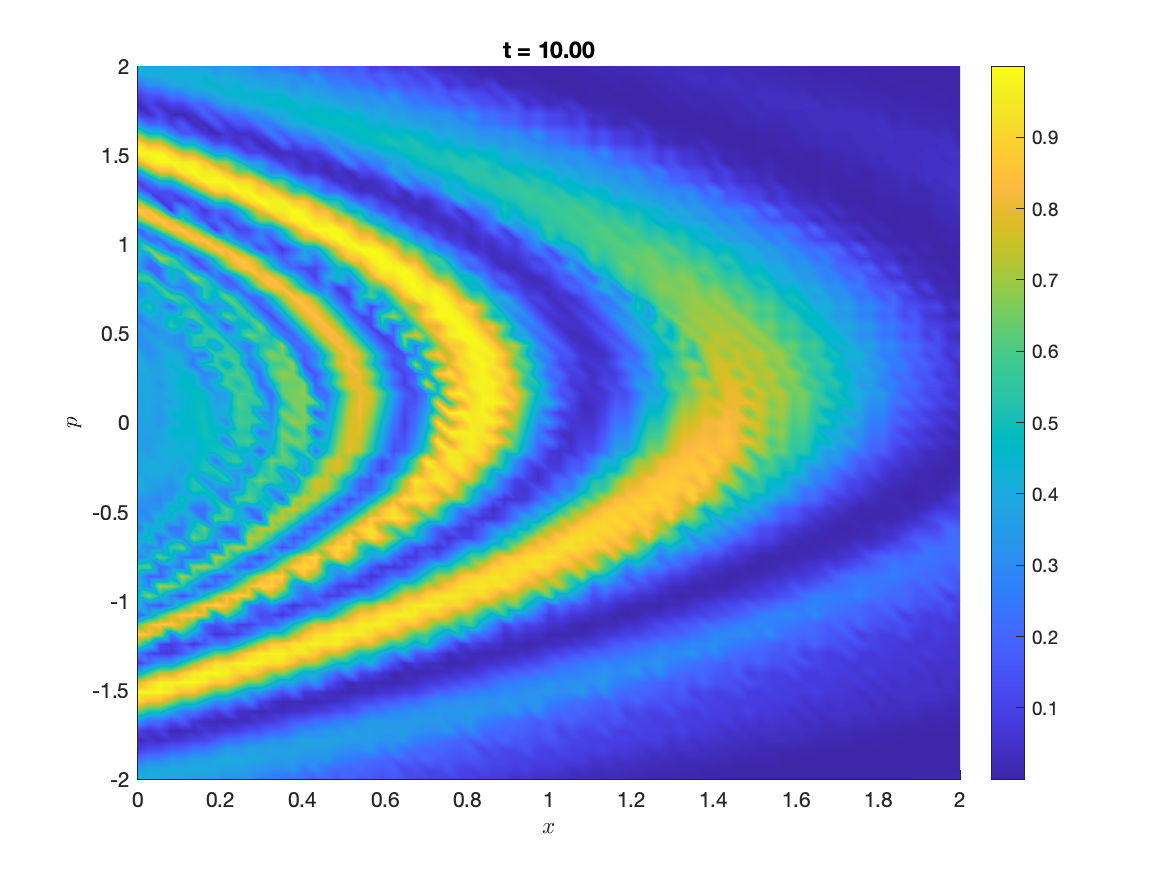}
         \caption{As time becomes large, this effect is stronger.}
         \label{subfig:4a}
     \end{subfigure}
     \hfill
    \caption{Snapshots of the evolution of an initial density under the hybrid transfer operator PDE \eqref{eq:BB} corresponding the the elastic bouncing ball.}
    \label{fig:BB simulations}
\end{figure}

%% file: Applications/Chaplygin_Sleigh_2.tex
The Chaplygin sleigh was introduced in the work of Chaplygin \cite{chaplygin}, and is one of the best studied examples of nonholonomic systems. It describes the motion of a rigid body in the plane supported on a knife edge. The body can move freely without friction in the direction parallel to the edge, and it can rotate around it, but it cannot move perpendicularly. The Chaplygin sleigh model is useful for describing the motion of a car, in particular when discussing parallel parking \cite{car_chaplygin}, for motion planning problems of skaters in the plane \cite{rhodes_skating}, and for more complex robotic systems with wheels \cite{roller_racer}. The natural state space is $\t{SE}_2$ with coordinates $(x, y, \theta)$ where $x$ and $y$ are Cartesian coordinates that represent the position of the sleigh within the plane and $\theta$ is the sleigh's orientation. The Lie algebra is composed of matrices of the form $A$ as stated in \eqref{eqn:ref A}.
For this system, the Hamiltonian is $H = \frac{1}{2}p^TM^{-1}p$, where $p = (p_x, p_y, p_\theta)$ and 
\begin{align}\label{eqn:ref A}
    A = \begin{pmatrix}
    0 & \dot\theta & \dot x \\
    -\dot\theta & 0 & \dot y \\
    0 & 0 & 0
    \end{pmatrix} \qquad \qquad \qquad
    M = \begin{pmatrix}
    m & 0 & -ma\sin\theta \\
    0 & m & ma\cos\theta \\
	-ma\sin\theta & ma\cos\theta & I+ma^2
\end{pmatrix}.
\end{align}
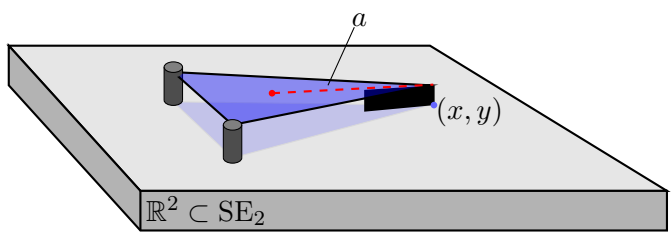
\begin{figure}[ht!]
    \centering
    \input{TikZ_Code/Dora_Chaplygin_Sleigh}
    \caption{Schematic drawing of the Chaplygin sleigh, where $a$ is the distance from the center of mass to the knife edge.}%, $I$ is the moment of inertia and $m$ is the mass}
    \label{fig:chaplygin}
\end{figure}
\!\!The Hamiltonian is left invariant, i.e., $H(g, p) = H(e, g^{-1}p)$ for any $g \in \t{SE}_2$, where $e$ is the identity matrix in $\t{SE}_2$. It is easier to work in the Lagrangian formulation with $L = \frac{1}{2}\dot{g}M \dot{g}$. Then left translation to identity is 
\begin{align*}
    \zeta = g^{-1}\dot{g} = 
    \begin{pmatrix}
        \dot{x}\cos\theta + \dot{y}\sin\theta \\
        -\dot{x}\sin\theta + \dot{y}\cos\theta \\
        \dot\theta
    \end{pmatrix}
\end{align*}
The first component is the velocity parallel to the edge, the second is the perpendicular velocity and the third is the angular velocity of the sleigh. 
The dimension can further be reduced by taking the nonholonomic constraint $ -\dot{x}\sin\theta + \dot{y}\cos\theta = 0$ into consideration. This is exactly the second component of $\zeta$. Let $v = \dot{x}\cos\theta + \dot{y}\sin\theta$ and $\omega = \dot\theta$. Then the equations of motion can be written in the $(v, \omega)$ coordinates as
\begin{equation*}
    \begin{cases}
        \dot v = a\omega^2\\
        \dot \omega = -\dfrac{ma}{I + ma^2}v\omega
    \end{cases}
\end{equation*}

Now assume that the impacts are angle dependent, i.e. $\Sigma = \{(x, y, \theta) \in \t{SE}_2 : \theta \in \{ \pm\theta_0\}\} %= \{(x, y, \theta_0)\}
$ which is a right coset of $\R^2$, who is in turn a codimension $1$ subgroup of $\t{SE}_2$. In the example of a simplified vehicle, this is equivalent to having the wheel turn suddenly after reaching a critical angle. The reset map in the reduced coordinates is:
\begin{equation*}
    \begin{cases}
        v^+  = v^-\\
        \omega^+  = -\omega^-
    \end{cases}
\end{equation*}

Apart from $v$ and $\omega$ we need another variable to determine where impacts happen. This is the role of $q$ in Theorem \ref{thm:Impact_Reduction}.  $\R\backslash \t{SE}_2\approx \mathbb{S}^1$ so $q = \theta \in \mathbb{S}^1$. For the local section $\sigma(\theta) = (0, 0, \theta)$, the equations of motion \eqref{eq:LP cont} are equivalent to $\dot\theta = \omega$ and the reset map for $\theta$ is identity, so $\theta^+ = \theta^-$. Putting everything together, the reduced hybrid dynamics are given by\\
\begin{center}
    \begin{tabular}{lcl}
        Continuous: & \hspace{5mm} & Discrete: \\[1ex]
        $\ds \dot v = a\omega^2$ & & $\ds v^+ = v^-$ \\[1ex]
        $\ds \dot\omega = -\frac{ma}{I + ma^2}v\omega$ & & $\ds \omega^+ = -\omega^-$ \\[2ex]
        $\ds \dot \theta = \omega $ & & $\ds \theta^+ = \theta^-$. \\[1ex]
    \end{tabular}
\end{center}
Plugging this vector field into \eqref{eq:chikos} and using the usual measure on $\t{SE}_2$: $\mu = dv \wedge d\omega \wedge d\theta$, we obtain the PDE for the transfer operator
\begin{align*}
    & \frac{\partial u}{\partial t} -\frac{\partial u}{\partial v} a\omega^2 - \frac{ma}{I + ma^2} v\omega \frac{\partial v}{\partial \omega } + \omega\frac{\partial v}{\partial \theta} - \frac{ma}{I + ma^2}vu = 0 & \text{if } \theta \neq \theta_0; \\[2pt]
    &  u(t^+, v, -\omega, \theta) = u(t^-, v, \omega, \theta) & \text{if } \theta = \theta_0.
\end{align*}
The dynamics of the Chaplygin sleigh can be further reduced by using the conservation of energy propriety and restricting to an energy surface. Let $E$ be the constant energy of the  sleigh, which can be written in terms of $v$ and $\omega$ as
\begin{equation*}
    E = \frac{1}{2}mv^2 + \frac{1}{2}(I + ma^2)\omega^2,
\end{equation*}
We can express $\omega$ in terms of $v$ as $\omega = \pm\sqrt{C_1 - C_2\omega^2}$ where $C_1 = 2E/( I +ma^2)$ and  $C_2 = m/(I + ma^2)$. There are two branches of the solution to $\omega^2 = C_1 - C_2v^2$. When an impact happens, the branch is changed, leading us to the 2 dimensional system
\begin{align*}
    \begin{cases}
        \dot v = aC_1 - aC_2v^2\\
        \dot\theta = \pm \sqrt{C_1 - C_2v^2} = X(v)
    \end{cases} \t{if } \theta \neq \theta_0,
    \qquad \t{with} \qquad
    \begin{cases}
        v^+ =  v^-\\
        \theta^+ = \theta^-
    \end{cases} \t{ if }\theta = \theta_0;
\end{align*}
  This is a Filippov system where the vector field changes as we cross the lines $\theta = \pm \theta_0$. The state space looks like two copies of $\R\times\mathbb{S}^1$ glued along the lines $\theta = \pm\theta_0$ as in Figure \ref{fig:reduced chaplygin}.
  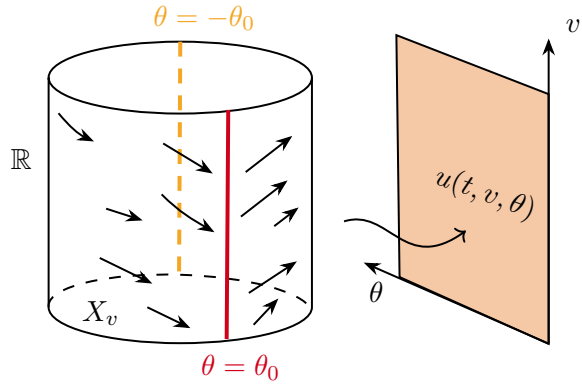
\begin{figure}
      \centering
      \input{TikZ_Code/reduced_chaplygin}
      \caption{State space for the Chaplygin sleigh restricted to an energy surface. The two halves of the cylinder are stacked on top of each other when displaying $u(t, v, \theta)$, the solution to the Hybrid transfer PDE \eqref{eq:chikos}.}
      \label{fig:reduced chaplygin}
  \end{figure}

Moreover, the PDE for the Frobenius-Perron operator is
\begin{align*}
    &\frac{\partial u}{\partial t} + a(C_1 - C_2v^2)\frac{\partial u}{\partial v} + X(v)\frac{\partial u}{\partial \theta} + 2aC_2vu = 0\\
   & u(t^+, \omega, \pm\theta_0) = u(t^-, -\omega, \pm\theta_0)
\end{align*}
In the plots below (Figure \ref{fig:chaplygin simulations}), the two copies of $\R\times\mathbb{S}^1$ are identified and the value of the solution is the sum of the value on each copy. 

\begin{figure}
    \centering
     \begin{subfigure}[t]{0.48\textwidth}
         \centering
         \includegraphics[width=0.8\textwidth]{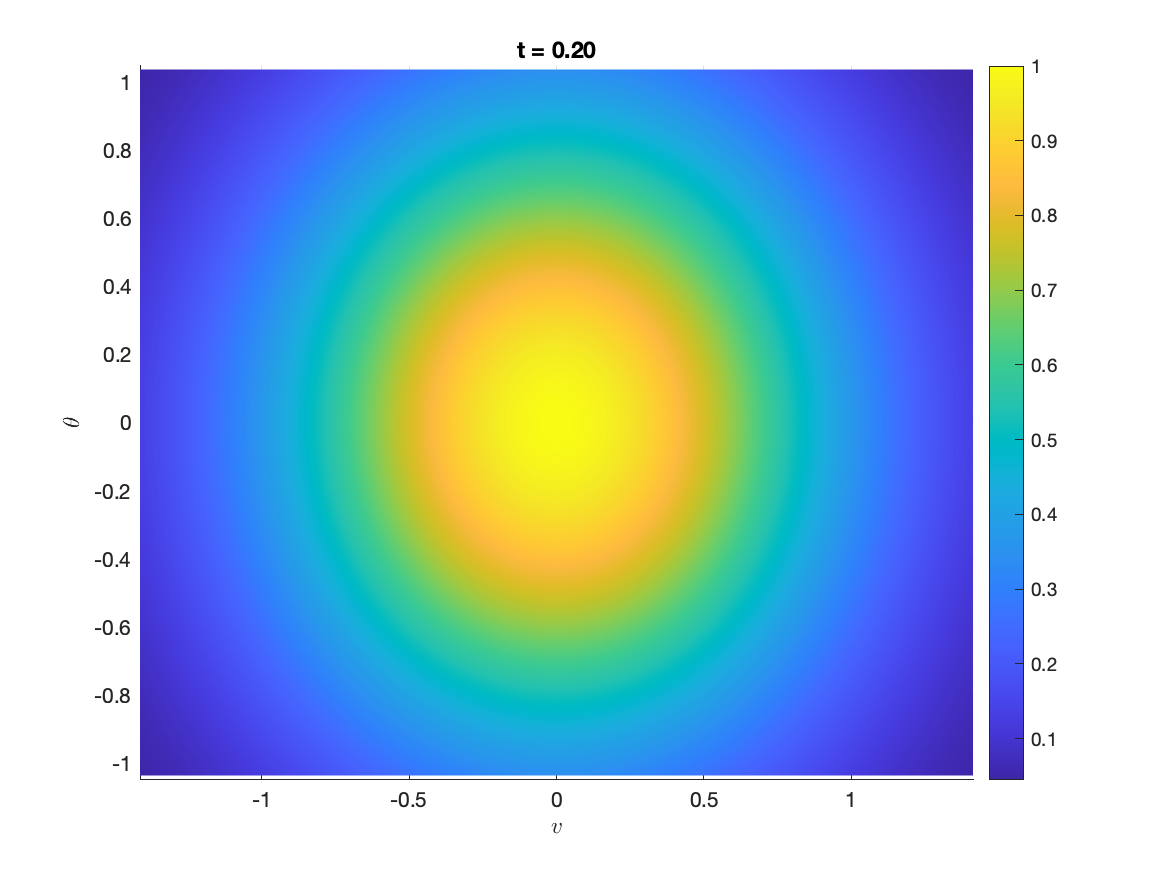}
         \caption{The initial density is $\rho(v, \theta) = e^{-v^2 - \theta^2}$}
         \label{subfig:1}
     \end{subfigure}
     \hfill
      \begin{subfigure}[t]{0.48\textwidth}
         \centering
         \includegraphics[width=0.8\textwidth]{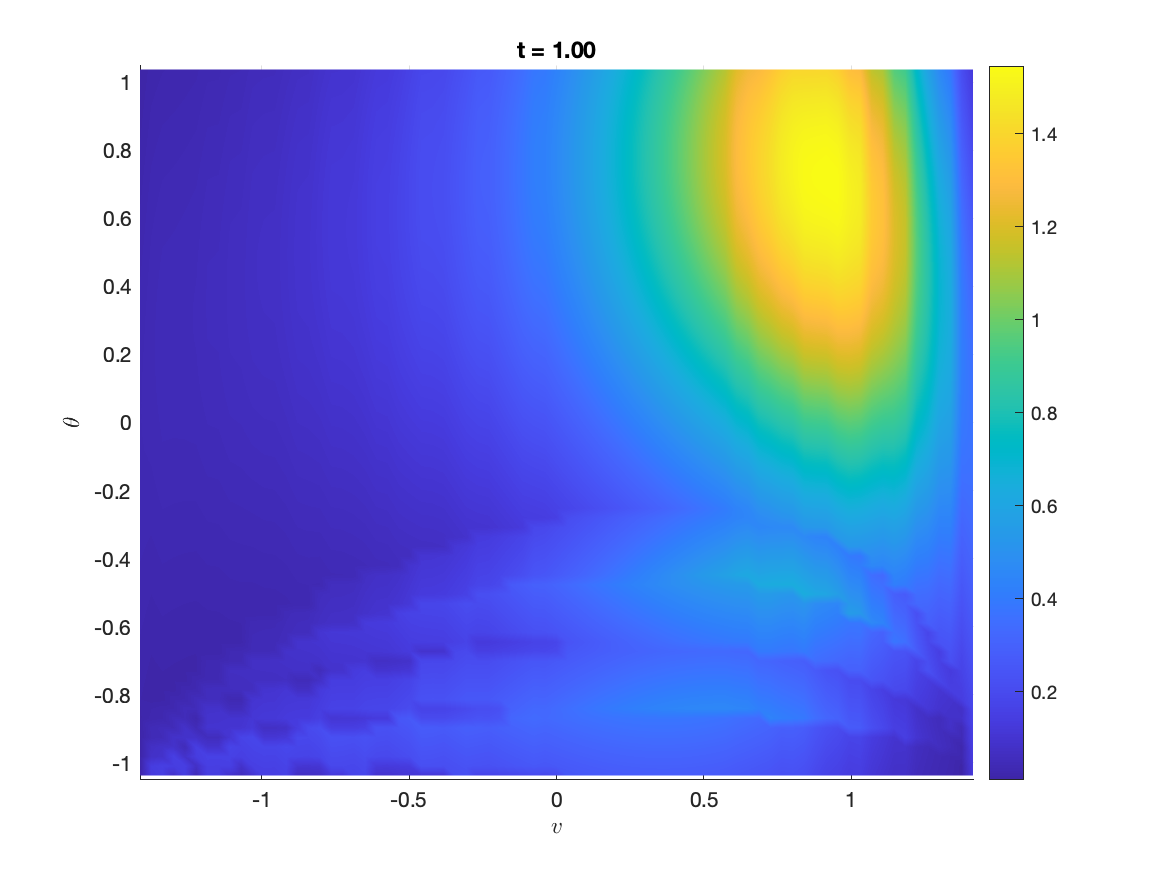}
         \caption{Immediately, the mass is pulled towards the right side, which corresponds to the stable fixed point $v^* = \sqrt{C_1/C_2}$.}
         \label{subfig:2}
     \end{subfigure}
     \hfill
      \begin{subfigure}[t]{0.48\textwidth}
         \centering
         \includegraphics[width=0.8\textwidth]{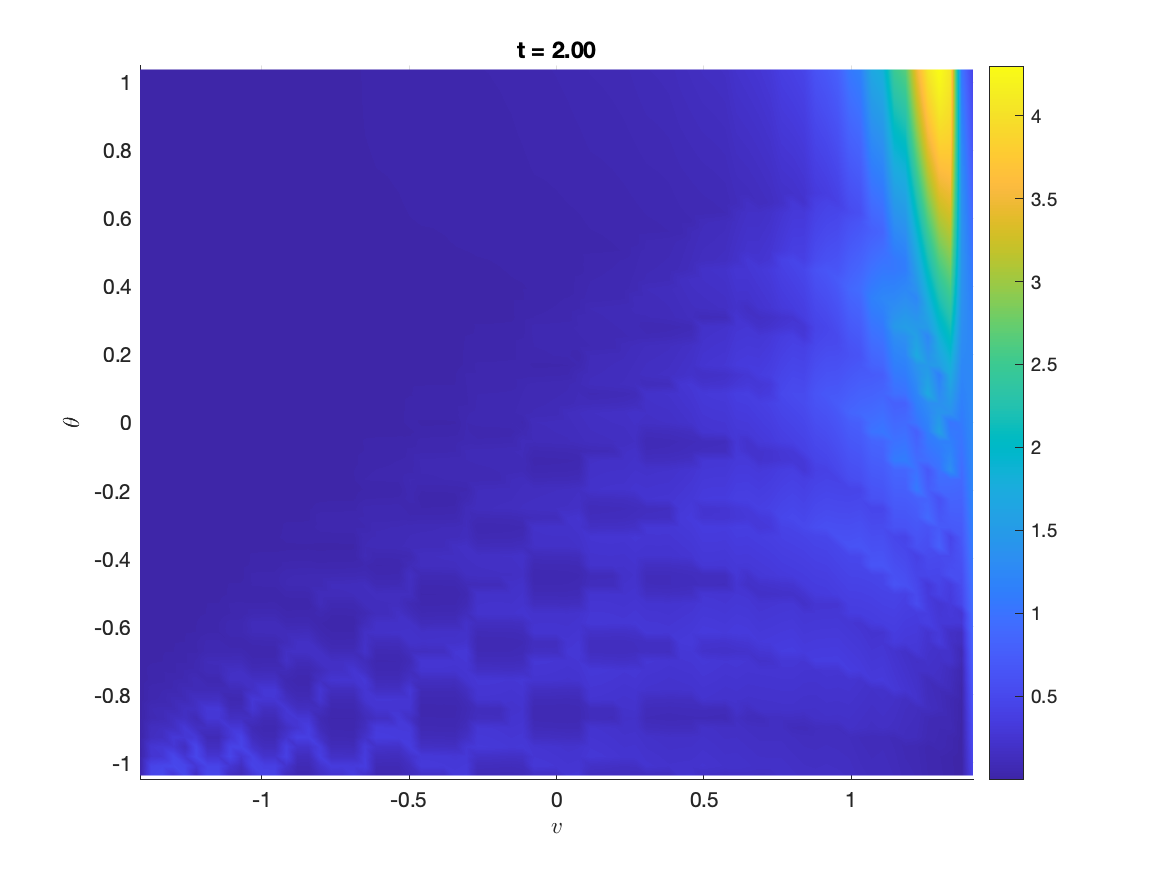}
         \caption{A sharp peak begins to form in the upper right corner. After an impact, $\theta$ rapidly decreases so some of the mass gathers around lower $\theta$ values.}
         \label{fig:y equals x}
     \end{subfigure}
     \hfill
      \begin{subfigure}[t]{0.48\textwidth}
         \centering
         \includegraphics[width=0.8\textwidth]{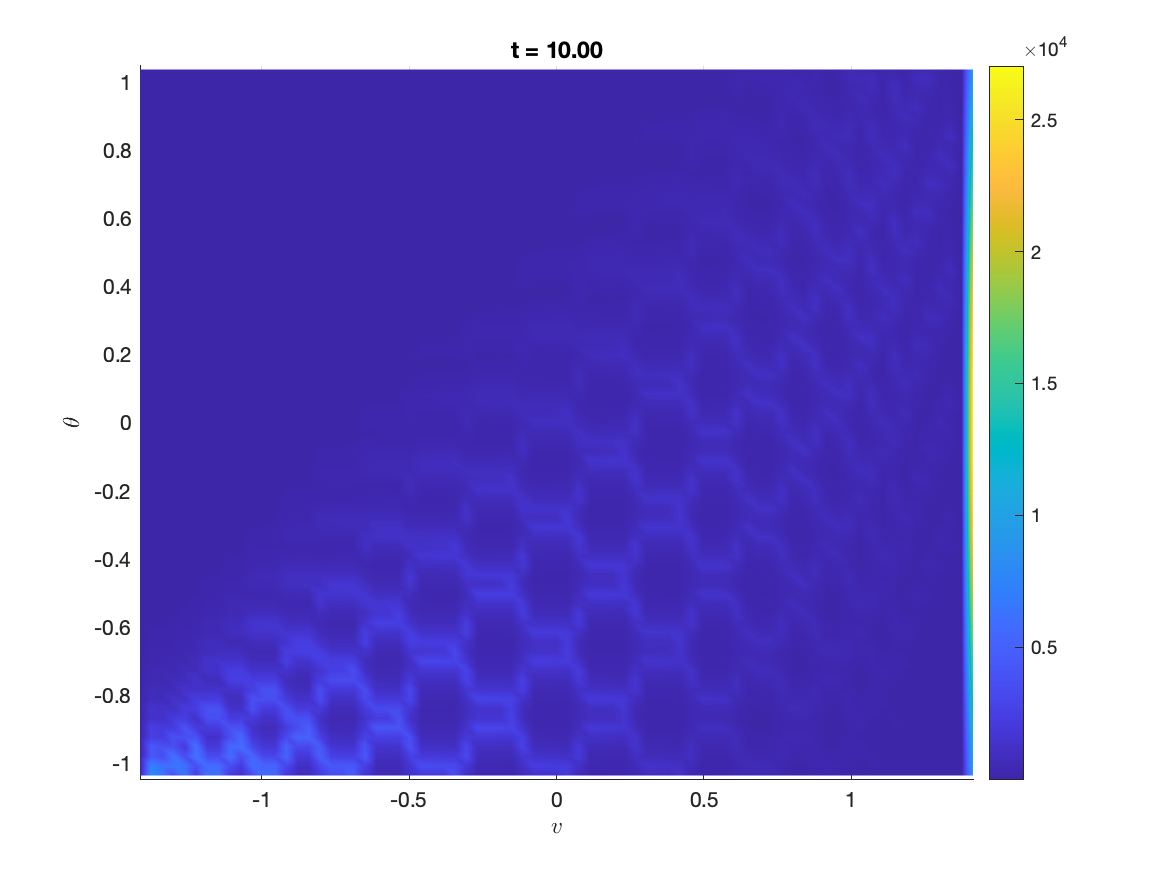}
         \caption{After a longer time, a delta-like peak forms in the upper right corner.}
         \label{subfig:4}
     \end{subfigure}
    \caption{Evolution of an initial density under the reduced hybrid Chaplygin sleigh equations of motion.}
    \label{fig:chaplygin simulations}
\end{figure}

%% file: TikZ_Code/Dora_Chaplygin_Sleigh.tex
\begin{tikzpicture}[x=1.75cm, y=1.75cm]
    \newcommand{\viewa}{-30}
    \newcommand{\viewb}{30}
    
    \draw[thick, fill=black!10!white] (1.8,-.4) -- (5,-.4) -- (6.8,-1.5) -- (2.8,-1.5) -- cycle;
    
    \draw[thick, fill = black!30] (1.8,-.7) -- (2.8,-1.8) -- (6.8,-1.8) -- (6.8,-1.5) -- (2.8,-1.5) -- (1.8,-.4) --cycle;
    \draw[thick](2.8,-1.5) -- (2.8,-1.8);
    
    %%%%%%%%%%%%%%%%%%%%%%%%%%%%%%%%%%%%%%%%%%%%%%%%%%%%%%%%%%%%%%%%%%%%
    %triangle
    \node(r0) at (3.5,  -1) {}; % bottom
    \node(s0) at (3.05, -0.6) {}; % top
    \node(s1) at (5, -0.7) {}; % front 
    
    %shadow
    \draw[draw=gray, fill=blue, opacity=0.15](3.5, -1.25) -- (3.05, -0.825) -- (5, -0.85) -- cycle;
    
    % skate
    \fill (4.5,-0.9) -- (5.035,-0.85) -- (5.035,-0.7) -- (4.5,-0.735) -- cycle;

    %back leg top
    \fill[fill=black!70] (2.98,-0.5625) -- (2.98,-0.8125) arc (180:360:0.07 and 0.04) -- (3.12,-0.5625) arc (0:180:0.07 and -0.04);
    \draw (3.12,-0.8125) -- (3.12,-0.608);
    \filldraw[thick, fill=blue, fill opacity=0.4] (r0.center)--(s0.center)--(s1.center) -- cycle;
    \fill [fill=black!70] (3.05, -0.6) ellipse (0.07 and 0.04);
    \draw (3.12,-0.61) -- (3.12,-0.5625);
    
    \draw [black,fill=black!50] (3.05, -0.5625) ellipse (0.07 and 0.04);
    \draw (2.98,-0.8125) arc (180:360:0.07 and 0.04);

    \draw (2.98,-0.5625) -- (2.98,-0.8125);

    %back leg bottom
    \fill[fill=black!70] (3.43,0) -- (3.43,-1.25) arc (180:360:0.07 and 0.04) -- (3.57,-1) arc (0:180:0.07 and -0.04);
    \draw[black,fill=black!50] ( 3.5,  -1) ellipse (0.07 and 0.04);
    \draw(3.43,-1) -- (3.43,-1.25);
    \draw(3.43,-1.25) arc (180:360:0.07 and 0.04);
    \draw(3.57,-1.25) -- (3.57,-1);
    
    % Schematic bits
    \draw[red, dashed, thick](3.8,-.76) -- (5.025, -0.695);
    \fill[red] (3.8,-.76) circle (1.25pt);
    \draw (4.6,-.2) node [anchor = east] {$a$};
    \draw (4.425,-.26) -- (4.225,-.7);

    \draw[white](3,-2) -- (3.1, -2);

    \fill[fill = blue!60] (5.03,-0.85) circle (1.25pt);
    \draw (5.3,-.9) node [anchor = center] {$(x,y)$};

    \draw (3.3,-1.65) node [anchor = center] {$\mathbb{R}^2 \subset \text{SE}_2$};

    % \node[color=black,cm={cos(\viewa),-sin(\viewa)*sin(\viewb),
    %     sin(\viewa),cos(\viewa)*sin(\viewb),
    %     (0,0)}]at (2.95,-1.35) {$\mathbb{R}^2$}; 
\end{tikzpicture}

%% file: TikZ_Code/reduced_chaplygin.tex
\tikzset{every picture/.style={line width=0.75pt}} %set default line width to 0.75pt 

\begin{tikzpicture}[x=0.75pt,y=0.75pt,yscale=-1,xscale=1, scale = 0.8]
%uncomment if require: \path (0,235); %set diagram left start at 0, and has height of 235

%Shape: Can [id:dp8118168356882467] 
\draw (256.76,51.71) -- (256.76,196.08) .. controls (256.76,208.56) and (219.6,218.68) .. (173.76,218.68) .. controls (127.92,218.68) and (90.76,208.56) .. (90.76,196.08) -- (90.76,51.71) .. controls (90.76,39.22) and (127.92,29.1) .. (173.76,29.1) .. controls (219.6,29.1) and (256.76,39.22) .. (256.76,51.71) .. controls (256.76,64.19) and (219.6,74.31) .. (173.76,74.31) .. controls (127.92,74.31) and (90.76,64.19) .. (90.76,51.71);
%Shape: Arc [id:dp7533292070690075] 
\draw [draw opacity=0][dash pattern={on 4.5pt off 4.5pt}] (90.76,196.08) .. controls (90.72,195.74) and (90.7,195.4) .. (90.71,195.06) .. controls (91.1,182.89) and (128.61,174.21) .. (174.49,175.69) .. controls (216.16,177.02) and (250.4,186.27) .. (256.04,197.08) -- (173.78,197.73) -- cycle; 
\draw [dash pattern={on 4.5pt off 4.5pt}] (90.76,196.08) .. controls (90.72,195.74) and (90.7,195.4) .. (90.71,195.06) .. controls (91.1,182.89) and (128.61,174.21) .. (174.49,175.69) .. controls (216.16,177.02) and (250.4,186.27) .. (256.04,197.08); 
%Straight Lines [id:da957116869663684] 
\draw [color={rgb, 255:red, 208; green, 2; blue, 27 } ,draw opacity=1 ][line width=1.5] (204,72) -- (203,216.1);
%Straight Lines [id:da4252202580407448] 
\draw [color={rgb, 255:red, 245; green, 166; blue, 35 } ,draw opacity=1 ][line width=1.5] [dash pattern={on 5.63pt off 4.5pt}] (173.76,29.1) -- (173,175.1);

%Straight Lines [id:da8716612670418078] 
\draw [-Stealth] (97,74.1) to[out = 50,in = 205] (119.43,91.86);%
\draw[-Stealth] (162,125.1) to[out = 45,in = 210] (195.37,148.94);
\draw [-Stealth] (127,134.1) -- (150.1,141.49);
\draw [-Stealth] (123,165.1) -- (156.2,181.23);
\draw [-Stealth] (153,196.1) -- (180.2,209.23);
\draw [-Stealth] (163,93.1) -- (194.29,112.06);

\draw [-Stealth] (215,111.1) -- (244.42,88.32);
\draw[-Stealth] (212,139.1) -- (241.42,116.32);
\draw [-Stealth] (233,145.1) -- (248.41,133.31);
\draw[-Stealth] (217,187.1) -- (247.33,167.2);
\draw[-Stealth] (220,207.1) -- (235.63,190.55);

\filldraw[fill = roseorange, fill opacity = .4] (312,177.1) -- (310,25) -- (406,62.1) -- (406,219.1) -- cycle;
\draw[-Stealth] (406,219) -- (406,27.1);
\draw[-Stealth] (406,219) -- (289.33,168);

% \draw [shift={(287.5,167.6)}, rotate = 23.49] [color={rgb, 255:red, 0; green, 0; blue, 0 } ][line width=0.75] (10.93,-3.29) .. controls (6.95,-1.4) and (3.31,-0.3) .. (0,0) .. controls (3.31,0.3) and (6.95,1.4) .. (10.93,3.29) ;
%Curve Lines [id:da15215799483010006] 
\draw[->] (277,142.1) .. controls (305.71,133.19) and (314.82,176.23) .. (353.81,147.98);

% Text Node
\draw (65,95) node [anchor=north west][inner sep=0.75pt] {$\mathbb{R}$};
\draw (185,224.4) node [anchor=north west][inner sep=0.75pt] {$\textcolor[rgb]{0.82,0.01,0.11}{{\theta = \theta_0}}$};
\draw (155,4.4) node [anchor=north west][inner sep=0.75pt] {$\textcolor[rgb]{0.96,0.65,0.14}{{\theta = -\theta_0}}$};
\draw (337.03,102.56) node [anchor=north west][inner sep=0.75pt] [rotate=-20] {$u(t, v, \theta )$};
\draw (291,178.4) node [anchor=north west][inner sep=0.75pt] {$\theta $};
\draw (415,13.4) node [anchor=north west][inner sep=0.75pt] {$v$};
\draw (110,190.4) node [anchor=north west][inner sep=0.75pt] {$X_v$};
\end{tikzpicture}

%% file: Applications/Matrix_Groups.tex
Consider the case where the configuration space is $Q = G = \t{GL}_n(\mathbb{R}) $, with impact surface $\Sigma = \t{SL}_n(\mathbb{R})$. The Lie algebra of $G$ is $\mf{gl}_n(\mathbb{R}) \cong \mathbb{R}^{n \times n}$ with Lie bracket $[A, B] = BA - AB$ %the usual commutator 
and structure coefficients $C_{ij}^k$ for $i, j, k \in \{1, \dots, n^2\}$. The dual of the Lie algebra consists of matrices $p \in \mathbb{R}^{n \times n}$ with the dual pairing given by $\langle p, \xi \rangle = \sum_{ij}p_{ij}\xi_{ij}$. The pullback of left translation on dual vectors can be written in terms of matrix multiplication as $\ell_A^* p = A^T p$.  

Let $H(A, P) = \frac{1}{2}\tr(A^TPP^TA)$. This Hamiltonian is left-invariant as  $H(I_n, A^TP) = H(A, P)$. Throughout this section $I_n$ denotes the identity in $\t{GL}_n(\R)$. Let $\zeta := A^TP\in\mf{gl}_n^*$ and denote by $h \colon \mf{gl}_n^* \to \R$ the Hamiltonian restricted to the dual of the Lie algebra; in coordinates  $h(\zeta) = \frac{1}{2}\tr(\zeta\zeta^T)$. To determine the continuous component of the reduced motion, we view $\zeta\in\mf{gl}_n^*$ as a row vector in $\mathbb{R}^{n^2}$ and $X,Y\in\mf{gl}_n$ as column vectors in $\R^{n^2}$.

The coadjoint representation can be written in terms of the structure coefficients as
\begin{equation*}
        \ad^*_X(\zeta)(Y)  = \zeta([X, Y]) = C_{ij}^kX^iY^j\zeta_k \implies 
        (\ad^*_X(\zeta))_j = C_{ij}^k X^i \zeta_k
\end{equation*}
Therefore the continuous equations of motion are $\dot{\zeta}_i = C_{ij}^k (\zeta^T)^i \zeta_k $.

For the discrete jumps, let $\Sigma = \t{SL}_n(\R)$ be the impact surface, and note that $\t{SL}_n(\R) = s^{-1}(0)$, where $s(A) = \det(A) - 1$. Let $\adj(A)$ denote the coadjoint of $A$, so $ds|_A = \adj(A)^T$. Then, since $\Sigma = \t{SL}_n(\R)\cdot I_n $ and $g_0 = I_n$, the formula \eqref{eq:Delta_zeta} for this case gives
\begin{equation*}
    \Delta \zeta = -\frac{2}{n}\tr(\zeta)I_n.
\end{equation*}

\begin{remark}
    Let $f(\zeta) = \t{tr}(\zeta)$. Then $df = I_n d\zeta$ which commutes with any $dh$ for any Hamiltonian $h$. Thus, the trace is a Casimir for the bracket on $\mathfrak{gl}_n^*$. Moreover, across impacts $\t{tr}(\zeta^+) = -\t{tr}(\zeta^-)$.
\end{remark}
 The extra variable that keeps track of whether impact occur is $q = \det(A) - 1$. Its equations of motion are given by \eqref{eq:ILP_cont}
\begin{align*}
    \begin{cases}
        \dot{q} = (q + 1)\t{tr}(\zeta)\\
        q^+ = q^- 
    \end{cases}
\end{align*}
The $q$ variable system can also be solved exactly, using the fact that $\tr(\zeta) = C$ is always constant; we obtain $q(t) = q_0e^{Ct} -1$. 

Therefore, to determine when impacts occur we need to keep track of two scalars: the trace of the lie algebra element, and the determinant of the group element. Independently on the dimension of the configuration space, impacts are completely determined by
\begin{gather}\label{eq:2Dmatrixgroups}
    \begin{cases}
        q(t) = q_0e^{Ct} - 1 \\
        C(t) = C
    \end{cases}
    \hspace{.5cm}\t{if }q \neq 0; \hspace{3cm}\begin{cases}
        q^+ = q^-\\
        C^+ = -C^- 
    \end{cases} \hspace{.5cm}\t{if } q = 0.
\end{gather}
No matter the initial $\tr(\zeta)$ and $\det(A)$, there is at most one impact on any trajectory (see Figure \ref{fig:matrix_groups}).
\begin{figure}
    \centering
    \includegraphics[scale = 0.15]{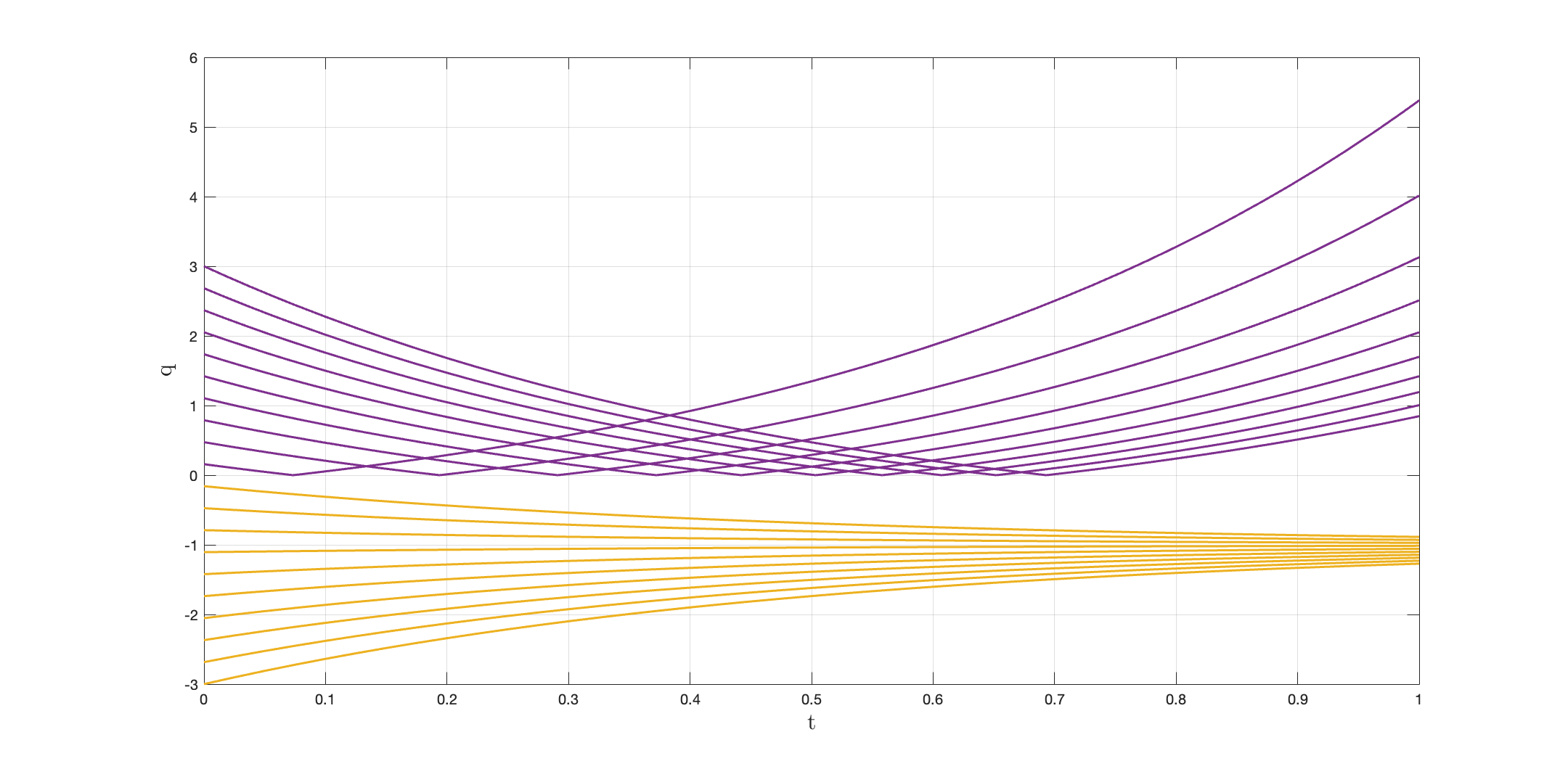}
    \caption{Sample trajectories of the 2-dimensional system \eqref{eq:2Dmatrixgroups} which determines exactly when impacts happen for $n$-dimensional Lie-Poisson impact system with $G = \t{GL}_n$, and $\Sigma = \t{SL}_n$. Initially, $\tr(\zeta) = -2$, and $q = \t{det}(A) - 1$ is sampled uniformly from $[-3, 3]$.}
    \label{fig:matrix_groups}
\end{figure}

Let $n = 2$, and denote the components of $dh$ by $(\alpha^1 \ \alpha^2 \ \alpha^3 \ \alpha^4)$. Then $\ad_{dh}^*$ is a $4 \times 4$ matrix depending on $dh$
\begin{align*}
    \ad_{dh}^* 
    = \begin{pmatrix}
        0 & -\alpha^3 & \alpha^2 & 0 \\
        -\alpha^2 & \alpha^1 - \alpha^4 & 0 & \alpha^2 \\
        \alpha^3 & 0 & \alpha^4 - \alpha^1 & -\alpha^3 \\
        0 & \alpha^3 & -\alpha^2 & 0  
    \end{pmatrix}
\end{align*}

For $h = \frac{1}{2}\tr(\zeta\zeta^T)$, $dh = \sum_i\zeta_\Id\zeta_i$ and the equations of motion are
\begin{equation}\label{eq:cont_eom_matrix_groups}
\begin{split}
    \begin{pmatrix}
        \dot{\zeta}_1 & \dot{\zeta}_2 & \dot{\zeta}_3 & \dot{\zeta}_4
    \end{pmatrix} = &
    \begin{pmatrix}
        \zeta_1 & \zeta_2 & \zeta_3 & \zeta_4
    \end{pmatrix}\begin{pmatrix}
        0 & -\zeta_3 & \zeta_2 & 0 \\
        -\zeta_2 & \zeta_1-\zeta_4 & 0 & \zeta_2 \\
        \zeta_3 & 0 & \zeta_4-\zeta_1 & -\zeta_3 \\
        0 & \zeta_3 & -\zeta_2 & 0
    \end{pmatrix} \\
     = & \begin{pmatrix}
         - \zeta_2^2 + \zeta_3^2 & (\zeta_1 - \zeta_4)(\zeta_2 - \zeta_3) & (\zeta_1 - \zeta_4)(\zeta_2 - \zeta_3) & \zeta_2^2 - \zeta_3^2
     \end{pmatrix}
    \end{split}
\end{equation}
and discrete transitions follow the rule $ \zeta^+ =
    \begin{pmatrix}
        -\zeta_4^- & \zeta_2^- & \zeta_3^- & -\zeta_1
    \end{pmatrix}.
    % (-\zeta_4^- \ \ \zeta_2^- \ \ \zeta_3^- \ \ -\zeta_1)
$

The system \eqref{eq:cont_eom_matrix_groups} can be solved exactly! Let $\zeta_1 + \zeta_4 = \t{tr}(\zeta) = C$ and $\zeta_2 - \zeta_3 = D$. Then, $\dot{C} = \dot{D} = 0$ so $C$ and $D$ are constants. Plugging this back into \eqref{eq:cont_eom_matrix_groups}, the equations of motion can be reduced to a 2-dimensional linear system

which can be solved analytically to give the following solution to \eqref{eq:cont_eom_matrix_groups}:
\begin{align*}
    \zeta_1 &=  D^2t + \zeta_1^0\cos(2Dt) + \zeta_3^0\sin(2Dt),
    \hspace{1.91cm} \zeta_3 =  - CDt + \zeta_1^0\sin(2Dt) + \zeta_3^0\cos(2Dt), \\
    \zeta_2 & =D -  CDt +\zeta_1^0\sin(2Dt) + \zeta_3^0\cos(2Dt),
    \hspace{1cm} 
    \zeta_4 = C -  D^2t - \zeta_1^0\cos(2Dt) - \zeta_3^0\sin(2Dt),
\end{align*}
where $D = \zeta_3^0 - \zeta_2^0$ and $C = \zeta_1^0 + \zeta_4^0$. 

Now, consider the problem of finding the flow of some density under the dynamics. For this we need to find \eqref{eq:chikos} for $\t{GL}_2(\R)$. Given the measure $\mu = d\zeta_1 \wedge d\zeta_2 \wedge d\zeta_3 \wedge d\zeta_4\wedge dq$, we obtain that the hybrid Jacobian is $\mathcal{J}_\mu^X = -1$ and the divergence $\t{div}_\mu(X) = \t{tr}(\zeta) = C$. Hence, \eqref{eq:chikos} becomes
\begin{gather*}
    \begin{cases}
        \begin{split}
            \partial_t u + (2D\zeta_3 + D^2)\partial_{\zeta_1}u + (2D \zeta_1 - DC) \partial_{\zeta_2}u + (2D \zeta_1 - DC) \partial_{\zeta_3}u {\color{white}{} = - Cu} \\
            -(2D\zeta_3 + D^2)\partial_{\zeta_4}u + C( q + 1) \partial_qu 
            = - Cu
        \end{split}
        & \t{if }q \neq 0; \\
        u(t^+, \zeta^+, 0) = - u(t^-, \zeta^-, 0) & \t{if }q = 0.\\
    \end{cases}
\end{gather*}

%% file: Applications/Disease_Spread.tex
\label{sec:SIR}
Consider an SIR model as in \cite{SIRS}, governed by the following equations of motion
\begin{equation*}
    \begin{cases}
        \dot{S} = \mu N - \dfrac{\beta S I}{N } - \mu S \\
        \,\dot{I} = \dfrac{\beta S I}{N } - \gamma I - \mu I - \delta I \\
        \dot{R} = \gamma I - \mu R
    \end{cases}
\end{equation*}
where $\mu$ represents the natural birth and death rate, $\gamma$ is the recovery rate, $\beta$ is the rate of infection, and $\delta$ is the death rate due to the disease. Now, consider the following scenario: after the number of infected individuals reaches a specific threshold fraction of the population $\alpha$, the government decides to intervene and treats a fraction $f$ of the infected individuals all at once. After this, the government lets the disease evolve again with its usual dynamics. Mathematically, the intervention is represented by the impact map
\begin{equation*}
    \begin{cases}
        S^+ = S^-\\
        I^+ = (1 -f)I^-\\
        R^+ = R^- + fI^- 
    \end{cases}
\end{equation*}
and impact surface $\Sigma = \R_+ \times \{\alpha N\}\times \R_+ \cong \R^2 \subset \R^3 $.

We are interested in the long term behaviour of this population. What percent of the individuals survive this disease? What are the best values for $\alpha$ and $f$ under which $\lim_{t \to \infty} N(t)$ is maximized?  

One interesting aspect about this example is that the hybrid Jacobian is no longer $1$, which is not surprising since the equations do not come from a Hamiltonian. 
\begin{align*}
    \Delta^* i_X \mu = - (1 - f) I \left(\frac{\beta S}{N} - \gamma - \mu - \delta\right)dS \wedge dR = (1 - f)\iota^*_S i_X \mu
\end{align*}
Hence, the hybrid Jacobian $\mathcal{J}_\mu^X(\Delta) = 1 - f$, so the hybrid Frobenius Perron-PDE \eqref{eq:chikos} becomes
\begin{gather*}
    \frac{\partial u}{\partial t} + \left(\mu N - \frac{\beta SI}{N} - \mu S \right)\frac{\partial u}{\partial S} + \left(\frac{\beta SI}{N}  - (\gamma + \mu + \delta) I\right)\frac{\partial u}{\partial I} \hspace{5cm} {} \\ {} \hspace{4cm}
    +(\gamma I - \mu R)\frac{\partial u}{\partial R} - \left(\beta \frac{I - S}{N } + 3\mu + \delta + \gamma\right)u = 0, \qquad I \neq \alpha N;\\
    {\color{white}I = \alpha N.,}\hspace{1.775cm} u(t^+,S, \alpha N ( 1- f), R) = (1 - f)u(t^-, S, \alpha N, R), \hspace{1.775cm} I = \alpha N.
\end{gather*}Assume we are only interested in the proportion of susceptible, infected and recovered individuals, and not in the total size of the population. Under this assumption we can further reduce the dimension by normalization. Let $i = I/N$, $s = S/N$, and $r = R/N$. Then $i + r + s = 1$,  the hybrid equations of motion become
\begin{align*}
    \begin{split}
        \dot{s} = & \ \mu - \mu s - (\beta - \delta) si\\
        \dot{i} = & \ \beta  si + \delta i^2 - (\gamma + \mu + \delta )i
    \end{split}
    \qquad \t{if } i \neq\alpha; \qquad \qquad 
    \begin{split}
        s^+ = & \ s^-\\
        i^+ = & \ (1 -f)i^-
    \end{split}
     \qquad \t{if } i  = \alpha,
\end{align*}
and the hybrid Frobenius-Perron PDE is
\begin{gather}\label{eq:SIR}
    \begin{split}
        \frac{\partial u}{\partial t} + (\mu - \mu s - (\beta - \delta) si )\frac{\partial u}{\partial s } + (\beta  si + \delta i^2 - (\gamma + \mu + \delta )i )\frac{\partial u}{\partial i }\qquad \qquad  \\ 
        = (2\mu - \beta(s - i) + \gamma + \delta - 3\delta i)\, u,
    \end{split}\\
    u(t^+, s, \alpha(1 - f)) =  \frac{- \beta\alpha s + \delta \alpha^2 - (g + \mu + \delta)\alpha}{- \beta\alpha s ( 1-f)+ \delta \alpha^2( 1-f)^2 - (g + \mu + \delta)\alpha( 1-f)}u(t^-, s, \alpha).\nonumber
\end{gather}
Consider an initial density $\rho(s, i) = e^{-(s - 0.5)^2 - 10(i - 1 + s)^2}$. Figure \ref{fig:simulations SIR} shows the evolution of this density under the PDE given in \eqref{eq:SIR}. Table \ref{table:SIR_param} shows the values of the parameters used during the simulation. Note that since $ i + s = 1 - r \leq 1$ the configuration space for this model is the lower triangle of $[0, 1]\times[0, 1]$.
\begin{figure}[!ht]
    \centering
     \begin{subfigure}[t]{0.32\textwidth}
         \centering
         \includegraphics[width=\textwidth]{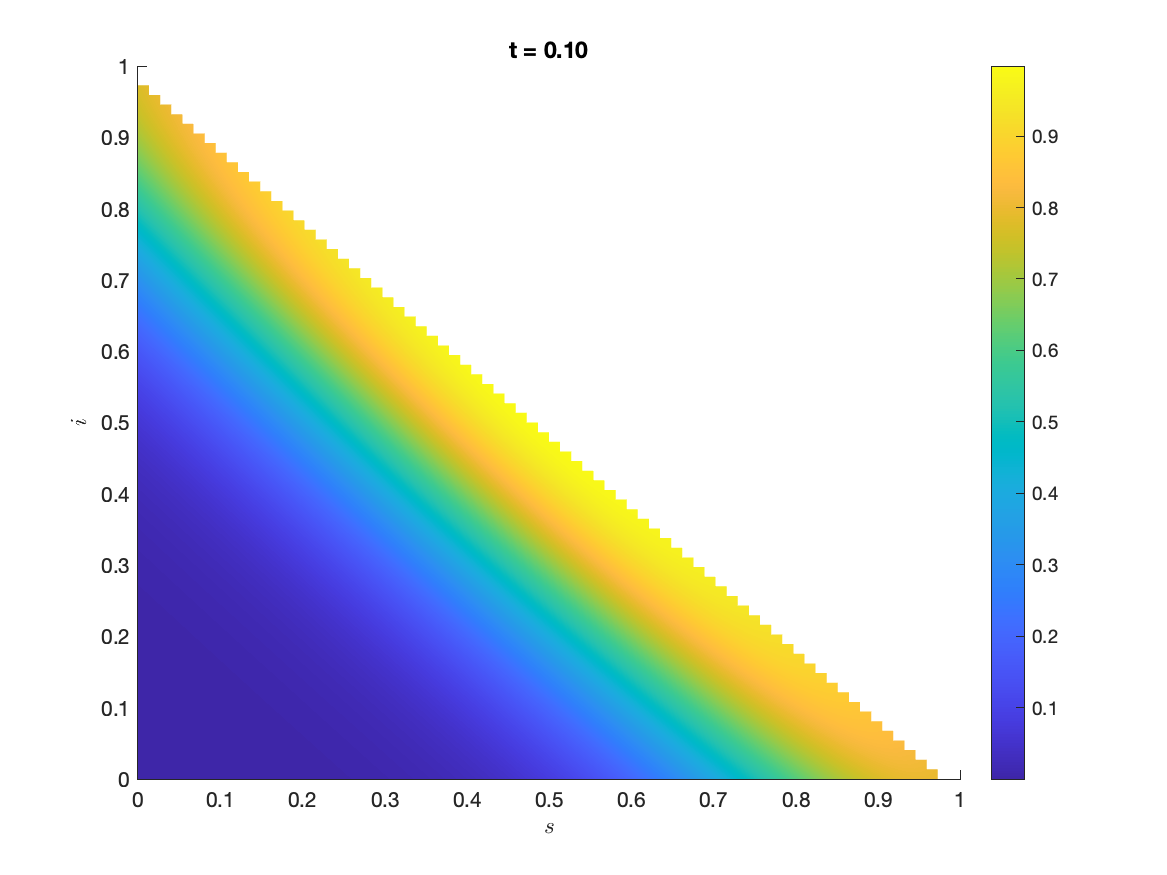}
         \caption{The initial density of the population, which is concentrated along the line $ s + i = 1$. }
         \label{fig:SIR_init_density}
     \end{subfigure}
     \hfill
      \begin{subfigure}[t]{0.32\textwidth}
         \centering
         \includegraphics[width=\textwidth]{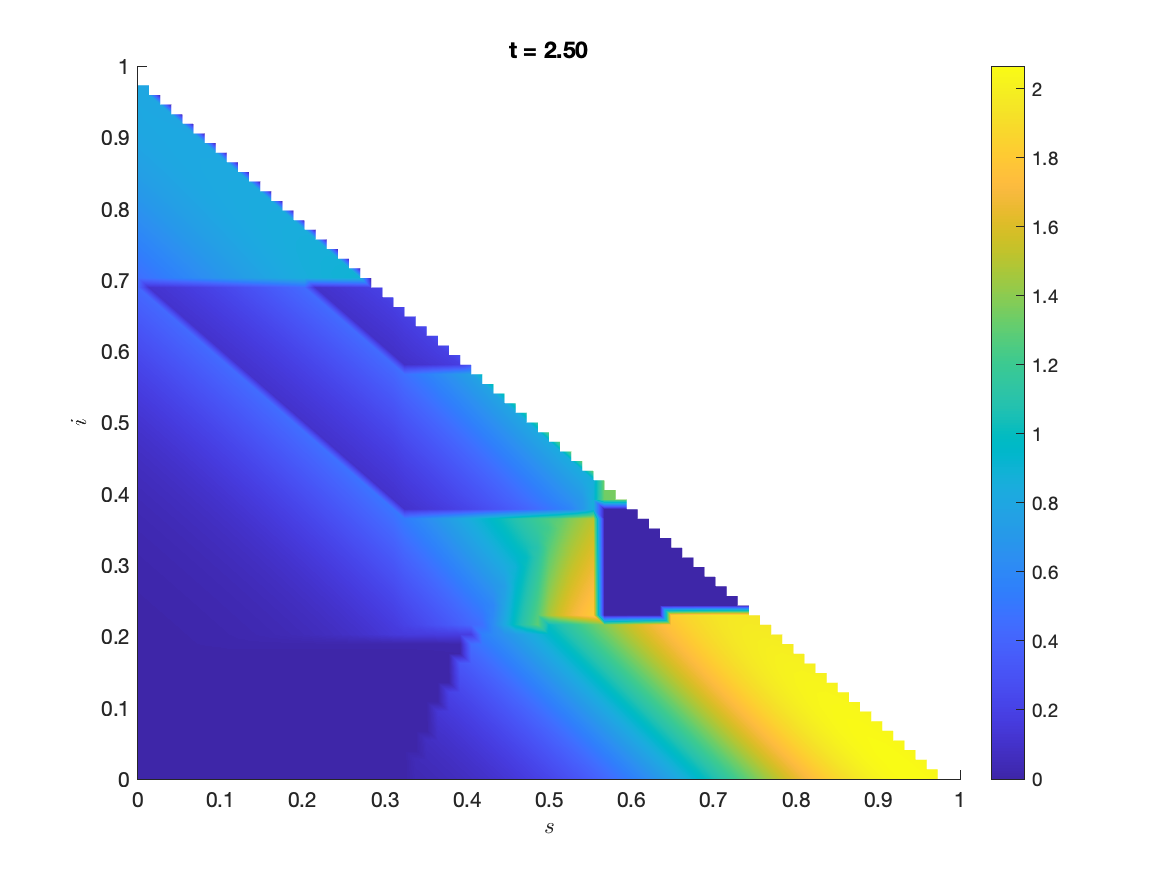}
         \caption{The intervention decreases the number of infected individuals by $30\%$ when it reaches $70\%$. If $I$ is initially greater than $70\%$, then the population becomes infected and eventually dies.}
         \label{fig:SIR 2.5}
     \end{subfigure}
     \hfill
     \begin{subfigure}[t]{0.32\textwidth}
         \centering
         \includegraphics[width=\textwidth]{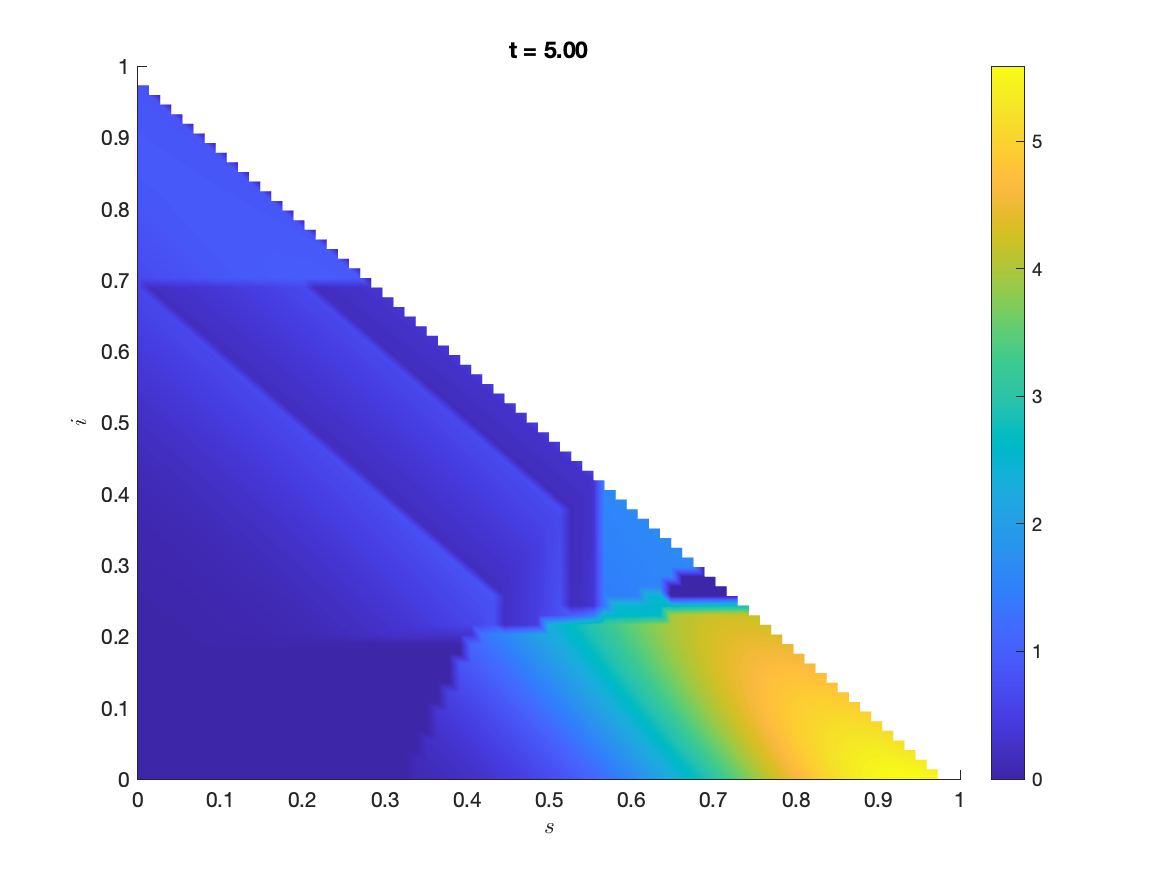}
         \caption{ At $t = 5$ these effects are exacerbated and we obtain a population where the percentage of infected individuals is likely to be less than $20\%$. }
         \label{fig:SIR 5}
     \end{subfigure}
     \hfill
    \caption{Evolution of initial density $\rho = e^{-(s - 0.5)^2 - 10(i - 1 + s)^2}$ under the hybrid transfer PDE \eqref{eq:SIR} corresponding to the SIR model with human intervention.}
    \label{fig:simulations SIR}
\end{figure}

\begin{table}[!ht]
    \centering
    \begin{tabular}{c|c|c|c|c|c|c}
        Parameter & $\beta$ &  $\gamma$ & $\delta$ &   $\mu $ & $f$ & $\alpha$ \\
        \hline
        Value & 0.1 & 0.1 & 0.2 & 0.1 & 0.3 & 0.7
    \end{tabular}
    \caption{Values of the parameters of the SIR model used in the simulations shown in Figure \ref{fig:simulations SIR}.}
    \label{table:SIR_param}
\end{table}

Although the simulations show that the percentage of infected population is likely to be low, this does not imply that the intervention was successful. It might be the case that the size of the  population decreases so much, that the individuals are mostly recovered.  In particular for initial percentage of infected individuals greater than 0.7 the population quickly becomes totally infected and the number of individuals $N$ decreases drastically.
In the future we would like to study in detail the entire 3 dimensional dynamics and the effect of the disease on the population count. 

%% file: Conclusions/Conclusions.tex
This paper contributes to the theoretical understanding of hybrid systems in two ways. First, it presents a continuity equation for the Frobenius-Perron operator (Theorem \ref{thm:general chikos}). This simplifies the problem of computing the time evolution of a given density under the underlying dynamics, to the quest of finding the solution of a partial differential equation. Second, it develops a theory of reduction for hybrid systems where the state space is a Lie group (Theorem \ref{thm:Impact_Reduction}). This is particularly useful in high dimensional systems, where even solving equation \eqref{eq:chikos} becomes difficult. 
%By combining these two results we are able to faster compute the Frobenius Perron operator for 2 dimensional (the bouncing ball \ref{subsec:BB}), 3 dimensional (the SIR model \ref{subsec:SIR}), 6 dimensional (the Chaplygin sleigh \ref{subsec:chapsleigh_angle}) and even 8 dimensional (matrix groups \ref{subsec:matrix_groups}) systems.

We identify several future research directions.  Although the Hybrid Frobenius-Perron PDE offers great theoretical insight into the analysis of the transfer operator, the numerical methods employed for solving it should be further improved. A particular issue comes into play due to the fact that characteristics do not always stay within the predefined grid. In our current approach, when that happens, we integrate over the entire characteristic up to time $0$. However, this increases the computation time and becomes problematic when a large number of the characteristics suffer from this issue. Therefore, better methods to deal with this situation are needed. Moreover, different approaches for solving \eqref{eq:chikos} such as finite differences and finite volumes should be studied.

In order for Hybrid Lie-Poisson reduction to be valid, the impact surface needs to be the right coset of a codimension 1 normal Lie subgroup of the configuration space. These proprieties restrict the types of hybrid systems we can consider for reduction. In the future, we wish to study the instances where the conditions mentioned above are not satisfied. Even if reduction to $n + 1$ dimensions might not be plausible without these conditions, there might still be some $k<n$ such that the system can be reduced to $n + k$ dimensions. 

Finally, we are interested in the inverse problem, where we would like to find the underlying hybrid dynamics, using the evolution of a given density under the Frobenius-Perron operator. Suppose we can sample points from the flow of some initial distribution at any given time. These points will give an approximation to the solution of the \eqref{eq:chikos}. Can we tell whether the dynamics is hybrid just by looking at these points? Can we find an analytical expression for the ODEs that generate this dynamics?

%% file: Appendices/Proofs.tex
\section{Proof of the transverse property}\label{appendix:A}
\begin{lemma}[Transverse property]
    \label{lemma:Transverse Property Proof}
    Let $\Delta_*^X \colon TM|_\mathcal{S}\to TM|_{\Delta(\mathcal{S})}$ be the augmented differential and let $\iota_\mathcal{S} \colon \mathcal{S}\hookrightarrow M$ and $\iota_{\Delta(\mathcal{S})} \colon \Delta(\mathcal{S})\hookrightarrow M$ denote the respective inclusion maps. Then, 
    \begin{equation}\label{eqn:Transverse Property Equality in Proof}
        \det_{\mu} \Delta^X_* = \ \ \ \ \det_{\mathclap{\ \ \ \ \ \ \iota^{\!*}_\mathcal{S} i_X \mu \to \iota^{\!*}_{\Delta(\mathcal{S})} i_X \mu }} (\Delta_*) \ \ \ \ \, .
    \end{equation}
\end{lemma}
\begin{proof}
    Fix $x \in \mathcal{S}$ and let $v^i \in T_x M$ for $i = 1, \ldots, n$ where $n = \dim M$.  Then, we can decompose $v^i$ into $v^i = u^i +\alpha^i X_x$ where $u^i \in T_x\mathcal{S}$ and $\alpha^i \in \R$. By the definition of $\det_\mu$, as stated in Definiton (3.4), we know that 
    \begin{align}
       \big(\!\det_\mu \Delta^X_*\big)\mu( v^1, \ldots, v^n)
       = & \, \mu_{\Delta(x)}\big( \Delta^X_* v^1, \ldots, \Delta^X_*v^n\big) \label{eqn:A.1 Proof start}
    \intertext{Expanding $v^i$ through the mentioned decomposition: $v^i = u^i +\alpha^i X_x$}
        = & \, \mu_{\Delta(x)}\left(\Delta^X_*\big(u^1 + \alpha^1 X_x\big), \ldots, \Delta^X_*\big(u^n + \alpha^n X_x\big)\right) \nonumber
    \intertext{Since is $\Delta^X_*$ a linear map and $\Delta^X_*(X_x) = X_{\Delta(x)}$ by definition}
        = & \, \mu_{\Delta(x)}\big(\Delta_*\left(u^1\right) + \alpha^1 X_{\Delta(x)}, \ldots, \Delta_*\left(u^n\right) + \alpha^n X_{\Delta(x)}\big) \nonumber
    \intertext{By the linearity of $k$-forms, we can expand the previous expression as follows}
        % & \qquad \mu_{\Delta(x)}\big(\Delta_*(u^1) + X_{\Delta(x)},\Delta_*(u^2_x) + X_{\Delta(x)}, \ldots, \Delta_*(u^n) + X_{\Delta(x)}\big) \\
        = & \, \mu_{\Delta(x)}\big(\Delta_* u^1, \Delta_* u^2_x, \ldots, \Delta_* u^n\big) \label{eqn:vanishes}\\
        + & \,  \sum_{k = 1}^n (-1)^{k - 1} \alpha^k \mu_{\Delta(x)}\big( \Delta_* u^1_x, \ldots, \widehat{\Delta_* u^k}, X_{\Delta(x)}, \ldots, \Delta_* u^n\big) \nonumber
        % & \quad + \sum_{k = 1}^n \mu_{\Delta(x)}\big(\Delta_*(u^1), \ldots , \Delta_*(u^{k -1}_x), X_{\Delta(x)}, \Delta_*(u^{k +1}_x), \ldots, \Delta_*(u^n)\big) \\
    \intertext{Since $T_x\mathcal{S}$ is $n - 1$ dimensional and each $u^i \in T_x\mathcal{S}$, the set $\{u^i\}_{i=1}^n$ is linearly dependent. Then, since $k$-forms with linearly dependent inputs are equal to zero, we know that \eqref{eqn:vanishes} vanishes, so
    }
        = & \, \sum_{k = 1}^n (-1)^{k - 1}\alpha^k\mu_{\Delta(x)}\big(\Delta_* u^1, \ldots, \widehat{\Delta_* u^k}, X_{\Delta(x)}, \ldots, \Delta_* u^n\big)\nonumber
    \intertext{Rewriting using the fact that $\mu_{\Delta(x)}$ is an alternating product}
        = & \, \sum_{k = 1}^n (-1)^{k - 1} \alpha^k \mu_{\Delta(x)}\big(X_{\Delta(x)}, \Delta_* u^1_x, \ldots, \widehat{\Delta_* u^k}, \ldots, \Delta_* u^n\big) \nonumber \\
        = & \, \sum_{k = 1}^n (-1)^{k - 1} \alpha^k i_X \mu_{\Delta(x)}\big(\Delta_* u^1, \ldots, \widehat{\Delta_* u^k}, \ldots, \Delta_* u^n \big) \nonumber
    \end{align} 
    Before we solve for $\det_\mu \Delta^X_*$, we will rewrite $\mu(v^1, \ldots, v^{n - 1})$, the term that $\det_\mu \Delta^X_*$ is multiplied by in \eqref{eqn:A.1 Proof start}. Yet again using the decomposition $v^i = u^i + \alpha^i X_x$ and a process similar to that of rewriting the right-hand side of \eqref{eqn:A.1 Proof start}, yields
    \begin{align*} 
        \mu(v^1, \ldots, v^{n - 1}) 
         = & \sum_{k = 1}^n \mu_{x}(u^1, \ldots\widehat{, u^k}, \alpha^k X_x, \ldots, u^n) \\
         = & \sum_{k = 1}^n (-1)^{k - 1} \alpha^k i_X \mu(u^1, \ldots \widehat{, u^k}, \ldots, u^n).
    \end{align*}
    Now, we can isolate $\det_\mu \Delta^X_*$ from all other terms in \eqref{eqn:A.1 Proof start},
    \begin{gather}\label{smiley}
        \det_\mu \Delta^X_* = \frac{\sum_{k = 1}^n (-1)^{k - 1} \alpha^k  i_X \mu_{\Delta(x)}(\Delta_* u^1\ldots, \widehat{\Delta_* u^k}, \ldots, \Delta_* u^n)}{\sum_{k = 1}^n (-1)^{k - 1} \alpha^k i_X \mu(u^1, \ldots \widehat{, u^k}, \ldots, u^{n - 1})}.
    \end{gather}
    \noindent We have finished deriving the left-hand side of \eqref{eqn:Transverse Property Equality in Proof}. For computing the right-hand side, we must carefully choose which vectors we plug into the form determinant definition. With a later part of the proof in mind, we modify the earlier set of $u^k$ vectors and choose 
    \begin{align*}
        \begin{cases}
            \{\alpha^k u^1, \ldots, \widehat{u^k}, \ldots, u^n\} & \t{for } k \in \set{2, 3, \ldots, n}; \\
            \set{\alpha^1 u^2_x, u^3_x, \ldots, u^n} & \t{for } k = 1,
        \end{cases}
    \end{align*}
    both of which we will denote by the top case since the proofs are identical. The reasoning behind this choice becomes apparent in \eqref{eqn:summing}. Now computing the the right-hand side of \eqref{eqn:Transverse Property Equality in Proof},
    \begin{align*}
        & \ \ \ \ \det_{\mathclap{\ \ \ \ \iota^{\!*}_\mathcal{S} i_X \mu \to \iota^{\!*}_{\Delta(\mathcal{S})} i_X \mu }} (\Delta_*) \ \ \ \alpha^k i_X\mu\big( u^1, \ldots\widehat{, u^k}, \ldots, u^n\big) 
        && \t{} \\ 
        & \qquad = \quad \ \det_{\mathclap{\ \ \ \ \iota^{\!*}_\mathcal{S} i_X \mu \to \iota^{\!*}_{\Delta(\mathcal{S})} i_X \mu }} (\Delta_*) \ \ \ \iota^{\!*}_\mathcal{S}(i_X\mu)\big(\alpha^k u^1, \ldots\widehat{, u^k}, \ldots, u^n\big) 
        && \t{Since }u^k \in T_x \mathcal{S} \subset T_x M \\ 
        & \qquad = \Delta^* \iota^{\!*}_{\Delta(\mathcal{S})}(i_X \mu)\big(\alpha^ku^1, \ldots, \widehat{\Delta_* u^k}, \ldots, u^n\big) 
        && \t{Definition of the form determinant} \\ 
        & \qquad = \iota^{\!*}_{\Delta(\mathcal{S})}(i_X \mu_{\Delta(x)}) \big(\alpha^k \Delta_* u^1, \ldots, \widehat{\Delta_* u^k}, \ldots, \Delta_* u^n\big) 
        && \t{Applying the pullback of $\Delta$} \\ 
        & \qquad = i_X \mu_{\Delta(x)} \big(\alpha^k \Delta_* u^1, \ldots, \widehat{\Delta_* u^k}, \ldots, \Delta_* u^n\big) 
        && \t{Since $\Delta_* u^k \in T_{\Delta(x)} (\Delta(\mathcal{S})) \subset T_{\Delta(x)} M$}
    \end{align*}
    Since equality still holds if we multiply both sides by $(-1)^{k-1}$, we can then sum over $k$ to get
    \begin{align}
        & \quad\det_{\mathclap{\ \ \ \ \iota^{\!*}_\mathcal{S} \alpha^k i_X \mu \to \iota^{\!*}_{\Delta(\mathcal{S})} i_X \mu }} (\Delta_*) \ \ \ \ \sum_{k = 1}^n (-1)^{k - 1} i_X \mu_{\Delta(x)} \big(u^1, \ldots\widehat{, u^k}, \ldots, u^n\big) \nonumber \\
        & \hspace{4cm} = \sum_{k = 1}^n (-1)^{k - 1} \alpha^k i_X \mu_{\Delta(x)}\big(\Delta_* u^1, \ldots, \widehat{\Delta_* u^k}, \ldots, \Delta_* u^n\big). \label{eqn:summing}
    \end{align}
    Finally, isolating the determinant and looking back at previous equations, we get that
    \begin{align*}
        \det_{\mathclap{\ \ \ \ \iota^{\!*}_\mathcal{S} i_X \mu \to \iota^{\!*}_{\Delta(\mathcal{S})} i_X \mu }} (\Delta_*) \ \ \ 
        = \frac{\sum_{k = 1}^n (-1)^{k - 1} i_X \mu_{\Delta(x)}\big(\Delta_* u^1, \ldots, \widehat{\Delta_*(u^k)}, \ldots, \Delta_* u^n)}{\sum_{k = 1}^n (-1)^{k - 1} i_X \mu \big(u^1, \ldots\widehat{, u^k}, \ldots, u^n\big)} 
        = \eqref{smiley} 
        = \det_{\mu}(\Delta_*^X).
    \end{align*}
\end{proof}

%%%%%%%%%%%%%%%%%%%%%%%%%%%%%%%%%%%%%%%%%%%%%%%%%%%%%%%%%%%%%%%%%%%%%%%%%%%%%%%%%%%%%%%%%%%%%%%%%%%
\section{Proof of the form determinant and inverse form determinant relation}\label{appendix:B}
\begin{lemma}[Relation between form determinant and inverse form determinant]
\label{lemma:Relation Between Form Determinant and Inverse Form Determinant Proof}
Given a diffeomorphism $f \colon M \to N$ and volume forms $\mu \in \Omega^n(M)$ and $\eta \in \Omega^n(N)$, the relation between $\ds \det_{\mu \to \eta} f_*$ and $\ds \det_{\eta \to \mu } f_*^{-1}$ is given by
\begin{align*}
    \det_{\eta \to \mu}(f_*) = 
    \frac{1}{\ds\det_{\mu \to \eta}(f_*^{-1}) \circ f } \qquad \text{or} \qquad
    \det_{\mu \to \eta}(f_*^{-1}) = \frac{1}{\ds\det_{\eta \to \mu}(f_*)} \,\circ\, f^{-1}.
\end{align*}
\end{lemma}
\begin{proof} 
Let $\mu \in \Omega^n(M)$ and $\eta \in \Omega^{n}(N)$ be volume forms. Recall that, by definition (3.4), the determinant of $f_*^{-1} \colon TN \to TM$ is the $C^{\infty}(N)$ function $\det_{\eta \to \mu }(\Delta^{-1}_*)$ such that 
% , the determinant $\det_{\mu \to \eta }(f_*)$ is defined to be the unique $C^{\infty}(M)$ function such that $\det_{\mu \to \eta }(f_*)\mu = f^*\eta$, so by definition, the determinant of $f_*^{-1}$ for $\det_{\eta \to \mu }(f^{-1}_*) \in C^{\infty}(N)$ such that
\begin{align}
    \det_{ \eta \to \mu}(f_*^{-1})\eta = (f^{-1})^*\mu. %\label{}
\end{align}
Since $\mu$ and $\eta$ are non-singular, their form determinants are nonzero, so we can divide by the determinants to get
\begin{align*}
    \mu 
    = \big(\!\det_{\mu \to \eta} (f_*) \big)^{-1} f^*\eta 
    \quad \text{ and } \quad \eta 
    = \big(\!\det_{\mu \to \eta} (f_*^{-1}) \big)^{-1}( f^{-1})^*\mu.
\end{align*} 
As the above functions are inverses of each other, they have the property that $f^{-1} \circ f = \text{Id}_M$. %, where $\text{Id}_M$ is the identity on $M$. 
Taking the pullback of both sides of this property gives us
\begin{equation}\label{eqn:Pushforward Identity}
    (f^{-1} \circ f)^* = f^* \,\circ\, (f^{-1})^* = \text{Id}^*_M,
\end{equation}
Now that we have the pullbacks written out, we get
\begin{align*}
    \mu 
    = & \, f^* \circ (f^{-1})^*\mu 
    && \t{By \eqref{eqn:Pushforward Identity}} \\
    = & \, f^*\big( \det_{\eta \to \mu}(f^{-1}_*)\eta \big) 
    && \t{By the defintion of the form determinant of } (f^{-1})^* \\
    % = & \, \big(\det_{\eta \to \mu}(f^{-1}_*) \,\circ\, f \big) && 
    = & \, \big(\det_{\eta \to \mu}(f^{-1}_*) \circ f \big)(f^*\eta)
    && \t{Property of pullback on top forms \cite[Proposition~14.20]{Lee2003}} \\
    = & \, \big(\det_{\eta \to \mu}(f^{-1}_*) \circ f \big) \det_{\mu \to \eta }(f_*) \mu
    && \t{By the defintion of } \det_{\mu \to \eta}f_*
\end{align*}
% \begin{align*}
%    & f^*\big( \det_{\eta \to \mu}(f^{-1}_*)\eta \big) =  \\
%    \implies & \big(\det_{\eta \to \mu}(f^{-1}_*) \,\circ\, f \big)(f^*\eta) = \mu \\
%    \implies & \big(\det_{\eta \to \mu}(f^{-1}_*) \,\circ\, f \big) \det_{\mu \to \eta }(f_*) \mu = \mu,
% \end{align*}
Finally, dividing gives us
\begin{align*}
    \det_{\mu \to \eta}(f_*) = \frac{1}{\displaystyle{\det_{\mu \to \eta}}(f_*^{-1}) \circ f } 
    \qquad \text{and} \qquad 
    \det_{\eta \to \mu }(f_*^{-1}) = \frac{1}{\displaystyle{\det_{\mu \to \eta}}(f_*)} \circ f^{-1}.
\end{align*}
\end{proof}

%% file: Appendices/Hybrid_Jacobian_for_Conservative_Holonomic_Systems.tex
For unconstrained or holonomically constrained systems with a stationary impact surface, $\mathcal{J}^X_\mu(\Delta) = c^4$. We show this by explicitly computing $\Delta^* i_X \mu = \mathcal{J}^X_\mu(\Delta)\iota^* i_X\mu$ which turns out to be a relatively simple computation since all but one 1-form is equal to zero. 
\begin{theorem}\label{thm:inelastic_Jacobian}
    Let $H:T^*Q\to\mathbb{R}$ be a natural Hamiltonian and $\mathcal{S}\subset Q$ the impact surface. Let $\Delta$ be the inelastic impact map given by
    \begin{equation*}
        \Delta(x,p) = \left( x, p - (1+c^2)\frac{p(\nabla h)}{dh(\nabla h)}dh \right).
    \end{equation*}
    Then, the hybrid Jacobian is simply $\mathcal{J}_{\omega^n}^{X_H}(\Delta) = c^4$.
\end{theorem}
\begin{proof}  To calculate the hybrid Jacobian for inelastic impact systems without constraints, we must isolate $\mathcal{J}^X_\mu(\Delta)$ from $\Delta^* i_X \mu = \mathcal{J}^X_\mu(\Delta)\iota^* i_X\mu$, where $\mu = \omega^n$ and $\omega = dx^i\wedge dp_i$ is our symplectic form on $T^*Q$. First, calculating the exterior derivative of the Hamiltonian for later use
\begin{align*}
    dH = & \ \frac{\partial H}{\partial x^k}dx^k + \frac{\partial H}{\partial p_k}dx^ k\\
       = & \ \frac{\partial}{\partial x^k}\left( \frac{1}{2}g^{ij}p_ip_j + V(x)\right)dx^k + \frac{\partial }{\partial p_k}\left( \frac{1}{2}g^{ij}p_ip_j + V(x)\right)dp_k \\
       = & \ \left( \frac{1}{2}\frac{\partial g^{ij}}{\partial x^k}p_ip_j + \frac{\partial V(x)}{\partial x^k}\right)dx^k + \left( \frac{1}{2}g^{ij}\frac{\partial p_i}{\partial p_k} p_j + \frac{1}{2}g^{ij}p_i\frac{\partial p_j}{\partial p_k}\right)dp_k  \\
       = & \ \left( \frac{1}{2}\frac{\partial g^{ij}}{\partial x^k}p_ip_j + \frac{\partial V(x)}{\partial x^k}\right)dx^k + \left( \frac{1}{2}g^{ij}{\delta_i}^k p_j + \frac{1}{2}g^{ij}p_i{\delta_j}^k\right)dp_k \\
       = & \ \left( \frac{1}{2}\frac{\partial g^{ij}}{\partial x^k}p_ip_j + \frac{\partial V(x)}{\partial x^k}\right)dx^k + g^{kj} p_j dp_k 
\end{align*}
Now to calculate $i_X \omega$ for the equation for the hybrid Jacobian. Since $i_X \omega = dH$, we know that 
\begin{align*}
    i_X\omega^n = & \ i_X(\omega \wedge \ldots \wedge \omega )\\
    %dH \wedge \left(\sum_{i=1}^{n-1}(-1)^{i-1}\omega \right)\\
    = & \ dH \wedge \omega^{n-1}\\
    = & \ \left(\left( \frac{1}{2}\frac{\partial g^{ij}}{\partial x^k}p_ip_j + \frac{\partial V(x)}{\partial x^k}\right)dx^k + g^{kj} p_j dp_k  \right)\\
    & \quad \quad \quad  \quad \quad \quad \quad \quad \wedge \left(\sum_{k=1}^{n}dx^1\wedge dp_1 \wedge \ldots \wedge \widehat{dx^k \wedge dp_k} \wedge \ldots \wedge dx^n \wedge dp_n \right) \\
    = & \ \left( \frac{1}{2}\frac{\partial g^{ij}}{\partial x^k}p_ip_j + \frac{\partial V(x)}{\partial x^k}\right)\sum_{k=1}^{n}dx^1\wedge dp_1 \wedge \ldots \wedge dx^k \wedge \widehat{dp_k} \wedge \ldots \wedge dx^n \wedge dp_n  \\
    + & \ g^{kj} p_j \sum_{k=1}^{n}dx^1\wedge dp_1 \wedge \ldots \wedge \widehat{dx^k} \wedge dp_k \wedge \ldots \wedge dx^n \wedge dp_n
\end{align*} 
Since we have the quasi-smooth dependence property, we can choose the local coordinates about any point in the gurad so that $\mathcal{S} = \{(x^1,\ldots, x^n) \in \mathbb{R}^n : s(x^1,\ldots, x^n) =x^n \}$, which implies $ds = dx^n$. Consequentially, all terms in $i_X \omega$ that contain $dx^n$ are equal to zero on $\mathcal{S}$, so  
\begin{align}
    \iota^* i_X \omega^n = g^{nj} p_j dx^1\wedge dp_1 \wedge \ldots \wedge \widehat{dx^n} \wedge dp_n \label{eqn:iomegaapp}
\end{align}
and from this, we can see that the only term that we will need to calculate is $\partial H/\partial p_n$ since it is the only term that remains in \eqref{eqn:iomegaapp}. Writing the Hamiltonian, we have that
\begin{align*}
    H^+ - V(x) = & \ \frac{1}{2}g\left(p-(1-c^2) \frac{p(\nabla s)}{ds(\nabla s)}ds,p-(1-c^2) \frac{p(\nabla s)}{ds(\nabla s)}ds\right) \\
    % = & \ \frac{1}{2}g\left(p, p\right) - g\left(p,(1-c^2) \frac{p(\nabla s)}{ds(\nabla s)}ds\right) + \frac{1}{2} g\left((1-c^2) \frac{p(\nabla s)}{ds(\nabla s)}ds,(1-c^2) \frac{p(\nabla s)}{ds(\nabla s)}ds\right) \\
    = & \ \frac{1}{2}g\left(p, p\right) - (1-c^2) \frac{p(\nabla s)}{ds(\nabla s)}g\left(p,ds\right) + \frac{1}{2} \left((1-c^2) \frac{p(\nabla s)}{ds(\nabla s)}\right)^2g\left(ds,ds\right) 
\end{align*}
and taking the derivative of $H$ as needed,
\begin{align}
    \frac{\partial }{\partial p_n}(H^+ - V(x))= & \ \frac{1}{2}g\left(\frac{\partial p}{\partial p_n}, p\right) +  \frac{1}{2}g\left(p, \frac{\partial p }{\partial p_n}\right) - \frac{1}{2} \frac{\partial p(\nabla s) }{\partial p_n} \frac{1-c^2}{ds(\nabla s)}g\left(p,ds\right) \nonumber \\
    - & \ \frac{1}{2}(1-c^2) \frac{p(\nabla s)}{ds(\nabla s)}g\left(\frac{\partial p}{\partial p_n},ds\right) + \frac{1}{2} \left(1-c^2\right)^2\frac{\partial }{\partial p_n}\left( \frac{p(\nabla s)}{ds(\nabla s)}\right)^2g\left(ds,ds\right) \nonumber \\
    = & \ \frac{1}{2}g\left(\frac{\partial p}{\partial p_n}, p\right) + \frac{1}{2}g\left(p, \frac{\partial p }{\partial p_n}\right) - \frac{1}{2} \frac{\partial p(\nabla s) }{\partial p_n} \frac{1-c^2}{ds(\nabla s)}g\left(p,ds\right) \nonumber \\
    - & \ \frac{1}{2} (1-c^2) \frac{p(\nabla s)}{ds(\nabla s)}g\left(\frac{\partial p}{\partial p_n},ds\right) + \frac{1}{2} \left(1-c^2\right)^2\frac{\partial }{\partial p_n}\left( \frac{p(\nabla s)}{ds(\nabla s)}\right)^2ds(\nabla s) \nonumber \\
    = & \ g\left(\frac{\partial p}{\partial p_n}, p\right)  -  (1-c^2) \frac{p(\nabla s)}{ds(\nabla s)}g\left(\frac{\partial p}{\partial p_n},ds\right) + \frac{(1-c^2)^2}{ds(\nabla s)}\frac{\partial \left( p(\nabla s)\right)}{\partial p_n}. \label{eqn:terms}
\end{align}
Explicitly writing out the derivatives of \eqref{eqn:terms},
\begin{align*}
    g\left(\frac{\partial p}{\partial p_n}, p\right) = & \ g^{ij}\frac{\partial p_i}{\partial p_n}p_j = g^{ij} {\delta_i}^np_j = g^{nj}p_j = v^n \\
    \frac{\partial \left( p(\nabla s)\right)}{\partial p_n} = & \  \frac{\partial \left( g^{ij}p_i\partial_j h\right)}{\partial p_n} = g^{ij}\frac{\partial p_i}{\partial p_n}\partial_j h = g^{ij}{\delta_i}^n\partial_j h = g^{nj}\partial_j h = (\nabla s)^n.
\end{align*}
Now plugging those derivatives in,
\begin{align*}
    \implies \frac{\partial H^+}{\partial p_n} = & \ v^n - ( \nabla s )^n \frac{1-c^2}{ds(\nabla s)}g\left(p,ds\right)+  \frac{(1-c^2)^2}{ds(\nabla s)}(\nabla s)^np(\nabla s) \\
     = & \ v^n - ( \nabla s )^n \frac{1-c^2}{ds(\nabla s)}g\left(p,ds\right) + \frac{1-2c^2+c^4}{ds( \nabla s)}(\nabla s )^n p(\nabla s)\\
     = & \ v^n - ( \nabla s )^n \frac{1-c^2}{ds(\nabla s)}p(\nabla s) + \frac{1-2c^2+c^4}{ds (\nabla s)}(\nabla s )^np(\nabla s) \\ 
     = & \ v^n + \frac{c^4-1}{ds (\nabla s)}(\nabla s )^np(\nabla s)
\end{align*}
Relating this back to our original equation, we find that 
\begin{align*}
    \Delta^* i_X \mu = & \ \mathcal{J}^X_\mu(\Delta)\iota^* i_X\mu \\
    \left( v^n + \frac{c^4-1}{ds (\nabla s)}(\nabla s )^np(\nabla s) \right)dx^1\wedge dp_1 \wedge \ldots \wedge \widehat{dx^n} \wedge dp_n = &\\ 
     \ \mathcal{J}^X_\mu(\Delta) &g^{nj}p_j dx^1\wedge dp_1 \wedge \ldots \wedge \widehat{dx^n} \wedge dp_n \\
    \implies \left( v^n + \frac{c^4-1}{ds (\nabla s)}(\nabla s )^np(\nabla s) \right)= & \ \mathcal{J}^X_\mu(\Delta) v^n  \\
    \implies \mathcal{J}^X_\mu(\Delta) = &  \ \frac{ v^n + \frac{c^4-1}{ds (\nabla s)}(\nabla s )^np(\nabla s)}{v^n} \\
    = & \ 1 + \frac{c^4-1}{v^n ds (\nabla s)}(\nabla s )^np(\nabla s)
\end{align*}
As previously mentioned, $ds = dx^n$ so then $ds(\nabla s) = dx^n(\nabla s) = (\nabla s)^n$. Furthermore, $p(\nabla s) = g^{ij}p_i \partial_j s = ds(v) = dx^n(v) = v^n$. Thus, we can say that 
\begin{align*}
    \mathcal{J}^X_\mu(\Delta)     = & \ 1 + \frac{c^4-1}{v^n (\nabla s)^n}(\nabla s )^n v^n \\
    = & \ 1 + c^4 - 1 \\
    = & \ c^4
\end{align*}\end{proof}
This result is the simplest of inelastic collisions, but we believe that it is likely to carry over to nonholonmic systems with inelastic collisions so that $\mathcal{J}^X_\mu(\Delta) = c^4$ in general. However, this has yet to be shown and is reserved for future works. 
% Note: pose this as future work